\documentclass[reqno]{amsart}
\usepackage{amsmath,amsthm,amsfonts,amscd,graphicx,latexsym,amssymb,eucal,curves}
  \usepackage[all]{xy}
\usepackage{psfrag}

\usepackage{caption} 
\newlength{\halfbls}
\usepackage{enumitem}
\setitemize{leftmargin=*} 
\usepackage{tikz}
\usetikzlibrary{decorations.pathreplacing,patterns}
\usetikzlibrary{calc}

\newcommand{\changed}[1]{{#1}}

\makeatletter

\def\part{\@startsection{part}{0}%
  \z@{\linespacing\@plus\linespacing}{.5\linespacing}%
  {\normalfont\bfseries\centering}}

\renewcommand{\@bibtitlestyle}{%
  \@xp\part\@xp*\@xp{\refname}%
}

\makeatother

\newcommand{\ZZ}{\mathbb{Z}}   \newcommand{\CC}{\mathbb{C}}
\newcommand{\PP}{\mathbb{P}} \newcommand{\NN}{\mathbb{N}} \newcommand{\DD}{\mathbb{D}} \newcommand{\HH}{\mathbb{H}}
\newcommand{\GG}{\mathcal{E}}  \newcommand{\LL}{\mathbb{L}} 
    
\newcommand{\QQ}{\mathbb{Q}} \newcommand{\RR}{\mathbb{R}}  

   \newcommand{\End}{{\rm End}} 

\newcommand{\ord}{\O}   

    \def\df{\mathfrak d}  \def\fracb{\mathfrak b}
\newcommand{\fraca}{\mathfrak{a}}        \newcommand{\fracL}{\mathfrak{L}}

\newcommand{\cAA}{{\mathcal A}}
\newcommand{\cFF}{{\mathcal F}}   
\newcommand{\cKK}{{\mathcal K}} \newcommand{\cLL}{{\mathcal L}}
   \newcommand{\cMM}{{\mathcal M}}  

 \newcommand{\tr}{{\rm tr}} \newcommand{\diag}{{\rm diag}}
  \def\M{\cMM}  \def\OmM{\Omega\M}

 \def\O{\mathcal O}     

\newcommand{\SLO}{{\rm SL}(\O^\vee\oplus\O)} 
\newcommand{\SLA}{{\rm SL}(\fraca^\vee \oplus\fraca)} 
\newcommand{\SLOD}{{\rm SL}(\O_D^\vee\oplus\O_D)}
   \def\SLOA#1{{\rm SL}(#1 \oplus\O)}  

 \newcommand{\Sp}{{\rm Sp}}

\newcommand{\gr}{\nabla}   
    \newcommand{\PSL}{{\rm PSL}} \newcommand{\Jac}{{\rm Jac}}
 \newcommand{\GL}{{\rm GL}}  \newcommand{\N}{{\rm N}} \newcommand{\Teichmuller}{Teich\-m\"uller}
 
\newcommand{\ssm}{\smallsetminus}    
    
  \newcommand{\e}{{\bf e}}  \newcommand{\ve}{{\varepsilon}}
\newcommand{\ol}{\overline}  \newcommand{\ul}{\underline}

\newcommand{\Tchi}[2]{_{(#1,#2)}}
\newcommand{\TchiP}[3]{_{(#1,#2),#3}}
\newcommand{\PTD}{{\mathcal D\theta}} 
\newcommand{\PTDa}{{\mathcal D\theta_\fraca}} 
\newcommand{\PTDDD}{{\mathcal D^3\theta}}
\newcommand{\PTDDDa}{{\mathcal D^3\theta_\fraca}} 
\newcommand{\Wcov}{$W_\Pi$}

\newtheorem{Defi}{Definition}[section]  
    \newtheorem{Prop}[Defi]{Proposition}
\newtheorem{Lemma}[Defi]{Lemma}    \newtheorem{Cor}[Defi]{Corollary}
\newtheorem{Thm}[Defi]{Theorem}    

\def\={\;=\;}  \def\+{\,+\,}      \def\h{\tfrac12}  \def\d{\partial}
\def\i{^{-1}}         \def\c{\circ} 
   \def\wt#1{\widetilde{#1}}     
  \def\tJ{\wt J}  \def\tf{\wt f} 
\def\tA{\wt A} \def\tB{\wt B}  \def\tQ{\wt Q}  \def\ty{\wt y} \def\tx{\wt x} 
\def\tq{\wt q}  \def\tom{\wt\om}  
\def\tL{\wt L}   \def\hm{\hat{m}}  \def\sd{\sqrt{17}}

\newcommand\R{\mathbb R}  \newcommand\C{\mathbb C}  \newcommand\Z{\mathbb Z}  \newcommand\Q{\mathbb Q} 
\renewcommand\N{\mathbb N}  \renewcommand\P{\mathbb P}   
\renewcommand\H{\mathbb H}  \newcommand\Hm{\H^-}    

\def\a{\alpha} \def\b{\beta}  \def\th{\theta}   \def\v{\varepsilon} \def\g{\gamma}  \def\w{\delta} \def\la{\lambda} \def\l{\ell} \def\k{\kappa}
\def\om{\omega}  \def\s{\sigma}   \def\z{\zeta}  \def\p{\varphi}  \def\ph{\varphi} \def\2{\pi_2}
\def\G{\Gamma}   \def\D{\Delta}  \def\Th{\Theta}     \def\sq{\sqrt{17}} \def\sD{\sqrt D}
\def\zz{\mathbf z}  \def\vv{\mathbf v}  \def\uu{\mathbf u}  
\def\bfom{{\boldsymbol{\omega}}}

\def\Pred{standard}

\def\SL#1{{\rm SL}(2,#1)}  \def\PSL#1{{\rm PSL}(2,#1)}    \def\GL#1{{\rm GL}(2,#1)} \def\GLp#1{\text{GL}^+(2,#1)}
\def\sm#1#2#3#4{\bigl(\smallmatrix#1&#2\\#3&#4\endsmallmatrix\bigr)} 
\def\mat#1#2#3#4{\begin{pmatrix}#1&#2\\#3&#4\\ \end{pmatrix}}    \def\MT{\frac{az+b}{cz+d}}   \def\Tei{\Teichmuller}

\def\be{\begin{equation}}   \def\ee{\end{equation}}     \def\bes{\begin{equation*}}    \def\ees{\end{equation*}}
\def\ba{\be\begin{aligned}} \def\ea{\end{aligned}\ee}   \def\bas{\bes\begin{aligned}}  \def\eas{\end{aligned}\ees}

\begin{document}{\large}
\title[Theta derivatives]{Modular embeddings of Teichm\"uller curves}  
\author{Martin M\"oller and Don Zagier}
\maketitle
\tableofcontents

\newpage

\begin{abstract}
Fuchsian groups with a modular embedding have the richest
arithmetic properties among non-arithmetic Fuchsian groups.
But they are very rare, all known examples being related
either to triangle groups or to \Teichmuller\ curves.

In Part I of this paper we study the arithmetic properties of the
modular embedding and develop from scratch a theory of
twisted modular forms for Fuchsian groups with a modular embedding,
proving dimension formulas, coefficient growth estimates and
differential equations.

In Part II we provide a modular proof for an Ap\'ery-like integrality
statement for solutions of Picard-Fuchs equations. We illustrate
the theory on a worked example, giving explicit Fourier expansions
of twisted modular forms and the equation of a \Teichmuller\ curve
in a Hilbert modular surface.

In Part III we show that genus two \Teichmuller\ curves are
cut out in Hilbert modular surfaces by a product of theta derivatives.
We rederive most of the known properties of those \Teichmuller\ curves 
from this viewpoint, without using the theory of flat surfaces.
As a consequence we give the modular embeddings for all
genus two \Teichmuller\ curves and prove that the Fourier developments
of their twisted modular forms are algebraic up to one
transcendental scaling constant. Moreover, we prove that Bainbridge's
compactification of Hilbert modular surfaces is toroidal.
The strategy to compactify can be expressed using continued fractions
and resembles Hirzebruch's in form, but every detail is different.
\end{abstract}

\part*{Introduction} \label{sec:intro}

Modular forms are certainly best understood for the full modular group~$\SL\Z$, 
closely followed by those for its congruence subgroups and other arithmetic 
groups. Among the non-arithmetic Fuchsian groups, the groups having the best
arithmetic properties are those admitting a modular embedding. Here, {\em 
modular embedding} refers
to the existence of a map $\varphi: \HH \to \HH$ intertwining the
action of a Fuchsian group $\Gamma$ and its Galois conjugate.  
The notion of modular embedding (in this sense) appears for the first
time in work of Cohen and Wolfart~(\cite{CoWo90}). They show that triangle
groups admit modular embeddings, and for more than a decade these remained
the only examples. An infinite collection of new examples were found
with the discovery of new \Teichmuller\ curves by Calta~\cite{Ca04} and 
McMullen~\cite{McM03}. To find the modular embeddings for them is
one of the motivations for this paper.
\par
For a reader whose main focus is modular forms this paper wants to 
advertise an interesting new class of modular forms. For example, 
we explain an integrality phenomenon for the coefficients of a solution of a 
Picard-Fuchs differential equation, like Beukers's modular proof of the
corresponding phenomenon for Ap\'ery's famous differential equations, 
except that this time the explanation requires using a {\em pair} of ``$q$-coordinates.''  
For a reader with main focus on \Teichmuller\ curves, we show how to rediscover
many of their properties without referring to the theory of flat surfaces.
\par
The paper has three parts, linked by the aim to describe
modular embeddings. In Part~I we set up a general theory of modular forms 
for Fuchsian groups admitting a modular embedding. We call these
{\em twisted modular forms} and we prove the twisted analogs
of the properties that appear in most textbooks about modular
forms in the untwisted case. In Part~II we start from an
example of a Fuchsian group with modular embedding where
the Picard-Fuchs differential equations can be explicitly
computed. We invite the reader to discover the properties 
of Fourier coefficients of the modular embeddings and of
twisted modular forms via this worked example. In Part~III
we give explicitly the Fourier expansions of the modular
embedding for the genus two \Teichmuller\ curves found 
by Calta~\cite{Ca04} and McMullen~\cite{McM03}. In the rest of the
introduction we describe the results in more detail, highlighting
the main theorems (not necessarily in the same order) with bullet points. 
\par
\medskip
\medskip
{\bf Modular embeddings and twisted modular forms.} Suppose that the
Fuchsian group $\Gamma \subset \SL\RR$  has coefficients in a quadratic 
number field $K$ with Galois group generated by $\sigma$. The essential 
ingredient of a modular embedding for $\Gamma$ is a holomorphic function 
$\varphi: \HH \to \HH$ with the $\Gamma$-equivariance property
\bes \p\Bigl(\MT\Bigr)\=\frac{a^\s \p(z)+b^\s }{c^\s \p(z)+d^\s }\,. \ees
We show that such a modular embedding always has a ``Fourier expansion,"
and an old theorem of Carath\'eodory gives us a good estimate for its
Fourier coefficients. Analyzing $\varphi$ leads us to the definition of 
\changed{ a new kind of modular forms that we call}
{\em $\varphi$-twisted modular form} of bi-weight $(k,\ell)$. This is a 
holomorphic function $f: \HH \to \CC$ with the transformation property
$$ f\Bigl(\MT\Bigr) \ = (c\changed{z}+d)^k (c^\sigma \varphi(z) + d^\sigma)^\ell f(z)\,. $$
For example, direct calculation shows that $\varphi'(z)$ is a twisted modular form 
of bi-weight $(2,-2)$.
\par
We develop a theory of twisted modular forms from scratch, analyzing to
which extent classical topics of modular forms generalize to this new notion.
Our first topic is the coefficient growth.
\begin{itemize}
\item[$\bullet$] For $\ell > 0$ and $k+\ell >2$ the Fourier coefficients 
of a twisted modular form $f(z) = \sum_{n \geq 0} a_n q^n$ of $(k,\ell)$ satisfy the
estimate $a_n =\text{O}(n^{k+\ell -1})$.
\end{itemize}
Similar estimates are given for other bi-weights as well; see 
Theorem~\ref{thm:EstimateFC} for the complete statement. The proofs combine
the well-known Hecke argument in the untwisted case and the mechanism underlying  
the equidistribution of long horo\-cycles.
\par
The next classical topic is the dimension of the space of modular forms.
A modular embedding comes with one basic invariant $\lambda_2$, that one can view
is several ways: as an integral over a conformal density \eqref{Deflambda}, 
as a ratio of the degrees of the natural line bundles
whose sections are twisted modular forms, or as a Lyapunov exponent
for the \Teichmuller\ geodesic flow in the case of \Teichmuller\ curves.
\begin{itemize}
\item[$\bullet$] For $k+\ell$ even  and $k + \lambda_2 \ell > 2$ the dimension
of the space of twisted modular forms of bi-weight $(k,\ell)$ is
the sum of $(k+\lambda_2\ell -1)(g-1)$ and contributions from the
cusps and elliptic fixed points. Here $g$ denotes the genus of $\HH/\G$.  
\end{itemize}
For a torsion-free Fuchsian group this is of course a classical Riemann-Roch
calculation. Hence the main point is to determine the elliptic
fixed point contributions in the twisted case. See~\eqref{eq:defcharbx} for the
definition of the characteristic quantities of elliptic fixed points and
Theorem~\ref{thm:dimmodgeneral} for the complete statement.
\par
Finally, the statement that modular forms expressed in terms of a modular function
satisfy differential equations also carries over to the twisted case.
\begin{itemize}
\item[$\bullet$] If $f(z)$ is a twisted modular form of bi-weight 
$(k,\ell) \geq (0,0)$ 
and $t$ a modular function, then the function $y(t)$ defined locally 
by $y(t(z)) = f(z)$ satisfies a linear differential equation of 
order $(k+1)(\ell+1)$ with algebraic coefficients
(Theorem~\ref{thm:MFDE}).
\end{itemize}
\par
\medskip
{\bf Modular embeddings via differential equations.} The starting point of
the whole project was a worked example, the \Teichmuller\ curves for $D=17$, 
that we present in Section~\ref{sec:BouwMoeller}--\ref{sec:eqfromdiffeq}. (The definition of 
\Teichmuller\ curves along with a summary of the classification results
for \Teichmuller\ curves in genus two is given in \S\ref{sec:defTeich}
and \S\ref{sec:topoTeich}.) Starting from the flat geometry definition we 
briefly explain the derivation of the equation of the  \Teichmuller\ curve 
as family of hyperelliptic curves and computation of the Picard-Fuchs 
differential equations, following~\cite{BM07}.
\par
In this way, we present in \S\ref{sec:DEmodular} the Fourier expansion
of twisted modular forms explicitly. The corresponding group $\Gamma$ is 
\changed{in thise case} neither arithmetic nor commensurable to a 
triangle group, and the \changed{differential equations expressing the
modular forms in terms of a hauptmodule for $\G$ (= a
suitably scaled modular function $t:\H/\G \to \PP^1$
of degree 1) is not hypergeometric.}
\par
The twisted modular forms have two curious properties.
\begin{itemize}
\item[$\bullet$] The twisted modular forms do not have a power series expansion in
$K[[q]]$  for the standard modular parameter $q = e^{2\pi i z/\alpha}$, 
where $\alpha$ is the width of the cusp, but lie instead in $K[[Aq]]$, where
$A$ is a transcendental number of Gelfond-Schneider type (i.e., a number of
the form $\b_1^{\vphantom{\b_4}}\b_2^{\b_3}$ with all~$\b_i$ algebraic).
\end{itemize}
\begin{itemize}
\item[$\bullet$] If $f$ is a twisted modular form and $t$ a suitably
scaled modular function, \changed{then the
function $y(t)$ locally defined by $y(t(z))=f(z)$, with
$t(z)$ as above, has $\O_K$-integral Taylor coefficients.}
\end{itemize}
\par
The second of these observations was already proved in~\cite{BM07} using
$p$-adic differential equations. In Section~\ref{sec:eqfromdiffeq} we
will give a ``modular'' proof of both statements.  The surprising feature
here is that, while the classical proof by Beukers of the integrality of
the Ap\'ery coefficients using modularity relies on the integrality of
the Fourier coefficients of the~$q$-expansions of modular forms on arithmetic
groups, here the expansions of both~$f(z)$ and $t(z)$ with respect to~$Aq$
have coefficients with infinitely many prime factors in their denominators, and
yet the integrality of~$f$ with respect to~$t$ still holds.  To give a modular
argument for this integrality, we have to use the relationship between twisted
modular forms on the Teichm\"uller curve and Hilbert modular forms on the ambient surface.

\par
For $D=17$ the full ring of symmetric Hilbert modular forms has been determined
by Hermann~(\cite{He81}). We recall and use his construction to write down  
explicitly the equation of the \Teichmuller\ curves for $D=17$ on the
(rational) symmetric Hilbert modular surface in Theorem~\eqref{eq:TeicheqUV}.
\par
\begin{itemize}
\item[$\bullet$] \changed{There} exist coordinates $U$, $V$ on the Hilbert modular
surface $X_{17}$, explicitly given in terms of theta functions,  
such that the two Teichm\"uller curves on $X_{17}$ are cut out by the 
quadratic equations~\eqref{Eq1} and~\eqref{Eq2}.
\end{itemize}
\par
\medskip
{\bf Modular embeddings via derivatives of theta functions.} The concrete
example $D=17$ led us to the discovery of a general construction of
the modular form cutting out  \Teichmuller\ curves.
\begin{itemize}
\item[$\bullet$] The vanishing locus of the Hilbert modular form $\PTD$ of weight
$(3,9)$, given as a product of derivatives of odd theta functions, is precisely
the union of one or two \Teichmuller\ curves on the Hilbert modular surface
$X_D$ (Theorem~\ref{thm:TeichviaDTH}).
\end{itemize}
\par
Starting from the theta function viewpoint we prove the characterizing
properties of genus two \Teichmuller\ curves without relying either on the
geodesic definition or on any flat surface properties. Maybe these ideas
can be used to construct new \Teichmuller\ curves. Given the length of the
paper, we simplify our task and prove the following results only for
fundamental discriminants $D$. With appropriate care, the proofs can certainly
be adapted to the general case.
\begin{itemize}
\item[$\bullet$] The vanishing locus of $\PTD$ is transversal to one of the
two natural foliations of the Hilbert modular surface $X_D$ 
(Theorem~\ref{thm:transversal}).
\end{itemize}
\par
\begin{itemize}
\item[$\bullet$] The vanishing locus of $\PTD$ is disjoint from the reducible
locus (Theorem~\ref{thm:disjred}).
\end{itemize}
\par
On the compactified Hilbert modular surface, the reducible locus and the
vanishing locus of a Hilbert modular form always intersect and the number
of intersection points is proportional to the volume of the Hilbert modular
surface. So the claim is that all these intersection points lie on the
boundary of the Hilbert modular surface, hence at cusps of the 
vanishing locus of $\PTD$. While for the second statement we also 
have a proof using theta functions, we give proofs of both these statements 
relying on the following description of the cusps.
\par
\begin{itemize}
\item[$\bullet$] The cusps of the vanishing locus of $\PTD$ are in bijection
with pairs consisting of a \Pred\ quadratic form \changed{$ax^2+bxy+cy^2$} 
of discriminant $D$ 
and a class $r \in \ZZ/{\rm gcd}(a,c)\ZZ$ (Theorem~\ref{thm:cuspsWD}).
\end{itemize}
\par
Here an indefinite quadratic form \changed{$ax^2+bxy+cy^2$} is called 
{\em standard} if $a>0>c$ and $a+b+c<0$.  As a statement about cusps 
of \Teichmuller\ curves, this result already appears in 
\cite{McM06} and~\cite{Ba07}. Our proof, however, is completely different. 
We now explain the main idea. Suppose that a power series 
$f = \sum c_\nu q_1^{\nu} q_2^{\sigma(\nu)}$ has to vanish on a branch of a
curve parameterized by $q_1 = q^{\alpha_1} e^{\ve_1(q)}$ and $q_2 = q^{\alpha_2} e^{\ve_2(q)}$.
After these parametrizations are plugged into $f$, the lowest order
exponent (in $q$) has to appear twice, since otherwise the terms cannot cancel. 
In the concrete case of $f = \PTD$ we are led to the following notion.
Given a fractional $\O_D$ ideal $\fraca$, we say that a primitive element
$\alpha \in (\fraca^2)^\vee$ is a {\em multiminimizer} for $\fraca$ if 
the quadratic form  $x \mapsto \tr(\alpha x^2)$ on $\fraca$ takes its minimum 
on one of the three non-trivial cosets $\tfrac 12 \fraca/\fraca$ at least (and then, 
as we show, exactly) twice (with $x$ and $-x$ not distinguished). We show on 
the one hand that multiminimizers for $\fraca$ are in bijection with \Pred\ 
quadratic forms in the wide ideal class. (\changed{See e.g.~\cite{Z81} for the
correspondence between ideal classes and quadratic forms.}) On the
other hand, on any branch of the vanishing locus the local parameter can
be chosen such that $\alpha=\alpha_1=\sigma(\alpha_2)$ is a multiminimizer
and that the multiminimizers (up to multiplication by the
square of a unit) determine the branch uniquely up to an element of 
$\ZZ/{\rm gcd}(a,c)\ZZ$.
\par
We have given the definition of multiminimizers and the description of branches 
of the vanishing locus in detail since this notion and construction reappears
twice in the rest of the paper. First, multiminimizers appear prominently in
the discussion of Bainbridge's compactification below and, secondly,
this description of the branches immediately 
gives the Fourier expansion of the modular embedding of the uniformizing
group of the genus two \Teichmuller\ curves for any $D$ (see Theorem~\ref{thm:arithphi}).
Moreover, both ``curious properties'' mentioned in the case $D=17$ hold in general.
In particular, we have:
\par
\begin{itemize}
\item[$\bullet$] For any $D$, any cusp of the vanishing locus of $\PTD$
with corresponding Fuchsian group $\Gamma$ and modular embedding $\varphi$, 
the $\varphi$-twisted modular forms of bi-weight $(k,\ell)$ have a basis
with Fourier expansions of the form $\sum_{n \geq 0} a_n (Aq)^n$ with $a_n$
algebraic and $A$ transcendental of Gelfond-Schneider type 
(Theorem~\ref{thm:FCtwistedmodgeneral}).
\end{itemize}
\par
As another application of the description of \Teichmuller\ curves via theta derivatives,
we give in Theorem~\ref{QuadDiff} a description of the quadratic differentials on the leaves
of the natural foliation of a Hilbert modular surface whose integral measures the
flat distance between the two zeros of the eigenform for real multiplication.
These quadratic differentials can be packaged together to a meromorphic modular form of weight~$(-2,4)$
that we give as the quotient of theta series and their derivatives. Our result
has been used by McMullen~\cite{McM12} to describe the beautiful and complicated
flat structure on the leaves visually (``snow falling on cedars").
\par
\smallskip
{\bf Hirzebruch's compactification and  Bainbridge's compactification.}
Hir\-zebruch constructed a minimal smooth compactification of Hilbert modular
surfaces. His compactification is remarkable in many ways. First, it is
the prototype of what is nowadays a called a toroidal compactification, i.e.,
it is given by a {\em fan} of decreasing slopes, periodic under the action of
the squares of units. The fan is given for each cusp $\fraca$ of the 
Hilbert modular surface by the corners of the lower convex hull of 
$(\fraca^2)^\vee$ in $\RR^2_+$.
\par
Bainbridge (\cite{Ba07}) observed that the closure in the Deligne-Mumford 
compactification of the Torelli preimage of Hilbert modular surfaces provides another
compactification. This compactification is in general neither smooth nor
minimal, but it was useful in his calculation of Euler characteristics
of \Teichmuller\ curves.  It is amusing to compare the two types of 
compactifications  and to discover that they are parallel in spirit, 
but different in every concrete detail.
\begin{itemize}
\item[$\bullet$] Bainbridge's compactification is a toroidal compactification, 
given for each cusp $\fraca$ by the fan of multiminimizers (lying in $(\fraca^2)^\vee$)
for $\fraca$. (Theorem~\ref{thm:bainviamultimin}).
\end{itemize}
\par
The second remarkable property of Hirzebruch's compactification is that
it can easily be computed using a continued fraction algorithm.
\begin{itemize}
\item[$\bullet$] Hirzebruch's compactification is driven by the ``fast minus"
continued fraction algorithm, while Bainbridge's compactification is driven by
a ``slow plus'' continued fraction algorithm.
\end{itemize}
The reader will find the precise description of the algorithms in \S\ref{sub:HirzComp}
and \S\ref{sec:compMMM} respectively. The bijection between \Pred\ and reduced 
indefinite quadratic forms induces a subtle relationship between the number of boundary
components of Hirzebruch's 
and Bainbridge's compactification. In
particular, the number of curves in the Bainbridge compactification of any cusp is always 
the same as the number for the Hirzebruch compactification of some cusp, but not necessarily 
the same one! The \changed{definitions and} details, and several examples,
are given in Section~\ref{sec:HirzBain}.
\par
\medskip
\medskip
{\bf Acknowledgements.} 
The first named author is partially supported
by the ERC starting grant~257137 ``Flat surfaces.'' He would also like to 
thank the Max Planck
Institute for Mathematics in Bonn, where much of this work was done. 
\newpage

\part*{Part I: Modular embeddings and twisted modular forms} 

The notion modular embedding in the sense used here appears for the first
time in a paper by Cohen and Wolfart~(\cite{CoWo90}). They study holomorphic maps 
$\HH \to \HH$ equivariant with respect to a Fuchsian group
and its Galois conjugate. 
\par
In particular Cohen and Wolfart show that all triangle groups
admit modular embeddings. Subsequent work of Schmutz-Schaller and Wolfart~(\cite{ScWo00})
gave some necessary conditions for a group to admit a modular embeddding.
Some Fuchsian quadrangle groups were shown in~\cite{Ri02} not to admit modular embeddings, 
but it took more than a decade until new examples of modular embeddings were discovered.
\par
The first new examples arose from the \Teichmuller\ curves discovered by Calta
and McMullen (see \cite{Ca04} and~\cite{McM03}, and~\cite{Mo04} for the modular viewpoint). 
All \Teichmuller\ curves give rise to modular embeddings. We summarize the 
known results of
\Teichmuller\ curves (and thus the known groups admitting a modular embedding)
briefly at the end of Section~\ref{sec:hme}.
\par
Here, in Part~I, we think of the group Fuchsian group $\Gamma$ as given (e.g.\ in
terms of a presentation) and study properties of the modular embeddings as 
holomorphic maps. We \changed{define} an extension of the notion of modular 
forms to this context that we call twisted modular forms. The aim of the
first part is to study this new object and to derive the analogues of the standard results on modular forms 
(Fourier coefficients, dimension, differential equations) from scratch
in the context of twisted modular forms.
\par

\section{Hilbert modular embeddings} \label{sec:hme}

The term {\em modular embedding} is used in the literature both for equivariant maps 
from $\HH \to \HH^g$ (starting with~\cite{CoWo90}) and from $\HH^g \to \HH_g$ (already
in~\cite{Ha66}) . To distinguish,  we call them ``Hilbert modular 
embeddings'' and ``Siegel modular embeddings,'' respectively, according to the range 
of the corresponding map. We will be interested mostly in the quadratic case $g=2$ 
and refer to~\cite{ScWo00} for basic notions is the general case.
\par 

Throughout this paper we denote by $K$ a real quadratic field, with a fixed embedding $K\subset\RR$. We
use the letter~$\s$ to denote the Galois conjugation of~$K$ or the second embeddding of $K$ into $\RR$,
writing  $\s(x)$ or $x^\s$ interchangeably for $x\in K$.  By a {\it Hilbert
modular group} for $K$ we will mean any subgroup $\G_K$ of $\SL K$ commensurable with $\SL\O$ for some
order $\O\subset K$. (Later we will make specific choices.)  Such a group acts discretely and cofinitely on 
$\H^2$ by $(z_1,z_2)\mapsto\Bigl(\dfrac{az_1+b}{cz_1+d}\,,\;\dfrac{a^\s z_2+b^\s }
{c^\s z_2 +d^\s }\Bigr)$. Here $\H$ denotes the upper half-plane. 

We will be interested only in {\em Hilbert modular embeddings} for which the first coordinate in $\H^2$ is 
a local coordinate everywhere. A modular embedding of this type is described by the data $(\G,\ph)$, where
\begin{itemize}
\item $\G$ is a subgroup of some Hilbert modular group $\G_K\subset\SL K$ which, viewed as a subgroup 
of $\SL\RR$, is Fuchsian (i.e., discrete and cofinite).
\item $\p:\H\to \H$ is a holomorphic map satisfying $\p\c\g=\g^\s \c\p$ for all $\g\in\G$.
\end{itemize}
\noindent For such a pair $(\G,\p)$, the map $z\mapsto(z,\p(z))$ defines a map from the curve $\H/\G$
to the Hilbert modular surface $\H^2/\G_K\,$.

Written out explicitly, the condition on $\ph$ means that we have
\be \label{eq:phiequivar} \p\Bigl(\MT\Bigr)\=\frac{a^\s \p(z)+b^\s }{c^\s \p(z)+d^\s } \ee 
for all $\sm abcd\in\G$ and $z\in\H$.  We will use this transformation property in~\S\ref{sec:twistedMF} to
define a 1-cocycle on~$\G$ and hence a new type of modular form (``twisted'' form) on~$\G$.  Just as for usual
modular forms, these must satisfy suitable growth conditions at the cusps of~$\G$, and to formulate these we need to know 
how~$\p$ behaves near the cusps.  Assume first that one of these cusps is at~$\infty$, with the stabilizer of~$\infty$ 
in the image $\bar\G$ of~$\G$ in~$\PSL\R$ being generated by $\pm\sm1\a01$ with $\a\in K\cap\R_+\,$. Then we have:

\begin{Prop} \label{prop:phiFourier} Suppose that the stabilizer of~$\infty$ in~$\G$ is generated by $z\mapsto z+\a$ 
with $a\in K$, $\a>0$. Then $\a^\s$ is also positive and $\p(z)$ has an expansion of the form
  \be\label{eq:phiatinfinity} \p(z)\=\dfrac{\a^\s }\a z\+\sum\limits_{n=0}^\infty B_n\,\e\bigl(\dfrac{nz}\a\bigr)
      \qquad\bigl(\,\forall z\in\HH\,\bigr)\,, \ee
where $\e(x):=e^{2\pi ix}$ and where the coefficients $B_n$ satisfy the inequalities
   \be\label{Carat}  |B_n| \; \le \; 2\,\Im(B_0)\qquad\text{for all $n\ge1$.}  \ee
\end{Prop}
\begin{proof} From~\eqref{eq:phiequivar} we have $\p(z+\a)=\p(z)+\a^\s$, so the function $\,\p(z)-{\a^\s}z/\a\,$ 
is invariant under $z\mapsto z+\a$ and hence equals $f(\e(z/\a))$ for some holomorphic function~$f(q)$ in the punctured 
disc~$\DD^*=\{q\,:\,0<|q|<1\}$. Define a second holomorphic function~$F$ in~$\DD^*$ by $F(q)=\e\bigl(f(q)/|\a^\s|\bigr)$.
From $\p(\HH)\subseteq\HH$ we deduce that $|q^{\pm1}F(q)|<1$ in~$\DD^*$, where the sign is chosen so that $\pm\a^\s>0$. 
It follows that $F(q)$ extends to a meromorphic function in $\DD=\{q\,:\,|q|<1\}$
with at most a simple pole at~$q=0$.  But then the fact that~$F$ has a single-valued
logarithm in~$\DD^*$ implies that its order of vanishing at~$0$ must be zero, so $\a^\s$ must be positive and $f$ extends
holomorphically to the full disc and hence has a convergent Taylor expansion $\sum_{n=0}^\infty B_nq^n$, proving the first
claim.  For the second, we note first that the estimate $|F(q)|\le|1/q|$ and the holomorphy of~$F$ at~$0$ imply by the
maximum principle that $|F(q)|\le1$ in the disk~$\DD$ (this is just the Schwarz lemma, applied to the function $qF(q)$),
so $B_0$ has positive imaginary part and the function $f(q)/B_0$ takes on values in the right half-plane.  An elementary
argument then gives the estimate $|B_n|\le2n\Im(B_0)$. (Write $f(q)=(B_0-\bar B_0\la(q))/(1-\la(q))$ where $\la$ sends
$\DD$ to~$\DD$ and 0 to~0; then $B_n=2i\Im(B_0)\sum_{m=1}^n[\la^m]_n$, where $[\la^m]_n$ denotes the coefficient of~$q^n$
in~$\la(q)^m$, which is bounded in absolute value by~1 because $\la$ is.) The stronger estimate $|B_n|\le 2\Im(B_0)$ follows 
from a theorem of Carath\'eodory~\cite{Car}, which says precisely that a holomorphic function mapping~$\DD$ to the right
half-plane and sending 0 to~1 has all its Taylor coefficients at~0 bounded by~2 in absolute value.
\end{proof}
\begin{Cor} \label{Cor:estimateY} The imaginary part of~$\p(z)$ satisfies the inequalities
\be \label{estimateY}  \dfrac{\a^\s }\a y \;\le\; \Im(\p(z))  \;\le\;\dfrac{\a^\s }\a y \+ \changed{C\bigl(1 \+\frac 1y\bigr)} \ee
for all $z=x+iy\in\H$, where $C$ is a constant independent of~$z$. 
\end{Cor}
\begin{proof} The first statement is just the inequality $|F(q)|\le1$ established in the course of the above proof, and \changed{because $\sum_{n=0}^\infty|q|^n\ll 1 + 1/y$ the second follows from~\eqref{Carat} .} \end{proof}

Exactly similar statements hold for all of the other cusps of~$\G$.  Recall that by definition, 
such a cusp is an element $\nu\in\PP^1(K)$ whose stabilizer $\bar\G_\nu$ in $\bar\G$ is infinite 
cyclic, say $\bar\G_\nu=\langle\pm\g_\nu\rangle$.  Choose $g\in\SL K$ with $g(\infty)=\nu$. Then $g\i\g_\nu g=\pm\sm1\a01$
for some positive element $\a$ of $K$. Equation~\eqref{eq:phiequivar} implies that the function $\p_g={g^\s }\i\c\phi\c g$ 
satisfies $\p_g(z+\a)=\p_g(z)+\a^\s$, because $\p_g\c\sm1\a01=\p_gg\i\g_\nu g
={g^\s }\i\p\,\g_\nu g={g^\s }\i\g_\nu^\s \p\,g={g^\s }\i\g_\nu^\s g^\s \p_g=(g\i\g_\nu g)^\s \p_g=\sm1{\a^\s }01\c\p_g\,$.
Then the same proof as for the case $g=\text{Id}\,$ shows that $\a^\s$ is positive and that
  \be\label{eq:phiabsexpansion} \p_g(z)=\dfrac{\a^\s }\a z\+\sum\limits_{n=0}^\infty B_n\,\e\bigl(\dfrac{nz}\a\bigr) \ee
for some constants $B_n\in\C$, satisfying the same estimate~\eqref{Carat} as before, and for all \hbox{$z\in\HH$}. Of course 
$\a$ and $B_n$ depend on~$\nu$, and also on the choice of~$g$, but the expansions for $\G$-equivalent cusps are the same up 
to trivial rescalings, because $\p_{\g g}=\p_g$ for $\g\in\G$, so that there are only finitely many essentially distinct expansions. 

Another basic property of~$\p$, obtained by applying the Schwarz lemma to the map $I_{\p(z')}\circ\p\circ I_{z'}^{-1}:\DD\to\DD$,
where $I_a$ for $a\in\H$ denotes the standard isomorphism $(\H,a)\to(\DD,0)$ given by $z\mapsto(z-a)/(z-\bar a)$, is that one has the inequalities
 \be \label{contract} \biggl|\frac{\p(z)-\p(z')}{\p(z)-\overline{\p(z')}}\biggr| \;\le\;\biggl|\frac{z-z'}{z-\overline{z'}}\biggr|\,,
  \qquad \frac{|\p(z)-\p(z')|^2}{\Im(\p(z))\,\Im(\p(z'))} \;\le\;\frac{|z-z'|^2}{\Im(z)\,\Im(z')} \ee
or equivalently $d(\p(z),\p(z'))\le d(z,z')$ for any $z,\,z'\in\H$, where $d:\H\times\H\to\R_{\ge0}$ denotes the Poincar\'e metric.
(This is of course a standard property of any holomorphic map from the upper half-plane to itself.)  Fixing~$z$ and letting $\Im(z')$ tend
to infinity, or fixing $z'$ and letting $\Im(z)$ tend to~0, we obtain second proofs of the two inequalities in~\eqref{estimateY}, while
letting $z'$ tend to $z$ we obtain the estimate
 \be\label{estKappa} |\p'(z)| \;\le\; \frac{\Im(\p(z))}{\Im(z)}\qquad(z\in\H)\,. \ee
or equivalently $|\k(z)|\le1$, where $\k:\H\to\C$ is the  map defined by 
 \be\label{defKappa} \k(z) \= \frac{\Im(z)}{\Im(\p(z))}\,\p'(z)\;. \ee
From the equivariance property~\eqref{eq:phiequivar} we obtain the formulas
 \be\label{eq:invariance} \Im(\p(\g z)) \,=\, \frac{\Im(\p(z))}{|c^\s \p(z)+d^\s|^2}\,, \quad
    \p'\bigl(\MT\bigr) \,=\, \frac{(cz+d)^2}{(c^\s \p(z)+d^\s )^2}\,\p'(z)\,, \ee
and these together with the standard formula $\Im(\g z)=\Im(z)/|cz+d|^2$ imply that the function~$\k$ is $\G$-invariant.
We can therefore introduce a basic invariant $\la_2\= \la_2(\G,\p) $ 
of the pair $(\G,\p)$ by
 \be\label{Deflambda}  \la_2 \= \frac1{\text{vol}(\G\backslash\H)}\,\iint_{\G\backslash\H}\,\frac{|\p'(z)|^2}{\Im(\p(z))^2}\,dx\,dy 
  \=  \frac{\iint_\H |\k|^2\,d\mu}{\iint_\H \,d\mu} \;\,\in\;\, (0,1]\,,\ee
where $d\mu=y^{-2}dx\,dy$ (with $z=x+iy$ as usual) is the standard $\SL\R$-invariant measure on~$\H$, and where the integral can be taken over 
any fundamental domain for the action of~$\G$ on~$\H$ and is absolutely convergent because~$|\kappa|\le1$.
The invariant~$\la_2$, whose values are always rational numbers, 
can be interpreted  either as a ratio of the intersection numbers 
on the Hilbert modular surface \changed{with the line bundles~$\cLL$ 
and~$\wt\cLL$ defined in~\S\ref{sec:dimvan},} or as the second 
Lyapunov exponent~(\cite{Mo11}), which explains the notation. 
\changed{It is a commensurability invariant of the Fuchsian group $\Gamma$}. 

We close this section by describing briefly the four known types of Hilbert modular embeddings.  The first two will
play a role in this paper and will be discussed in more detail in Section~\ref{sec:MCTC}. The other two are
mentioned only for the sake of completeness.  
 
{\it Type 1: Modular curves.}  Here $\G=\G_A=\{\g\in\G_K\mid A\g=\g^\s A\}$ where $A$ is a suitable element of $\GLp K$,
and the map $\p$ is defined by $\p(z)=Az$.  The corresponding curves $(1,\p)(\H/\G)\subset \H^2/\G_K$ in this 
case are the irreducible components of the curves $T_N$ studied in~\cite{HZ76} and~\cite{HZ77} and reviewed
in~\S\ref{sec:redloc}.  In particular, there are infinitely many curves of this type on each Hilbert modular surface, 
and conversely each of these curves lies on infinitely many Hilbert modular surfaces in $\HH_2/\Sp(4,\ZZ)$.

{\it Type 2: Teichm\"uller curves.} These are defined abstractly as the algebraic curves in the moduli space
$\M_g$ of curves of genus~$g$ that are totally geodesic submanifolds for the \Tei\ metric.  In genus~2, they
always lie on Hilbert modular surfaces and (apart from one exception for the field $\Q(\sqrt5)$) have a
modular interpretation as the components of the moduli space of genus~2 curves whose Jacobian has real multiplication by
an order in a real quadratic field and such that the unique (up to a scalar) holomorphic form on the curve 
that is equivariant with respect to this action\footnote{i.e. $\,(\mu\circ)^*\p=\mu\p$ for all $\mu$ in the order;
cf.~\S\ref{sec:HMGHMS} (``first eigendifferential") for details}  has a double zero. There 
are at most two curves of this type on each Hilbert modular surface~$X_K$, and conversely each
\Teichmuller\ curve lies on exactly one~$X_K$. The proof that these
curves have a modular embedding comes from~\cite{Mo04}. (See also the proof of Proposition~\ref{Prop5.6}.)

There exists a variant of these curves, not used in this paper but studied in detail by Wei\ss\ in~\cite{We12}, called
 ``twisted \Tei\ curves,'' obtained as the images of \Tei\ curves under the action of elements of $\GL K^+$. They are still
geodesic for the Kobayashi metric in $X_K$, but no longer for the Kobayashi (= \Tei) metric in $\M_2$.  There are in
general infinitely many of these curves on any Hilbert modular surface.

{\it Types 3 and 4: Curves related to Prym varieties.}  Recall that a {\it Prym variety} is the kernel  
of the map $\Jac(C)\to\Jac(C_0)$ induced by a double cover $C\to C_0$ of curves.
By the Riemann-Hurwitz formula, it is 2-dimensional if and only if the genus~$g$ of~$C$
lies between~2 and~5. For the cases~$g=3$ or~4, there is a construction of
\Tei\ curves in the moduli space~$\M_g$ corresponding to certain Prym varieties having real 
multiplication by an order in a real quadratic field~\cite{McM06b}.  
Our cases~3.\ and~4.\ are these two cases, in the order~$g=4$, $g=3$.  

The four types 1.--4.~are distinguished by the invariant~$\la_2$, which takes 
on the values $1$ for Type~1, $\frac13$ for Type~2, $\frac15$ for Type~3, 
and~$\frac17$ for Type~4 (\cite{Ba07}, \changed{Theorem~15.1} 
and~\cite{Mo11}, Proposition~5.1). 
For each of these values, there is an infinite
number of commensurability classes of Fuchsian groups with this invariant.
\par 
The exceptional Teichm\"uller curve over $\Q(\sqrt5)$ corresponds to 
$\la_2 = 1/2$, and at the time of writing this is the only known 
commensurability class of a Fuchsian group with this value of the 
invariant $\la_2$.
\par

\section{Twisted modular forms}  \label{sec:twistedMF}

For any function $\p:\H\to\H$ satisfying~\eqref{eq:phiequivar} we can define two factors $J(g,z)$ 
and $\tJ(g,z)$ for $g\in\SL K$ and $z\in\H$ by
\be J(g,z)\=cz+d\,,\qquad \tJ(g,z)\=c^\s \p(z)+d^\s \qquad \text{if $g=\sm abcd\,$.} \ee
The calculation
  \bas \tJ(\g_1\g_2,z)&\=(c_1a_2+d_1c_2)^\s \,\p(z)\+(c_1b_2+d_1d_2)^\s \\
   &\=\Bigl(c_1^\s \,\frac{a_2^\s \p(z)+b_2^\s }{c_2^\s \p(z)+d_2^\s }\+d_1^\s \Bigr)\,\bigl(c_2^\s \p(z)+d_2^\s \bigr) \\
   &\=\Bigl(c_1^\s \,\p\bigl(\frac{a_2z+b_2}{c_2z+d_2}\bigr)\+d_1^\s \Bigr)\,\bigl(c_2^\s \p(z)+d_2^\s \bigr) 
   \qquad\quad\text{(by eq.~\eqref{eq:phiequivar})}\\
   &\= \tJ(\g_1,\g_2z)\,\tJ(\g_2,z) \qquad\text{for $\g_1,\,\g_2\in\G\;$ }  \eas
shows that $\tJ$ is a cocycle for $\G$.  (The corresponding statement for $J$, which follows from the same
calculation with $\p\equiv\text{Id}$, is, of course, standard.) It follows that the map $f\mapsto f|_{(k,\l)}g$
of the space of holomorphic functions in~$\H$ to itself defined for $k,\,\l\in\Z$ and $g\in\GLp K$ by
\be \bigl(f|_{(k,\l)}g\bigr)(z)\=J(g,z)^{-k}\,\tJ(g,z)^{-\l}\,f(gz) \ee
is a group action when restricted to~$\G$. 

 We now define a {\em $\p$-twisted modular form of bi-weight $(k,\l)$}
on~$\G$ to be a holomorphic function $f:\H\to\C$ satisfying $f|_{(k,\l)}\g=f$ for all $\g\in\G$
together with the growth requirement that the function  $f_g$ is bounded as $\Im(z)\to\infty$ for every
$g\in\SL K$, where $f_g(z) = (cz+d)^{-k}(c^\s\p_g(z)+d^\s)^{-\l}f(gz)$.  The function $f_g$ depends only 
on the coset $\G g$. 
\par
\changed{To describe its Fourier expansion, we need to distinguish cases. If 
$-1 \in \G$, then the modular transformation property implies that $k+\l$ is even 
if $M_{k,\l}(\G)$ is to be non-zero. If $-1  \not\in \G$, then there may exist
non-zero twisted modular forms for both parities of $k+\l$. Let $\a$ as before be
the totally positive element of $K$ with $\bigl(g\i\G g\bigr)_\infty=\bigl\langle\pm\sm1\a01\rangle$. Recall that the cusp 
$g\infty$ is called {\it irregular} if $\G$ does not contain~$-1$ and 
if $\bigl(g\i\G g\bigr)_\infty=\bigl\langle-\sm1\a01\rangle$.  Now there are 
two cases. If the cusp is regular or if $k+l$ is even, then  $f_g(z+\a)= f_g(z)$ 
and the Fourier expansion of~$f_g$ has the form 
\be \label{eq:FE}
f_g(z) \ = \sum_{n\ge0} a_n\,\e(nz/\a) \quad \text{as} \quad \Im(z)\to\infty\,, \ee
where $n$ ranges over non-negative integers. Only if the cusp is irregular and $k+\l$ is odd, 
then   $f_g(z+\a)= - f_g(z)$ and the Fourier expansion is as in~\eqref{eq:FE}, 
but now with $n$ ranging over $\Z_{\ge0}+\h$ instead of $\Z_{\ge0}$.}
%
\par
If the Fourier coefficient $a_0$ is 0 for all cusps (a condition that is automatically satisfied at regular cusps if $k+\l$ is odd),
we call $f$ a {\em cusp form}.   The spaces of $\p$-twisted modular forms and cusp forms of bi-weight $(k,\l)$ 
will be denoted by $M_{k,\l}(\G,\p)$ and $S_{k,\l}(\G,\p)$, respectively. We will often omit the ``$\p$'' when no 
confusion can result.

Obviously, ordinary modular forms of weight~$k$ on~$\G$ are $\p$-twisted modular forms of bi-weight $(k,0)$ for any~$\p$,
and in fact $M_{k,0}(\G,\p)=M_k(\G)$.  We give three examples of twisted modular forms with $\l\ne0$. 

\smallskip
i) We always have $\p'\in M_{2,-2}(\G,\p)$, by virtue of the second equation in~\eqref{eq:invariance} and the expansions of $\p$ at 
the cusps given in Section~\ref{sec:hme}. This example shows that the weights $k$ and $\l$ of a holomorphic twisted modular form do not both have to be positive. 
A similar example, which follows from the calculations given at the
end of Section~\ref{sec:twistedDE}, is that $2\p'\p''' - 3\p''^2$ 
belongs to $M_{8,-4}(\G,\p)$.
Note the quotient of this by $\p'^2$ is the Schwarzian derivative of $\p$.
\par
\smallskip
ii) In the case of modular curves, when the map $\p:\H\to\H$ is given by a fractional linear transformation $A\in\GLp K$, the calculation
 $$ \tJ(\g,z)=c^\s Az+d^\s =J(\g^\s ,Az)=\frac{J(\g^\s A,z)}{J(A,z)}=\frac{J(A\g,z)}{J(A,z)}=\frac{J(A,\g z)}{J(A,z)}\,J(\g,z)$$
 for $\g\in\G$ shows that if $f$ belongs to $M_{(k,\l)}(\G,\p)$, then the function $f_A(z)=J(A,z)^{-\l}f(z)$
 belongs to $M_{k+\l}(\G)$ in the usual sense, so that here we do not get a new kind of modular forms.
For \Tei\ curves, on the other hand, the automorphy factor $\tJ$ cannot be reduced to an automorphy factor 
of the classical sort, and the twisted forms are a genuinely new type of modular form.

\smallskip
iii) If $\G_K$ is a Hilbert modular group containing~$\G$ and $F$ is a Hilbert modular form of weight $(k,\l)$ on $\G_K$
(i.e., $F:\H^2\to\C$ is a holomorphic map satisfying $F(\g z_1,\g^\s z_2)=(cz_1+d)^k(c^\s z_2+d^\s )^\l F(z_1,z_2)$ 
for all $\g\in\G_K$ and all $z_1,\,z_2\in\H$), then the restriction of $F$ to $\H$ under the embedding 
$(1,\p):\H\to\H^2$ is an element of $M_{k,\l}(\G,\p)$. 

\smallskip
The last example provides many twisted modular forms for any $(\G,\p)$. But not all twisted modular forms arise in 
this way, and it makes sense to study the twisted forms independently of the two-variable theory.  In particular,
one can ask for the dimensions of the spaces $M_{k,\l}(\G,\p)$ and $S_{k,\l}(\G,\p)$ and for the structure of the
bigraded ring $M_{**}(\G,\p)=\oplus_{k,\l}M_{k,\l}(\G,\p)$,  just as is usually done for classical modular
forms when $\l=0$, and we can also study  the classical topic of growth of Fourier coefficients.  We shall give
a general formula for the dimensions in the next section and a description of the ring of twisted forms in a special
example in Section~\ref{sec:arithmeticW17}, while the rest of this section is devoted to the study of the coefficient~growth.

\par
\bigskip
\begin{Thm} \label{thm:EstimateFC}
Let $f(z) = \sum a_n q^n$ be a twisted modular form of bi-weight $(k,\ell)$, and set $K=k+|\ell|$.
Then the Fourier coefficients of~$f$ satisfy the estimates
$$ a_n \= \begin{cases} \text{\rm O}(n^{K/2}) & \text{if $k+\ell<2$ or $f$ is a cusp form}, \\ 
    \text{\rm O}(n^{K/2}\log n) & \text{if $k+\ell=2$}, \\  
     \text{\rm O}(n^{k-1+\max(0,\ell)}) & \text{if $k+\ell>2$}\,. \\ \end{cases} $$
\end{Thm}
\par
\begin{proof} Suppose first that $f$ is cuspidal. Here we use a modification of
the well-known argument given by Hecke in the untwisted case. We construct the 
real-valued continuous function 
$$F(z) = |f(z)|\, y^{k/2} \,\, \ty^{\ell/2},$$
where $y = y(z) = \Im(z)$ and $\ty = \ty(z) = \Im (\p(z))$. This function
is $\Gamma$-invariant by the definition of a twisted modular form. 
Since $f$ is a cusp form, $F$ decays rapidly at cusps and hence is bounded 
\changed{on the whole upper half plane},
so $f(z) = \text{O}(y^{-k/2} \ty^{-\l/2})$.
On the other hand, $a_n = \frac1\alpha\,\int_{0}^{\alpha} f(x+iy) \e(-n(x+iy)/\alpha)\,dx$ 
for any $y$. 
Specializing to $y = 1/n$ and using the first or the second inequality in \eqref{estimateY} depending on 
the sign of $\ell$, we obtain the estimate stated.
\par
In the remaining cases, still 
$$ |a_n| \,\leq\, \frac1\alpha  \int_0^{\alpha} F(x+i/n) \,n^{k/2}\, \Im (\p(x+1/n))^{-\ell/2} dx 
\,\ll \, n^{K/2} \int_0^{\alpha} F(x+i/n) \, dx, $$
by~\eqref{estimateY} (where the constant implied by $\ll$ depends only on~$\Gamma$), 
but now  $F(z) = F(\gamma(z)) = \text{\rm O}(H(z)^{(k+\ell)/2})$ instead of $\text{\rm O}(1)$, 
where $H(z)$ is defined in the lemma below.
Since the exponent $k-1+\max(0,\ell)$ is equal to $K/2 + (k+\ell)/2 - 1$, 
the remaining statement is precisely the content of the \changed{following} lemma.
\end{proof}
\par
\begin{Lemma} Let $\Gamma$ be a non-cocompact Fuchsian group, with the width of the cusp 
at $\infty$ equal to~$1$, and define the height function
$H(z) = \max_{\gamma \in \Gamma} \Im(\gamma z)$. Then for $n>1$ and $\la>0$ one has the estimates
$$\int_{0}^1 H(x + \tfrac in)^\lambda \,dx \= \begin{cases} \text{\rm O}(1) & \text{if 
$0<\lambda < 1$}, \\ 
    \text{\rm O}(\log n) & \text{if $\lambda=1$}, \\  
     \text{\rm O}(n^{\lambda -1} ) & \text{if $\lambda>1$}\,, \\ \end{cases} $$
where the implied constant does not depend on $n$.
\end{Lemma}
\par
\begin{proof}
The case $\lambda = \tfrac12$ is \cite{St04}, Proposition~2.2. Essentially the same method
can be used to give all cases. We provide the details only for $\lambda > \h$.
We let $T = \sm 1101$ and choose for each of the $h$ cusps $\eta_j$ 
of $\Gamma$ a matrix $N_j \in \SL\RR$  such that $N_j \eta_j = \infty$, 
and such that the stabilizer of $\infty$  in $\G_j = N_j \G N_j^{-1}$
is always $\langle T \rangle$. 
Let $\cFF$ be a closed fundamental domain 
for $\Gamma$, which we may choose 
so that the cusp neighborhoods have the shape
$$ N_j(\cFF) \cap \{\z \in \HH: \Im(z) > B\} = [0,1] \times [B,\infty) \qquad (j=1,\ldots,h)$$
for some $B>1$, and are disjoint.  
We define the truncation function  $\lfloor x \rfloor_B$ to be
$x$ if $x > B$ and $0$ otherwise. Since the complement in the fundamental domain 
of the cusp neighborhoods is compact, it suffices to prove the statement of the lemma
with $\lfloor H \rfloor_B$ in the place of $H$. Note that 
\be \label{eq:cuspcontr}
\int_{0}^1 \lfloor H(x + \tfrac in) \rfloor_{\!B} ^{\;\lambda} \,dx 
\= \sum_{j=1}^h \sum_{\gamma \in \langle T \rangle \backslash  \Gamma_j}  
\int_0^1 \lfloor \Im(\gamma(x + \tfrac in))\rfloor_{\!B}^{\;\lambda} \, dx\,.
\ee
Suppose that $\gamma = \sm abcd$ gives a non-zero contribution to the right hand side.
Recall that Shimizu's Lemma states that in a Fuchsian group, normalized so that the cusp~$\infty$ 
has width $1$, either $c=0$ or $1 \leq |c|$. Here the truncation implies that $c \neq 0$ and 
that 
$$1 \leq \frac{1/n}{| (c(x+\tfrac in) +d|^2 }\, ,$$
from which 
$$ 1 \leq |c| \leq \sqrt{n} \quad \text{and}\quad  -\frac dc \in [-1,2].$$
On the other hand, with the substitution $x = - \tfrac dc+ \tfrac tn$ we get
\bas
\int_0^1 \lfloor \Im(\gamma(x + \tfrac in))\rfloor_{\!B}^{\;\lambda} \, dx
& \,\leq\, \int_{-\infty}^\infty \left( \frac{1/n}{(cx + d)^2 + c^2/n^2} \right)^{\lambda}\, dx \\
& \,\leq\, \frac{n^{\lambda-1}}{|c|^{2\lambda}} \int_{-\infty}^\infty \frac{dt}{(t^2+1)^\lambda} 
 \= \text{\rm O}\Bigl( \frac{n^{\lambda-1}}{|c|^{2\lambda}}\Bigr)\,.
\eas
(It is this estimate which has to be changed, taking into account $\lfloor \;\cdot\;\rfloor_B$, 
when $\lambda  \leq \tfrac 12$.) We define
$$ C^j_{\mu,\nu}(x,X) \= \bigl\{\gamma = \sm abcd \in \langle T \rangle \backslash 
\Gamma_j \,:\, x < |c| \leq X, \,\,
-\frac dc \in [\mu,\nu] \bigr\}$$
and $C^j_{\mu,\nu}(X) =  C^j_{\mu,\nu}(0,X)$. 
The crucial observation now is that  the cardinality
of this set is bounded for $X>1$, any $j$ and $\mu < \nu$ by
$$ \# C^j_{\mu,\nu}(X) \leq (\nu-\mu) X^2 +1.$$
This again follows from Shimizu's Lemma, which implies that for two matrices 
$\gamma = \sm abcd$ and $\gamma' = \sm {a'}{b'}{c'}{d'}$ in $C^j_{a,b}(X)$ we have
$|\tfrac{d'}{c'} - \tfrac dc| \geq |cc'|^{-1} \geq X^{-2}$ (see \cite{Iw95}, Proposition~2.8, for details).
\par
If $\lambda =1$, then the contribution of the $j$-th cusp to the right hand side
of~\eqref{eq:cuspcontr} is bounded above by a constant times 
\bas
\sum_{\gamma \in C^j_{-1,2}(\sqrt{n})} \frac{1}{|c|^2} &\,\leq\,  \sum_{k=0}^{\tfrac12 \log_2 n}  
\sum_{\gamma \in C^j_{-1,2}(2^{k-1},2^k)} \frac{1}{|c|^2} \\ &\,\leq\, \sum_{k=0}^{\tfrac12 \log_2 n} 
 \frac{3\cdot 2^{2k}}{2^{2k-2}} \= \text{O}(\log n).
\eas
The other cases with $\lambda > 1/2$ are calculated the same way, the estimate for the left-hand side 
of~\eqref{eq:cuspcontr} now being $\,\text O\bigl(n^{\la-1}\sum_{k\le\frac12\log_2n}2^{2k(1-\la)}\bigr)$, 
which is $\text O(n^{\la-1})$ for $\la>1$ and~$\text O(1)$ for~$\la<1$.
%
\end{proof}
\par
\changed{The bounds in the preceding theorem are not sharp for the classical case
$\ell = 0$, since the methods of Rankin and Selberg give an improvement
for any Fuchsian group. Note, however, that the deeper results of
Deligne cannot be applied here, since on Teichm\"uller curves
one does not dispose of Hecke operators.
\par
In the more interesting case of strictly twisted modular forms (i.e.\ 
$\ell \neq 0$) we know in the example of $\varphi'$ that the coefficients grow 
like $O(n)$ by Caratheodory's estimate~\eqref{Carat}, while the preceding theorem
with $k=2$, $\ell=-2$ gives only the bound $O(n^2)$. However, we do not know
if the behaviour of the coefficients of $\varphi'$ is typical for a
twisted modular form of bi-weight $(2,-2)$ or if there exist other 
elements whose coefficient growth is closer to $O(n^2)$.  
}

\section{Dimensions and degrees}
\label{sec:dimvan}

Twisted modular forms of bi-weight $(k,\l)$ can be thought of as sections of the bundle $\cLL^{\otimes k}\otimes\wt\cLL^{\otimes\l}$
with appropriate growth conditions at the cusps, where $\cLL$ and $\wt\cLL$ are the line bundles over $\H/\G$ defined as the 
quotients of $\H\times\C$ by the equivalence relations $(z,u)\sim(\g z,J(\g,z)u)$ and $(z,u)\sim(\g z,\wt J(\g,z)u)$ for $\g\in\G$.
The dimension of the space of such forms for a given group can therefore be computed by the Riemann-Roch theorem for curves,
just as in the case of classical modular forms, if we know the degrees of the two bundles $\cLL$ and~$\wt\cLL$ and the numbers
of cusps and elliptic fixed points of various orders of the group.  In particular, if $\G$ had no cusps and no fixed points 
(a situation of which, so far as we know, there is no example), then 
the Riemann-Roch theorem would give 
$\dim M_{k,\ell}(\Gamma) = (k\,\text{deg}(\cLL)+\ell\,\deg(\wt\cLL) \changed{-1})\,(g-1)$, where~$g$ is the genus of~$\HH/\G$. The
presence of cusps (including possibly irregular ones) and elliptic fixed points will make the actual formula more
complicated.

Let $\Pi$ be a torsion-free subgroup of finite index in $\Gamma$. Such a group always exists, since the level three 
subgroup of a Hilbert modular group is torsion-free and $\Gamma$ is a subgroup of the Hilbert modular group.
We define $\Pi_0$ to be a subgroup of finite index such that the eigenvalues of all parabolic elements are
one (i.e.\ all cusps are regular). Such a subgroup exists, since $\Pi$ is free if it is has a cusp. By passing to 
a smaller subgroup if necessary, we may suppose $\Pi_0 \subset \G$ to be normal.   We let $\cLL_0$ and $\wt\cLL_0$ be 
the line bundles over $\H/\Pi_0$ defined by the automorphy factors $J$ and $\wt J$ respectively. 

The basic invariant we attach to a Hilbert modular embedding is the ratio
\be
\lambda_2 \= \deg(\wt\cLL_0) \, / \, \deg(\cLL_0).
\ee
As a consequence of the proof of Theorem~\ref{thm:dimmodgeneral} below, 
we see that this number does not depend on the choice of $\Pi_0$ among 
torsion-free subgroups of~$\G$ with regular cusps.  The value of the 
invariant $\lambda_2$ 
is given by $1$, $1/3$, $1/5$, $1/7$ on the four classes of 
Hilbert modular embeddings described in Section~\ref{sec:hme}.

The definition of~$\lambda_2$ and the classical result that ordinary modular forms of weight~2 are differential forms 
on $\HH/\Pi_0$ imply the following Proposition.
\begin{Prop} \label{prop:dimmodPi}
Let $g_0$ and $s_0$ denote the genus and the number of cusps of $\HH/\Pi_0$, respectively. Then 
\be \label{eq:degonPi9}
 \deg(\cLL_0) = g_0 - 1 + s_0/2  \quad \text{and} \quad \deg(\wt\cLL_0) = \lambda_2 \,(g_0 - 1 + s_0/2)\,. \ee
\end{Prop} 

Now we have to pass from~$\Pi_0$ to $\Gamma$, which may have both elliptic fixed points and irregular cusps. 
If $\Gamma$ contains $-I$, then we will assume that $k+\l$ is even, since otherwise the equation $f|_{(k,\l)}(-I)=-f$
implies that the space of twisted modular forms of bi-weight~$(k,\l)$ is~0.  
\par
We define {\em characteristic numbers} at elliptic fixed points and cusps for the bundle
of twisted modular forms of bi-weight $(k,\ell)$ in the following way.
Suppose that $x$ is an elliptic fixed point and that the isotropy group $\G_x$ is of order $n_x$ in $\SL\RR$.
We take a generator $\gamma=\sm**c* \in \G_x$ that acts on the tangent space at $x$ by a rotation by $2\pi/n_x$
in the positive direction, i.e.\ such that $\arccos(\tr(\gamma)/2) = 2\pi/n_x$ and $c\,\sin(2\pi/n_x)\le 0$. We let $r_x = \tfrac1{n_x}$. \changed{Since $\gamma$ is of finite
order, so is $\gamma^\sigma$} and \changed{we can} 
define $r^\sigma_x \in \tfrac1{n_x}\ZZ$ by 
$$\cos(2\pi r^\sigma_x) \= \tr(\gamma^\sigma)/2 \quad \text{\and} \quad c^\sigma \sin(2\pi r^\sigma_x)\,\le\,0\,.$$
Then the characteristic number at $x$ is defined as
 \be \label{eq:defcharbx}
b_x(k,\ell) = \bigl\{-k r_x  - \ell r^\sigma_x\bigr\}\,,  \ee
where the curly braces denote the fractional part (in $[0,1)$) of the 
rational number. If $x$ is a cusp, we define the characteristic number $b_x(k,\l)$ to be~1/2 if the cusp 
is irregular, $-I\not\in\G$, and $k+\l$ is odd, and we let  $b_x(k,\l)=0$ in all other cases.
\par
We remark that characteristic numbers are a finer information than
the usual {\em type} of the elliptic fixed point (\cite{vG87}, Section~I.5 and~V.7), 
since there are two possibilities even for fixed points of order two (in $\PSL\RR$).
We let $\Delta = \left(\begin{smallmatrix} 1&0\\0&{\sqrt{D}} \end{smallmatrix}\right)$ and
 $S = \left(\begin{smallmatrix} 0&{-1}\\1&0\end{smallmatrix}\right) \in \SL{\ord_D}$.
Then the contribution of $\Delta S \Delta^{-1} \in \SLOD$ is $b_x(k,\ell) \=  
\bigl\{ \frac{-k + \ell}{4} \bigr\}$, whereas the contribution of $S$
is  $b_x(k,\ell) \=  \bigl\{ \frac{-k - \ell}{4} \bigr\}$. Note that in
all these calculation we consider modular embeddings to $\HH^2$. If we
consider a modular embedding to $\HH \times \HH^-$, the fixed point
of $S$ is $-i$ in the second factor, so its contribution is
$b_x(k,\ell) \=  \bigl\{ \frac{-k + \ell}{4} \bigr\}$.

We can now give the dimension of the space of modular forms in terms
of the topology of $\HH/\G$ and those characteristic numbers. We let 
$\ol{n_x}$ be the order of the isotropy group  $\G_x$ in $\PSL\RR$.

\begin{Thm} \label{thm:dimmodgeneral} 
Let $k$ and $\l$ be integers.  If $k+\l$ is odd and $\G$ contains~$-I$, then $\dim M_{k,\ell}(\Gamma,\varphi)=0$.
If $-I\not\in\G$ or $k+\l$ is even, then twisted modular forms of
bi-weight $(k,\ell)$ are precisely the global sections of a line 
bundle $\cLL_{k,\ell}$ of degree
$$\deg (\cLL_{k,\ell}) \=  (k+\lambda_2\ell)\,\Bigl(g -1 + \frac s2 
  +\frac12 \sum_{x \in \HH/\G} \bigl(1- \frac1{\ol{n_x}}\bigr) \Bigr) \,-\, \sum_{x \in \ol{\HH/\G}} b_x(k,\ell)\,, $$
where $g$ and $s$ denote the genus and the number of cusps of~$\G$. If also $k + \lambda_2 \ell \geq 2$, then
$$ \dim M_{k,\ell}(\Gamma,\varphi) \= \deg (\cLL_{k,\ell}) \+ (1-g)\,.$$ 
\end{Thm}
\par
\begin{proof}
We mimic the standard argument for ordinary modular forms and describe $\cLL_{k,\ell}$ as a subsheaf of 
$\cLL^{\otimes k}\otimes\wt\cLL^{\otimes\l}$. If $t$ is a local parameter at $x$, for $x$ both in $\HH/\Gamma$ 
or being a cusp, the stalk $(\cLL_{k,\ell})_x$ at consists of all germs of holomorphic functions $f$ with 
$f(\gamma \, t) =  J(\gamma, t)^k \wt J(\gamma,\varphi(t))^\ell f(t)$ for all $\gamma \in \G_x$ in the stabilizer of $x$. 
With this definition, twisted modular forms of bi-weight $(k,\ell)$ \changed{
for the modular embedding~$\varphi$} are obviously the global sections of $\cLL_{k,\ell}$.
\par
In order to compute the degree of $\cLL_{k,\ell}$ we use the map $\pi: \HH/\Pi_0 \to \HH/\G$. This induces an inclusion 
$\pi^* (\cLL_{k,\ell}) \to \cLL_0^k\cLL_0^\ell$. Since we know the degree of the image in terms of $g$, $s$ and $\lambda_2$, it suffices 
to compute the degree of its cokernel $\cKK$. This cokernel is supported at the elliptic fixed points and at the cusps.
\par
Suppose first that $z_\gamma \in \HH$ is the fixed point of $\gamma \in \G$. 
The point $\varphi(z_\gamma)$ is fixed by $\gamma^\sigma$, so that by the cocycle
condition both $J(\gamma,z_\gamma)$ and  $\wt J(\gamma,\varphi(z_\gamma))$ are roots
of unity of some  order that divides the order of the isotropy group $n_x$. 
\par
Now let $y$ be one of the preimages of $x$ and let $u$ be a local parameter at $y$, so that $t = u^{\ol{n_x}}$. Then
$$(\cLL_0^k\wt\cLL_0^\ell)_y \cong \CC[[u]] \quad \text{and} \quad (\pi^*(\cLL_{k,\ell}))_y = u^{\ol{n_x}B_x(k,\ell)}\,\CC[[u]],$$
where  $B_x(k,\ell) \in [0,1)$ is the unique rational number such that 
$$J(\gamma,z_\gamma)^{k}\wt J(\gamma,\varphi(z_\gamma))^\ell  \= \e({\,B_x(k,\ell))}\,.$$
Consequently, $\dim \cKK_y = \ol{n_x} \, B_x(k,\ell)$ for each of the $\deg(\pi)/\ol{n_x}$ points $y$ above $x$.
\par
Next we want to show that $B_x(k,\ell) = b_x(k,\ell)$. If $x = i$, then
the generator of $\G_x$ specified above is $\gamma =  \left(\begin{smallmatrix} \cos(2\pi/n_x) & \sin(2\pi/n_x)\\
-\sin(2\pi/n_x)&\cos(2\pi/n_x)\end{smallmatrix}\right)$ and hence 
$J(\gamma_1,i)^k = \e(-kr_x)$ and $\wt J(\gamma_1,i)^\ell = \e(- \ell r_x^\sigma)$ by definition of $r_x$ and $r_x^\sigma$. This proves the claim in the special case  $x = i$. For the general case note first that for any $\alpha \in \SL\RR$ the cocycle property implies 
$J(\gamma,z_\gamma) = J(\alpha \gamma \alpha^{-1}, z_{\alpha \gamma \alpha^{-1}})$, where
$z_{\alpha \gamma \alpha^{-1}} = \alpha z_\gamma$ is the fixed point of $\alpha \gamma \alpha^{-1}$. The equivariance property~\eqref{eq:phiequivar} implies that $\varphi(z_\gamma)$ is the fixed point of $\gamma^\sigma$ and hence 
$\wt J(\gamma, z_\gamma) = J(\gamma^\sigma,z_{\gamma^\sigma})$. If $\alpha$ takes the fixed point of $\gamma$ to $i$, then $\alpha^\sigma$
takes the fixed point of $\gamma^\sigma$ to $i$ and so 
\begin{equation} \label{eq:Jcong}
\begin{aligned}
& J(\gamma,z_\gamma)^k\, \wt J(\gamma,\varphi(z_\gamma))^\ell = 
J(\gamma,z_\gamma)^k\,J(\gamma^\sigma,z_{\gamma^\sigma})^\ell 
\\ & = 
J(\alpha \gamma \alpha, i)^k\, J((\alpha^\sigma \gamma^\sigma \alpha^\sigma)^{-1},i)^\ell = J(\alpha \gamma \alpha, i)^k \, \wt J((\alpha \gamma \alpha)^\sigma,i)^\ell
\end{aligned}
\end{equation} 
reduces to the case already considered.
\par
Now suppose that $x$ is a cusp and let $y$ be one of the cusps of $\Pi_0$ above $x$. If $\ol{n_x}$ denotes
the degree of the covering $\pi$ at $y$, then there are $\deg(\pi)/\ol{n_x}$ cusps above $x$ since $\pi$ is Galois. 
\par
We start with the case $-I \not\in \G$. Then the stabilizer $\Gamma_x$ is infinite cyclic. Let $\gamma$ be a generator.
The same argument as for~\eqref{eq:Jcong} allows us to assume that the fixed point $z_\gamma = \infty$. 
Note that $\gamma^\sigma$ also fixes $\infty$, so that $J(\gamma,\infty) = \wt J(\gamma, \varphi(\infty)) = \gamma_{2,2}$,
the lower right entry of $\gamma$. Since the cusp is irregular if and only if the generator $\gamma$ has $\gamma_{2,2} = -1$, we deduce 
$$J(\gamma,\infty)^k \, \wt J(\gamma, \varphi(\infty))^\ell = \e(b_x(k,\ell))$$
for $b_x(k,\ell)$ defined above. On the other hand, let $a$ be the width of the cusp $\infty$ of $\Gamma_0$, 
so that $q = \e(z/a)$ is a local parameter at $y$.  Then 
$$(\cLL_0^k\wt\cLL_0^\ell)_y \cong \CC[[q]] \quad \text{and} \quad 
(\pi^*(\cLL_{k,\ell})_y = q^{\ol{n_x}b_x(k,\ell)}\,\CC[[q]],$$ 
so that in total $\dim \cKK_y = \ol{n_x}\, b_x(k,\ell)$. 
\par
With the same local calculation one checks that if $-I \in \G$ always
$\pi^*(\cLL_{k,\ell})_y  = (\cLL_0^k\wt\cLL_0^\ell)_y$. Hence in 
this case, too, $\dim \cKK_y = \ol{n_x}\, b_x(k,\ell) = 0$ holds by definition.
\par
Altogether, this implies 
$$ \deg(\cLL_{k,\ell}) \= \frac{1}{\deg(\pi)} \Bigl((k+\lambda_2\ell) \deg(\cLL_0)\Bigr)
- \sum_{x \in \ol{\HH/\G}}  b_x(k,\ell) .$$
The number of cusps of $\G$ is
$ s_0 = \deg(\pi) \, \sum_{x \in \partial(\HH/\G)} \frac1{\ol{n_x}}.$
Together with \eqref{eq:degonPi9} and the Riemann-Hurwitz formula
$$ \frac{g(\HH/\Pi_0)-1}{\deg(\pi)} = g(\HH/\G) -1 + \frac12 \sum_{x \in \ol{ \HH/\G}}  
\Bigl(1-\frac1{\ol{n_x}}\Bigr)$$ 
this implies the degree claim. The dimension statement then follows from Riemann-Roch
\changed{since the $H^1$-term vanishes for $\deg(\cLL_{k,\ell})>2g-2$, which is guaranteed
by the hypotheses on $k+\lambda_2\ell$ and the fact that 
$b_x(k,\ell) \leq 1-\frac{1}{\ol{n_x}}$ for all cusps and elliptic fixed points.} 
\end{proof}

\section{Differential equations coming from twisted modular forms}  \label{sec:twistedDE}

A basic fact about modular forms, whose proof will be recalled below, is that for any Fuchsian group $\G\subset\SL\R$, 
any modular function~$t$ on~$\G$ and any modular form $f$ of integral weight $k\ge1$ on~$\G$, the function $y(t)$ defined locally by
$f(z)=y(t(z))$ satisfies a linear differential equation of order $k+1$ with algebraic coefficients (and even
with polynomial coefficients if $\H/\G$ has genus~0 and $t$ is a hauptmodule\footnote{Recall that a ``hauptmodule"
(or ``Hauptmodul'' if one retains the German spelling) is a modular function $t$ giving an isomorphism between
$\overline{\H/\G}$ and $\P^1(\C)$ if the former has genus~0.}). In this subsection we prove the
corresponding statement for twisted modular forms.  This statement will give one of the two approaches used in
this paper to describe \Tei\ curves explicitly on Hilbert modular surfaces, by comparing the differential equations
coming from their geometric definition (Picard-Fuchs differential equations) with the differential equations
satisfied by suitable twisted modular forms on them.

\begin{Thm} \label{thm:MFDE} Let $f(z)$ be a twisted modular form on $(\G,\p)$ of bi-weight $(k,\l)$, with $k,\,\l\ge0$, 
and $t(z)$ a modular function with respect to the same group~$\G$. Then the function $y(t)$ defined locally by $f(z)=y(t(z))$
satisfies a linear differential equation of order $(k+1)(\l+1)$ with algebraic coefficients. \end{Thm}
\begin{proof} It suffices to prove this for the two cases $(k,\l)=(1,0)$ and $(0,1)$, since the general
case follows from these. (The number $(k+1)(\l+1)$ arises as the dimension of $\rm{Sym}^k(V_1)\otimes\rm{Sym}^k(V_2)$
where $\dim V_1=\dim V_2=2$.) The first case is the classical theorem mentioned above, of which several proofs are
known (see e.g.~\S5.3 of~\cite{123}). We reproduce one of them here since it generalizes directly to the more complicated
case of bi-weight $(0,1)$.

Let, then, $f(z)$ be an ordinary modular form of weight~1 and $t(z)$ a modular function on $\G$.  By definition we
have the two transformation equations
$$ t\Bigl(\frac{az+b}{cz+d}\Bigr) \= t(z)\,,\;\quad  f\Bigl(\frac{az+b}{cz+d}\Bigr)\=(cz+d)\,f(z)$$
for all matrices $\sm abcd\in\G$. Differentiating these equations gives the further transformation equations
\bas   t'\Bigl(\frac{az+b}{cz+d}\Bigr)&\=(cz+d)^2t'(z)\,,\\
  f'\Bigl(\frac{az+b}{cz+d}\Bigr)&\=(cz+d)^3f'(z)\+c(cz+d)^2\,f(z)\,,\\
  f''\Bigl(\frac{az+b}{cz+d}\Bigr)&\=(cz+d)^5f''(z)\+4c(cz+d)^4f'(z)\+2c^2(cz+d)^3f(z)\,.   \eas
The first of these equations says that $t'$ is a (meromorphic) modular form of weight~2,
and by combining the others we find that the expression $2{f'}^2-ff''$ is a modular form of weight~6.  It follows that
  \be \label{eq:tprime}
 \frac{t'(z)}{f(z)^2}\=A\bigl(t(z)\bigr),\qquad\frac{2f'(z)^2-f(z)f''(z)}{t'(z)f(z)^4}\=B\bigl(t(z)\bigr) \ee
for some rational (or, if $t$ is not a hauptmodule, algebraic) functions $A(t)$ and $B(t)$. A direct calculation shows that
  $$ \frac1{t'}\biggl(\frac{t'}{f^2}\,\frac1{t'}\,f'\biggr)'\+\frac{2{f'}^2-ff''}{t'f^4}\,f\=0\,,$$
It follows that the function $y(t)$ defined parametrically by the equation $y(t(z))=f(z)$ (which of course can only
hold locally, since $t(z)$ is $\G$-invariant and $f(z)$ isn't) satisfies the second order linear differential equations
  \be\label{diffeq10} \bigl(A(t)\,y'(t)\bigr)'\+B(t)\,y(t)\=0\,, \ee
or $Ay''+A'y'+By=0$. This proves the theorem in the case $(k,\l)=(1,0)$.  

Now suppose that $f$ is a twisted modular form of bi-weight (0,1), i.e., $f$ satisfies the transformation
equation $f\bigl(\frac{az+b}{cz+d}\bigr)=(c^\s \p(z)+d^\s )\,f(z)$ for $\sm abcd\in\G$.  From this equation,
and from equation~\eqref{eq:phiequivar} and its derivative (= second equation in~\eqref{eq:invariance}),
we find by further differentiating the transformation equations
\bas  \p''\Bigl(\frac{az+b}{cz+d}\Bigr)&\=\frac{(cz+d)^4}{(c^\s \p(z)+d^\s )^2}\p''(z)
     \+\frac{2c(cz+d)^3}{(c^\s \p(z)+d^\s )^2}\,\p'(z) \\
      &\quad\,-\,\frac{2c^\s (cz+d)^4}{(c^\s \p(z)+d^\s )^3}\,\p'(z)^2 \,, \\
   f'\Bigl(\frac{az+b}{cz+d}\Bigr)&\=(cz+d)^2(c^\s \p(z)+d^\s )\,f'(z)\+c^\s (cz+d)^2\p'(z)f(z)\,,\\
  f''\Bigl(\frac{az+b}{cz+d}\Bigr)&\=(cz+d)^4(c^\s \p(z)+d^\s )\,f''(z)\\
   &\quad\+\bigl[2c(cz+d)^3(c^\s \p(z)+d^\s )\+2c^\s (cz+d)^4\p'(z)\bigr]\,f'(z)\\
 &\quad \+\bigl[c^\s (cz+d)^4\p''(z)\+2cc^\s (cz+d)^3\p'(z)\bigr]\,f(z)\,.  
\eas
From these equations it follows that the combination $(2{f'}^2-ff'')\p'+ff'\p''$ is a modular form of weight~6.  But 
we have already seen that $t'$ and $\p'$ are twisted modular of bi-weights $(2,0)$ and $(2,-2)$, respectively.  It follows 
that 
  \ba \label{eq:tprimephi}
\frac{t'(z)}{\p'(z)f(z)^2} & \=A\bigl(t(z)\bigr),\quad\; \\
 \frac{(2f'(z)^2-f(z)f''(z))\,\p'(z)+f(z)f'(z)\p''(z)}{t'(z)\p'(z)^2f(z)^4} & \=B\bigl(t(z)\bigr)\, \ea
for some algebraic (resp.~rational if $t$ is a hauptmodule) functions $A(t)$ and $B(t)$,
and since by direct calculation we have
 \bes \frac1{t'}\Bigl(\frac{t'}{\p'f^2}\,\frac1{t'}\,f'\Bigr)'\+\frac{(2{f'}^2-ff'')\p'-ff'\p''} {t'{\p'}^2f^4}\,f\=0 
 \ees
in this case, we deduce that $f$ satisfies a  second order linear differential equation of the same form~\eqref{diffeq10} as before.
\end{proof}

{\bf Remark.}  The two weight 6 modular forms $2{f'}^2-ff''$ (for $f\in M_{1,0}(\G)$) and $(2{f'}^2-ff'')\p'+ff'\p''$ 
(for $f\in M_{0,1}(\G)$) used above, which are easily checked to be holomorphic at the cusps, are special
cases of the classical Rankin-Cohen bracket and of twisted versions of it, respectively.  Without going into details, we mention
that the twisted Rankin-Cohen brackets of two twisted modular forms $f_i\in M_{k_i,\l_i}(\G,\p)\;(i=1,\,2)$ can be 
defined as the product of the usual Rankin-Cohen brackets of ${\p'}^{\l_1/2}f_1$ and ${\p'}^{\l_2/2}f_2$ (which by example~i) 
of Section~\ref{sec:twistedMF} are ordinary modular forms of weight $k_1+\l_1$ and $k_2+\l_2$ on~$\G$) with a suitable power of~$\p'$.
\par
\bigskip
\bigskip
\bigskip
\newpage
\part*{Part II: Modular embeddings via differential equations}

In Section~\ref{sec:twistedDE} we have seen abstractly how classical or twisted modular forms 
give rise to differential equations. In Part II we show conversely, in a 
specific example, how to obtain from these differential equations the Hilbert modular embedding~$\varphi$. The example 
that we will consider in detail is  $D=17$, for which the differential equations needed were computed in~\cite{BM07}.  
In Section~\ref{sec:BouwMoeller}  we will sketch how these were obtained, 
referring to that paper for the full details. In Section~\ref{sec:arithmeticW17} 
we discuss the arithmetical properties of the solutions of these differential
equations and compute the Fourier expansions of the corresponding modular forms
at all cusps. We turn Theorem~\ref{thm:dimmodgeneral} into a concrete description of the ring of
modular forms (Theorem~\ref{thm:dimmodWD} for the general result and Propositions~\ref{Prop71} 
and~\ref{prop:dimGamma} for the special case $D=17$),
since the corresponding local invariants can be computed for \Teichmuller\ curves.
In Section~\ref{sec:eqfromdiffeq} we will show how to use these solutions to obtain 
an explicit embedding of the \Tei\ curve in the Hilbert modular surface. 
The introductory Section~\ref{sec:MCTC} provides the necessary background on
Hilbert modular surfaces and \Teichmuller\ curves in genus~$2$.

\section{Curves on Hilbert modular surfaces} \label{sec:MCTC}

As we have already said, there are two basic examples of the situation described 
in Section~\ref{sec:hme}: modular curves and \Tei\ curves.  In this section we describe 
both of these, the first relatively briefly since it is well known and the second in more detail. 
We begin with a preliminary subsection specifying more precisely the Hilbert modular surfaces 
that will be used in this paper. The main new result in this section is the dimension
formula Theorem~\ref{thm:dimmodWD}.

\subsection{Hilbert modular groups and Hilbert modular surfaces} \label{sec:HMGHMS}

As before, we denote by $K$ be a real quadratic field, together with a fixed embedding $K\subset\RR$, and
denote by~$\s$ both the Galois conjugation and the second embedding of~$K$ into~$\RR$.  In  \S\ref{sec:hme}
we briefly defined Hilbert modular groups and Hilbert modular surfaces, denoting them generically by $\G_K$ and
$\H^2/\G_K$.  Now we want to be more specific.  Our general reference are Hirzebruch's seminal paper~\cite{Hi73} and
the book \cite{vG87} by van der Geer.

Usually when one speaks of ``the'' Hilbert modular group for~$K$ one means the group $\SL{\O_D}$,
where $D$ is the discriminant
of an order $\O=\O_D \subset K$. However, since we want principally polarized abelian surfaces, 
we need to work instead with the modified  Hilbert modular group
$$\SLOA{\O^\vee} \= \mat \O{\O^\vee}{({\O^\vee})\i}\O \;\cap\;\SL K\,,$$
where  $\O^\vee$ denotes the set of $x\in K$
for which $xy$ has integral trace for all $y\in\O$. One has $\O^\vee=\df\i$, where in the case of quadratic fields
the ideal $\df$, called the {\it different} of~$K$, is simply the principal ideal $(\sD)=\sD\,\O$. Note that the two groups
$\SL\O$ and $\SLOA{\O^\vee}$ are conjugate in $\GL K$ by the action of the diagonal matrix $\Delta=\sm100\sD$, and in
particular are isomorphic as abstract groups.  But the action of $\D$ interchanges the upper and lower half-planes
in the second factor (since the Hilbert modular group acts on the second factor via its Galois conjugates and the
Galois conjugate of the determinant of $\Delta$ is negative), so the quotient is the
Hilbert modular surface
$$X_D \ = \H^2/\SLOA{\O^\vee},$$ 
which is isomorphic to $X_\O^-=\H\times\Hm/\SL\O$ and not in general isomorphic to
 the ``standard'' Hilbert modular surface $X_\O=\H^2/\SL\O$. (They do not even necessarily have the same Euler characteristic.)  
If $\O$ contains a unit $\ve$ of negative norm, which happens, for instance, when $D$ is prime, then $\df$ is
principal in the narrow sense and the varieties $X_\O$ and $X_\O^-$  are isomorphic via $(z_1,z_2)\mapsto(\ve z_1,\,\ve^{\changed{\sigma}} z_2)$.

To a point $\zz=(z_1,\,z_2)\in\H^2$ we associate the polarized abelian surface $A_{\zz}=\C^2/\fracL_\zz$,
where $\fracL_\zz\subset\C^2$ is the lattice 
  \be \label{eq:latticeemb} \fracL_\zz\=\bigl\{(az_1\+b,\, a^\s z_2+b^\s)\mid a\in\O^\vee,\;\,b\in\O)\bigr\}\,, \ee
with the action of $\O$ on $A_\zz$ induced from the action $\la(v_1,v_2) = (\la v_1,\la^\sigma v_2)$ of $\O$ on~$\C^2$
and with the polarization induced from the antisymmetric pairing
 \be\label{eq:tracepairing} \langle\,(a,b),\,(a',b')\,\rangle =
 \tr_{K/\Q}\bigl(ab'-a'b\bigr) \qquad(a,\,a'\in (\O^\vee)^{-1},\;\,b,\,b'\in\O)\,. \ee
This pairing is unimodular and the polarization is principal, which is why that case is of special interest.

We observe that the action of $\O$ on $A_\zz$ gives a canonical splitting of the 2-dimensional space of holomorphic 
1-forms on~$A$ into two 1-dimensional eigenspaces, generated by the differential forms $\om=dv_1$ and $\tom=dv_2$, which we 
will call the {\it first} and {\it second eigendifferential}, respectively.  If $A_\zz$ is the Jacobian of a curve~$C$ 
of genus~2, then by the canonical identification of the spaces of holomorphic 1-forms on~$C$ and on~$A_\zz$ we obtain 
corresponding eigendifferentials on~$C$.  These will be used in the definition of \Tei\ curves in \S\ref{sec:defTeich}.  

Since the isomorphism class of $A_\zz$ depends only on the image of $\zz$ in $X_D$, and since
polarized abelian surfaces are parametrized by points in the quotient of the Siegel upper half-space $\H_2$ by
$\Sp(4,\Z)$  we get an embedding of the Hilbert modular surface $X_D$ into the quotient $\H_2/\Sp(4,\Z)$,
a {\it Siegel modular embedding}. Explicitly, let $\psi$ be the map from $\H^2$ to $\H_2$ given by 
\be \label{eq:SME} \psi\,: \quad \zz\,=\,(z_1,z_2)\;\mapsto\; B\mat{z_1}00{z_2} B^T \ee 
where, for some $\Z$-basis  $(\om_1,\om_2)$ of $\O$ we let
  \be \label{eq:SMEA}  B = \mat{\om_1}{\om_1^\s}{\om_2 }{\om_2^\s } 
\; \quad \text{and} \quad A=B^{-1}\, .\ee 
We define a homomorphism $\Psi:\SLOA{\O^\vee} \to\Sp(4,\Z)$ by
\be \label{eq:SMEG}   \Psi\,: \quad  \mat abcd\mapsto\mat {B^T}00A\mat{\widehat a}{\widehat b}
{\widehat c}{\widehat d}\mat {A^T}00B\,,\ee
where $\widehat a$ for $a\in K$ denotes the diagonal matrix $\,\diag(a,a^\sigma)$. 
Then the map $\psi$ is equivariant with respect to the actions of $\SLOA{\O^\vee}$ on 
$\HH^2$ and of $\Psi(\SLOA{\O^\vee})$ on $\HH_2$, so it induces a map, also denoted by~$\psi$, on the level of quotient spaces.  
\par
{\bf Remark on notation.} We will use the letters $\psi$ and $\Psi$ for Siegel modular embeddings and 
$\p$ and $\Phi$ for Hilbert modular embeddings. The capital letter will denote the map on the level of
modular groups and the small letter the map on the level of symmetric spaces or quotient spaces.
\par
\smallskip
More generally, for any invertible $\O$-ideal $\fraca$ the trace pairing~\eqref{eq:tracepairing} on the 
lattice $\fraca^\vee \oplus \fraca$ is unimodular and consequently, the abelian surface 
$\fracL_{\zz,\fraca}$, defined as in~\eqref{eq:latticeemb} with $\O$ replaced by $\fraca$, is principally
polarized. This implies that the Hilbert modular surfaces
$$ X_{D,\fraca} \= \HH/\SLA,$$ 
where
\be \label{eq:defSLA}
\SLA \= \mat \O{\fraca^\vee(\fraca)^{\i}}
{\fraca ({\fraca^\vee})\i}\O \;\cap\;\SL K\,,
\ee
also parametrize principally polarized abelian varieties with real multiplication by~$\O$. The
only difference is that now \changed{the} cusp at $\infty$ \changed{of $X_{D,\fraca}$} 
is in general a different one than \changed{the cusp at $\infty$} for~$X_D$. 
We will use these variants $X_{D,\fraca}$ 
when we discuss cusps of Hilbert modular surface in Section~\ref{sec:cuspTheta0}. If we construct
$B$ using some $\ZZ$-basis $(\om_1,\om_2)$ of $\fraca$, then~\eqref{eq:SME} defines a map $\psi$ 
that is equivariant with respect to a homomorphism 
\be \label{eq:Psi_fraca}
\Psi:\SLA \to\Sp(4,\Z) \ee
given by the same definition~\eqref{eq:SMEA}. Hence the pair $(\psi,\Psi)$ defines a Siegel modular embedding of $X_{D,\fraca}$.  
\par

\subsection{Modular curves}  \label{sec:redloc}

We already defined the modular curves in Section~\ref{sec:hme} as the quotients of $\H$ by 
subgroups of $\SL\R$ of the form $\G_A=\{\g\in\G_K\mid A\g=\g^\s A\}$ where $A$ is ``a suitable element''
of $\GL K$, embedded into appropriate Hilbert modular surfaces via $z\mapsto(z,Az)$. Here ``suitable'' 
means that the adjoint of $A$ equals its Galois conjugate, so that $A=\sm{\la^\s}{-b\sD}{a\sD}\la$ 
for some $(a,b,\la)\in\Q\times\Q\times K$, which after multiplying $A$ by a suitable scalar in $\Q^\times$ 
(which does not change the definition of $\G_A$) we can assume belongs to $\Z\times\Z\times \O_K$.  The 
corresponding embedded curve in $\HH^2$ is defined as $az_1z_2+\nu^\s z_1+\nu z_2+b=0$, where $\nu=\la/\sD\in\O^\vee=\df\i$, 
and the union of these curves (or rather, of their images in the Hilbert modular surface) when $A$ ranges 
over all matrices as above with given determinant~$N\in\N$ is denoted by~$T_N$.  These curves $T_N$ were 
studied in detail (for the Hilbert modular surface $X_\O$) in~\cite{HZ76} and~\cite{HZ77}. We recall a
few results that we will use.  The curve $T_N$ is non-empty if and only if $N$ is congruent modulo~$D$ to 
the norm of an element of~$\O$, and is non-compact (i.e., passes through the cusps) if and only if $N$ 
is the norm of an integral ideal $\fraca$ of $K$, in which case each of its components is non-compact. 
It is not in general irreducible, for three reasons. First, we have $T_N=\bigcup_{d^2|N}F_{N/d^2}$, where 
$F_N$ is defined like~$T_N$ but with the additional requirement that $(a,b,\nu)$ is primitive in the lattice
$\Z\times\Z\times\df\i$.  Secondly, the $F_N$ are in general not irreducible either, but decompose as
$\bigcup_\a F_N(\a)$, where $\a$ ranges over the elements of $\df\i/\O$ with $N(\a)\equiv N/D\!\pmod1$
and $F_N(\a)$ is defined by requiring $\nu\equiv\a\pmod\O$ (\cite{Z75}, p.~4, Remark~1).  Finally, even
the $F_N(\a)$ need not be irreducible. (For instance, if $D=p$ is prime and $p^2|N$, then the two Legendre
symbols $(a/p)$ and $(b/p)$ cannot both vanish or have opposite values, so $F_N=F_N(0)$ has two components 
distinguished by the invariant $\ve\in\{\pm1\}$ defined by $(a/p)=\ve$ or $(b/p)=\ve$; cf.~\cite{Fr77}.)  
However, if $N$ is ``admissible'' in the sense of the proposition on p.~57 of~\cite{HZ77} (i.e., $N$ is the 
norm of a primitive ideal in the principal genus), then that proposition says that each $F_N(\a)$ is 
irreducible, as one shows by counting the number of cusps of each component separately and of the whole curve~$F_N$.

The same results apply to the curves in $X_\O^-$ defined (and denoted) in the same way but with the determinant 
of~$A$ now being~$-N$ rather than~$N$.  Using the identification $(z_1,z_2)\mapsto(\sD z_1,-\sD z_2)$ of $X_\O^-$ 
with $X_D$, we can consider these as curves in~$X_D$, the defining equation now being 
\be \label{TNeq}  ADz_1z_2 + \la^\s z_1 + \la z_2 + B \= 0 \qquad (A,\,B\in\Z,\;\la\in\O,\; ABD-\la\la^\s=N). \ee
A special union of these curves will be play a role a in characterization of \Teichmuller\ curves
below. In the moduli space of principally polarized abelian surfaces $\cAA_2$ we denote by $P$ the {\em product  
locus} (also called {\em reducible locus}), i.e.~the locus of abelian varieties that split, as polarized 
abelian varieties, into a product of elliptic curves. The Torelli map gives an isomorphism 
  $$ t: \M_2 \to \cAA_2 \ssm P\,.$$ 
The intersection of $P$ with the Hilbert modular surface $X_D$ will be denoted by $P_D$. It is a union of 
modular curves, as described in the following proposition.

\begin{Prop} \label{prop:redunionFN}
The decomposition into irreducible components of $P_D$ is given by
$$ P_D \= \bigcup_{\nu\in\df\i,\;\nu\gg0,\;\text{\rm Tr}(\nu)=1} P_{D,\nu}
 \= \bigcup_{r\equiv D\;(\text{\rm mod 2}),\,|r|<2\sD} P_{D,\frac{r+\sD}{2\sD}}\,, $$
where $P_{D,\nu}$ is the image in~$X_D$ of the curve $(\nu, \nu^\s)\HH \subset \HH^2\,$.
\end{Prop}
\begin{proof} This is essentially Corollary~3.5 of~\cite{McM07}, which states that
\be\label{RedLoc} P_D = \bigcup_{N,\,r\in \ZZ,\;N>0,\;D=r^2+4N} T_N\Bigl(\frac{r+\sD}{2\,\sD}\Bigr)\,. \ee
Since each $N$ occurring is admissible (it is the norm of the primitive principal ideal
generated by $\frac{r+\sD}2$), we have that $T_N(\nu)=F_N(\nu)$ is irreducible for each
$\nu=\frac{r+\sD}{2\sD}$ and hence coincides with its subcurve $P_{D,\nu}$.
\end{proof}

We would like to say a few words to explain where the equation $D=r^2+4N$ in~\eqref{RedLoc} comes from.
A point of $P_D$ corresponds to a product $E\times E'$ of elliptic curves having real multiplication by~$\O$,
i.e., for which there is an endomorphism $\Phi=\sm a{\la'}\la b$ of $E\times E'$ satisfying a quadratic
equation of discriminant~$D$ over~$\Z$. Since for generic points the elliptic curves $E$ and $E'$ do not
have complex multiplication, we have $a\in\End(E)=\Z$, $b\in\End(E')=\Z$, and $\la\la'=\la'\la=N$ with 
$N=\deg\la\in\N$ and hence $\Phi^2-(a+b)\Phi+(ab-N)=0$, $D=(a+b)^2-4(ab-N)=(a-b)^2+4N$.
We should also mention that the statement $P_D\subseteq\bigcup_{r^2+4N=D}T_N$ is just the special case $D'=1$ 
of the general statement that the intersection of two Humbert surfaces $\mathcal H_D$ and $\mathcal H_{D'}$
in the moduli space $\cAA_2=\H_2/\Sp(4,\Z)$ is contained in the union of $T_N$ with $DD'=r^2+4N$ for some
$r\in\Z$, $N\in\N$.  This statement is well known, and is given implicitly in the proof of Prop.~XI.2.8, p.~215,
of~\cite{vG87}, but since we could not find a convenient reference and since the proof is easy, we give
it here.  We recall (cf.~\cite{vG87}, Chapter~IX) that the Humbert surface $\mathcal H_D$ is defined as the 
image in $\cAA_2$ of the union of the curves
\be\label{Humbert}  \bigl\{\sm{\tau_1}{\tau_2}{\tau_2}{\tau_3}\in\H_2 \;\bigl|\;
      a\tau_1+b\tau_2+c\tau_3+d(\tau_2^2-\tau_1\tau_3)+e\,=\,0\bigr\} \ee
with $(a,b,c,d,e)\in\Z^5$, $b^2-4ac-4de=D$.  If~$D$ is a fundamental discriminant, then $\mathcal H_D$ is irreducible
and hence can be given by any one of the equations in~\eqref{Humbert}.
The locus of products of elliptic curves in~$\cAA_2$ is $\mathcal H_1$, because the standard embedding
$(\H/\rm{SL}_2(\Z))^2\to\H_2/\Sp_4(\Z)$ is given by the equation $\tau_2=0$ in~$\H_2$, which has the form~\eqref{Humbert}
with $(a,b,c,d,e)=(0,1,0,0,0)$, $b^2-4ac-4de=1$. The Hilbert modular surface~$X_D$ can be identified with~$\mathcal H_D$, because 
if we write $N(x\om_1+y\om_2)=Ax^2+Bxy+Cy^2$ ($A,\,B,\,C\in\Z$, $B^2-4AC=D$), then the map~\eqref{eq:SME} is given by
\be \label{HMSemb} \begin{pmatrix} \tau_1 & \tau_2 \\ \tau_2 &\tau_3\end{pmatrix} \= 
 \mat {\om_1^2z_1+{\om_1^\s}^2z_2}{\om_1\om_2z_1+\om_1^\s\om_2^\s z_2}{\om_1\om_2z_1+\om_1^\s\om_2^\s z_2}{\om_2^2z_1+{\om_2^\s}^2z_2}\,,\ee
which satisfies an equation of the form~\eqref{Humbert} with $(a,b,c,d,e)=(C,-B,A,0,0)$, $b^2-4ac-4de=D$.
In general, to compute the intersection $\mathcal H_D\cap\mathcal H_{D'}$ we substitute the expression in~\eqref{HMSemb} into the
an equation of the form~\eqref{Humbert} with $b^2-4ac-4de=D'$.  This gives the equation
$$ -dDz_1z_2 \+ (a\om_1^2+b\om_1\om_2+c\om_2^2)z_1 \+ (a\om_1^2+b\om_1\om_2+c\om_2^2)^\s z_2 \+ e \= 0\,,$$
which has the form~\eqref{TNeq} with 
$$ N = -deD - N(a\om_1^2+b\om_1\om_2+c\om_2^2) \= \frac{DD'\,-\, (2Aa+Bb+2Cc)^2}4  $$
as asserted. In the special case $D'=1$, we recover the equation $D=r^2+4N$ and also see that we are on the 
component $F_N\bigl(\frac{r+\sD}{2\sD}\bigr)$ of~$F_N$, as claimed in~\eqref{RedLoc}, since it is easily seen that
$a\om_1^2+b\om_1\om_2+c\om_2^2\equiv\frac{r+\sD}2\!\pmod\df$.

\subsection{Teichm\"uller curves and Veech groups}
\label{sec:defTeich}

A {\em Teichm\"uller curve} is an irreducible algebraic curve $W$ in the moduli space $\M_g$
of curves of genus~$g$ which is a totally geodesic submanifold for the \Teichmuller\ metric. Teich\-m\"uller curves are generated by a pair consisting of a curve $C$ and a non-zero holomorphic one-form 
$\omega\in H^1(C,\Omega_C^1)$.  Such pairs are called {\em flat surfaces}. 
\changed{An introductory text to flat surfaces is the survey \cite{zo06}, 
for example.} On the set of flat surfaces there is an action of $\GL{\RR}$ and 
Teich\-m\"uller curves are the projection to $\M_g$ of the orbit  $\GL{\RR}\cdot (C,\omega)$.
The uniformizing group $\G$ such that $W = \HH/\Gamma$, called a {\em Veech group},
 can be read off from the flat geometry of the pair $(C,\omega)$. Let $K_\Gamma$ be the trace field of $\Gamma$
and $r=[K_\Gamma:\QQ]$. Teich\-m\"uller curves with $r=g$ are called  {\em algebraically primitive}.
Under the Torelli map, algebraically primitive \Teichmuller\ curves map
to the locus of abelian varieties with real multiplication by $K$ 
(\cite{Mo04} Theorem~2.6). In particular for $g=2$
the universal covering of an algebraically primitive \Teichmuller\ curve defines a map
$$ (\varphi_0, \varphi) : \HH \to \HH^2$$ equivariant with respect to the action of the Veech group (acting
on the left in the obvious way and on the right via its embedding into $\SL K\hookrightarrow\SL\R^2\,$). The geodesic 
definition of \Teichmuller\ curves implies that $\varphi_0$ is a M\"obius
transformation.  Moreover we may suppose $\varphi_0={\rm id}$ using appropriate choices in the universal covering map.
Consequently, \Teichmuller\ curves define Hilbert modular embeddings in the above sense. (\cite{Mo04}, Section~3.)

The space of flat surfaces $(C,\omega)$ is naturally stratified
by the number and multiplicities of the zeros of $\omega$. In particular, for $g=2$ we have
two strata $\Omega \M_2(1,1)$ and  $\Omega \M_2(2)$, corresponding to $\om$ having two distinct zeros or one double zero,
respectively. For $g=2$ we have the following classification for algebraically primitive \Teichmuller\ curves (\cite{McM03}, \cite{McM05}, \cite{Mo06},~\cite{McM06}).

\begin{Thm} \label{thm:TMclassg2}
There is only one \Teichmuller\ curve in the stratum $\OmM_2(1,1)$, \changed{called the decagon curve}.  It lies in the Hilbert modular surface $X_5$.

The stratum $\OmM_2(2)$ contains infinitely many algebraically primitive 
\Teichmuller\ curves~\changed{$W_D$}, each lying in a unique
Hilbert modular surface. For each non-square discriminant $D \geq 5$ the 
Hilbert modular surface $X_D$ contains exactly
one \Teichmuller\ curve if $D\not\equiv 1\pmod 8$ and exactly two if~$D\equiv 1\pmod 8$.

The union $W_D$ of the \Tei\ curves in $X_D$ \changed{other than the decagon curve} 
is the locus in $\M_2$ of curves whose Jacobians have real multiplication by  
$\O_D$ and such that the eigendifferential on which $\O_D$ acts via the embedding $K \subset \RR$ has a double zero.
\end{Thm}
 
The two components in the case  $D \equiv 1 \pmod 8$  are distinguished by a spin invariant  $\delta \in \{0,1\}$ 
and will be denoted by $W_D^\delta=\HH/\G_D^\delta$, so that $W_D=W_D^0\cup W_D^1$ in this case.  The definition of the spin 
invariant is given in~\cite{McM05} and will not be repeated here, but in \S\ref{sec:WD_theta} we will be able
to give a new and equivalent definition in terms of our description of \Tei\ curves via theta functions.

The \Teichmuller\ curves in $\OmM_2(2)$ admit the following characterization, which is an adaptation
of the criterion in~\cite{Mo04}, Theorem~5.3. Let $\cFF_i$ ($i=1$ or~2) 
be the two natural foliations of a Hilbert modular surface~$X_D$ for which 
the $i$-th cooordinate is locally constant in the uniformization. 

\begin{Thm} \label{thm:charTeich}
An algebraic curve $W\subset X_D$ is a union of \Teichmuller\ curves if and only if 
\newline {\rm\phantom{i}(i)} $W$ is disjoint from the reducible locus and
\newline {\rm(ii)} $W$ is everywhere transversal to $\cFF_1$. 
\end{Thm}
\par
\begin{proof}[Sketch of proof] If $W$ is a \Tei\ curve, then (i) and (ii) hold
by definition and by the fact that we can use the first coordinate as a parameter, respectively.
\par
For the converse recall that over a Hilbert modular surface the relative first cohomology 
with  coefficients in~$K$ splits into two eigenspaces, two local systems over~$K$ that
we denote by $\LL$ and $\wt \LL$ and that are interchanged by the Galois group of~$K$.
Consequently, over any curve in a Hilbert modular surface the cohomology splits in the same way. 
\par
Condition~(i)  is equivalent to $W$ being in the image
of the locus of Jacobians with real multiplication under the Torelli map. To apply 
the criterion of~\cite{Mo04}, Theorem~5.3, we need to show that the Kodaira-Spencer map
for $\LL$ or $\wt \LL$ vanishes nowhere on $\ol{W}$. Condition~(ii) implies the 
non-vanishing of the Kodaira-Spencer map for the corresponding $\LL$ in the interior 
of $X_D$, while at the cusps non-vanishing is automatic, by
a local calculation as in \cite{BM05}, Proposition~2.2.
\end{proof}
\par
One can generalize this setup using algebraic curves in \changed{$\cAA_g$} that 
are {\em totally geodesic for the Kobayashi metric}.  See~\cite{MV08} 
for a characterization of these Kobayashi geodesics. 

\subsection{Twisted modular forms for $W_D$.} \label{sec:topoTeich}
The topology of $W_D$ and the ratio $\lambda_2$ are completely determined, combining the work 
of several authors. We summarize the results and combine them with Theorem~\ref{thm:dimmodgeneral} 
to determine the dimension of the space of twisted modular forms.
\par
\begin{Thm} \label{thm:topoTeich} For any non-square discriminant $D$, the
fundamental invariants of the curves $W_D$ are as follows.
\begin{itemize}
\item[(i)] The orbifold Euler characteristic of $W_D$  equals 
\be \label{eq:chiWD}
\chi(W_D) = -\tfrac92 \chi(X_D),
\ee  
where $X_D$ is the Hilbert modular surface $\HH^2/\SLOD$.
\item[(ii)] The cusps of $W_D$ are in bijection with \Pred\ quadratic forms
of discriminant~$D$ (see Section~\ref{sec:cuspTheta0}).
\item[(iii)] For $D=5$, the curve $W_D$ has two \changed{elliptic} fixed points, 
one of order two and one
of order five. For $D \neq 5$, there are $e_2(D)$ elliptic fixed points of order two
on $W_D$ and no other fixed points, where $e_2(D)$ is a sum of class numbers of imaginary 
quadratic orders (\cite{Mu11}, Table~1).
In particular for $D \equiv 1 \mod 8$, there are $e_2(D) = \frac12 h(-4D)$ elliptic
fixed points of order two.
\item[(iv)] The curves $W_D^0$ and $W_D^1$ are defined over $\QQ(\sqrt{D})$ and are Galois conjugate.
\item[(v)] The curves $W_D^0$ and $W_D^1$  are homeomorphic.
\item[(vi)]  For a torsion-free subgroup of the Veech group of any component $W_D^i$ of
$W_D$ the ratio $\lambda_2$ of the degrees of the line bundles $\cLL$ and $\wt\cLL$ equals~$1/3$.
\end{itemize}
\end{Thm}
\par
\begin{proof} Statement (i) is the main result of~\cite{Ba07}, \changed{Theorem~1.1}. 
Statement (ii) is implicit
in \cite{McM05} and explicit in~\cite{Ba07}, \changed{Theorem 8.7(5)}. 
Statement (iii) is the main result of~\cite{Mu11}.
Statement (iv) is Theorem~3.3(b) of~\cite{BM07} and (v) follows directly. 
Statement~(vi) was shown in 
\cite{Ba07}, Corollary~12.4, and with a different proof in~\cite{BM07}, Corollary~2.4.
\end{proof}
\par
We recall that the value of $\chi(X_D)$ is known, and is given  for a fundamental discriminant $D$ by
\be \label{eq:chiformula}
\chi(X_D) \= 2\, \zeta_K(-1) \= \frac1{30}\sum_{D=b^2+4ac} a
\ee
(see \cite{Hi73}), and in general by a similar explicit formula. 
\par
Every curve of genus two is hyperelliptic and consequently, $-I$ is in the Veech group
for every \Teichmuller\ curve in genus two. The dimension of the space of
twisted modular forms can now be deduced from Theorem~\ref{thm:dimmodgeneral}.
\par
\begin{Thm} \label{thm:dimmodWD}
For $D>5$ the space of twisted modular forms $M_{k,\ell}$ on $W_D$ is zero for $k+\ell$ odd.
For  $k+\ell$ even and $D \not \equiv 1 \mod 8$
$$\dim M_{k,\l}(\G) \=  -\frac{1}{2} \Bigl(k+\frac{\l}{3}\Bigr)\, \chi(W_D) - 
\,\Bigl\{ \frac{-k + \ell}{4} \Bigr\}\,e_2(D),$$
where $\{x\}$ is the fractional part of $x$, and for each of the two components for 
$D  \equiv 1 \mod 8$
$$\dim M_{k,\l}(\G) \= -\frac{1}{4} \Bigl(k+\frac{\l}{3}\Bigr)\,  \chi(W_D) -\frac14\,\Bigl\{ \frac{-k + \ell}{4} \Bigr\} \,e_2(D),$$
where $\chi(W_D)$ and $e_2(D)$ is given in Theorem~\ref{thm:topoTeich}. 
\end{Thm}
\par
\begin{proof} The first statement holds because the Veech group contains~$-I$.
Given the general dimension calculation in Theorem~\ref{thm:dimmodgeneral} and the Euler characteristic in~\eqref{eq:chiWD}
it remains to show that for all the fixed points $x$ of order two the local contribution $b_x(k,\ell)$ is 
$\{\tfrac{-k + \ell}{4}\}$, not $\{\tfrac{-k - \ell}{4}\}$ for some of them.
\par
Suppose that $M=\sm abcd \in \SLOD$ is of order $4$ and stabilizes $\zz = (z_1,z_2) \in \HH^2$.  Then multiplication 
by the diagonal matrix with diagonal entries $((cz_1+d)^{-1}, (c^\sigma z_2 +d^\sigma)^{-1})$ 
defines a linear map $J$ of $\CC^2$ that stabilizes the lattice $\fracL_\zz$ from \eqref{eq:latticeemb}
(i.e.\ the corresponding abelian surface has complex multiplication by the ring generated by $\ord_D$ and $J$).
To show that  $b_\zz(k,\ell) = \{\tfrac{-k + \ell}{4}\}$ is hence equivalent to showing that $J^{-1}$ (or $J$)
has two eigenspaces of dimension one, rather than a two-dimensional eigenspace.
\par
Mukamel (\cite{Mu11}) studies, along with his classification of fixed 
points of \Teichmuller\ curves, 
the locus $\M_2(D_8)$ of genus two surfaces with automorphism group containing
the dihedral group of order $8$. He shows that all the fixed points of order two on the
\Teichmuller\ curves $W_D$ lie on the intersection (in \changed{$\cAA_2$}) 
of the Hilbert modular 
surface $X_D$ with $\M_2(D_8)$. The family of curves over \changed{$\M_2(D_8)$}  
is given by the hyperelliptic equation 
$$Y^2 = (X^2-1)(X^2+aX+1), \quad a \in \CC \setminus \{\pm 2\}.$$
The automorphism of order four is $J(X,Y) = (\tfrac1X,\tfrac{iY}{X^3})$ and the eigendifferentials
are $dX/Y + XdX/Y$ and $dX/Y - XdX/Y$, which lie in the eigenspace for $+1$ and for $-1$ respectively. This proves the claim on the $J$-eigenspaces.
\end{proof}
\par

\subsection{Gauss-Manin connection and Picard-Fuchs equation}  \label{sec:GMandPF}

Here we explain why \Tei\ curves give rise to twisted modular forms and how
to obtain the differential equations we attached to them in 
Section~\ref{sec:twistedDE} geometrically. For the moment, let $W$ be any curve 
in~$\M_2$ such that the corresponding family of Jacobians has real multiplication 
by an order in $K$. Then the vector bundle with fiber $H^1(C,\CC)$ over the point 
$[C] \in \changed{\M_2}$ splits (over $\RR$, and in fact over $K$), as in the proof of 
Theorem~\ref{thm:charTeich}, into rank two subbundles $\LL$ and $\widetilde\LL$. 
This vector bundle also comes with a flat (Gauss-Manin) connection $\nabla$.  
The bundles $\LL$ and~$\wt\LL$ come with holomorphic subbundles $\cLL$ and 
$\wt{\cLL}$ respectively, whose fibers over $X$ are the holomorphic one-forms 
on $C$ that are eigenforms for the real multiplication.
The bundles  $\cLL$ and $\wt{\cLL}$ naturally extend over the cusps $\ol{W} \setminus
W$,  where the fibers are stable forms. We denote them by the same letters.
(We recall that a form is called {\it stable} if in the limit as $t\to t_0$, where the 
genus~2 curve parametrized by~$t\ne t_0$ degenerates to a curve of genus~0 with 
double points, the corresponding differential on the normalization of this curve
has simple poles with opposite residues at the points that get identified.)

Suppose for simplicity that $W$ is a rational curve with parameter $t$. If
we choose sections $\omega(t)$ of $\cLL$ and $\wt\omega(t)$ of $\wt \cLL$, then 
$\{\omega(t), \nabla(\partial/\partial t)\omega(t),\nabla(\partial/\partial t)^2\omega(t)\}$
are linearly dependent in cohomology. Concretely, this means that if 
 $L$ is the corresponding second order differential linear operator, a quadratic polynomial in 
$\partial/\partial t$, then the image of $\omega(t)$ under $L$ is exact. Similarly, 
$\{\wt \omega(t), \nabla(\partial/\partial t)\wt \omega(t),\nabla(\partial/\partial t)^2\wt\omega(t)\}$ 
are linearly dependent and give a second order differential operator~$\wt L$ that makes $\wt \omega(t)$ exact.
It follows that the periods, defined as the integral of $\omega(t)$ and $\wt\omega(t)$
over any fixed element of $H_1(C,\CC)$ are annihilated by $L$ and $\wt L$ respectively. (Here
``fixed'' means that we use the property of being a local system to identify the homology
groups $H_1(C_t,\CC)=H_1(C_t,\Z)\otimes\CC$ with each other locally.)
These are the well-known {\em Picard-Fuchs differential equations} satisfied by periods.

Now assume that $W$ is a \Tei\ curve \changed{in $\M_2$ or, more generally, 
with quadratic trace field}. We show that the periods just described are (twisted)
modular forms of weight $(1,0)$ and $(0,1)$ respectively, with respect to a
modular embedding $\varphi$ as defined in Section~\ref{sec:hme}.
More precisely we have the following correspondence.
\par
\begin{Prop} \label{Prop5.6} 
Suppose that $W$ is a \Tei\ curve with uniformization $\HH/\G$ as above, and let~$L$ and~$\wt L$ be
the rank two differential operators associated with sections $\omega(t)$ and $\wt \omega(t)$
of~$\cLL$ and~$\wt\cLL$ as above. Then there is a rank-one submodule 
(in the rank-two $\mathcal{O}_{\ol{W}}(W)$-module of
solutions of~$L$) consisting of holomorphic modular forms of weight $(1,0)$,
and a rank-one submodule (in the rank-two $\mathcal{O}_{\ol{W}}(W)$-module of
solutions of~$\wt L$) consisting of twisted holomorphic modular forms of 
weight $(0,1)$. If $\omega(t)$ (resp.\ $\wt\omega(t)$) extends to a stable 
form over a 
cusp of $W$, then the corresponding (twisted) modular form is holomorphic at this cusp.
\par
This defines a $1\!:\!1$ correspondence between holomorphic sections of $\cLL$ over $\ol{W}$ and holomorphic 
twisted modular forms on~$\G$ of weight~$(1,0)$, and a $1\!:\!1$ correspondence between holomorphic sections of $\wt\cLL$
over $\ol{W}$ and holomorphic twisted modular forms on~$\G$ of weight~$(0,1)$.
\end{Prop}
\par

\begin{proof} In \cite{Mo04} it was shown that there exists an oriented basis $\beta,\,\alpha$ of the 
kernel of $\wt\omega$ in $H_1(C,\RR)$ such that the monodromy representation of
$\pi_1(W) = \Gamma$ on that subspace is the identity, and similarly a basis 
$\beta^\sigma,\,\alpha^\sigma$ of the kernel of $\omega$ with respect to which the monodromy representation
 is given by the Galois conjugate group $\Gamma^\sigma$. 
\par
Consequently, the period map $z \mapsto \int_{\beta} \omega(z)/\int_{\alpha} 
\omega(z)$ is equivariant with $\Gamma$ acting on domain and range, hence the 
identity after an appropriate conjugation by a M\"obius transformation. 
Moreover, the period map $z \mapsto  \int_{\beta^\sigma} \wt \omega(z)/\int_{\alpha^\sigma} 
\wt \omega(z)$ is equivariant with $\Gamma$ acting on the domain and 
$\Gamma^\sigma$ on the range. Hence this map agrees with $\p$ in the definition
of the modular embedding by the uniqueness of modular embeddings.
\par
As said above, the periods $f(z) = \int_{\alpha} \omega(z)$ and $f_1(z) = \int_{\beta} \omega(z)=zf(z)$ 
span the space of solutions of $L$ (pulled back to $\HH$ via $t$). The statement above about the monodromy implies that 
for all $\gamma = \sm abcd \in \Gamma$
$$ \left(\begin{matrix}f_1(\gamma z) \\ f(\gamma z) \end{matrix}\right)  
= \gamma \, \left( \begin{matrix}f_1(z) \\ f(z) \end{matrix}\right)=
\left( \begin{matrix}af_1(z) + bf(z) \\ cf_1(z) +df(z) \end{matrix}\right)=
\left( \begin{matrix}(az + b)f(z) \\ (cz +d)f(z) \end{matrix}\right).$$
The second row implies that $f$ is a twisted modular form for $\G$ of weight $(1,0)$. Similarly, the periods 
$\wt f(z) = \int_{\alpha^\sigma} \wt\omega(z)$ and $f_1(z) = \int_{\beta^\sigma} \wt\omega(z)=\p(z)\wt f(z)$ 
span the space of solutions of $\wt L$ and we have
$$ \left(\begin{matrix}\wt f_1(\gamma z) \\ \wt f(\gamma z) \end{matrix}\right)  
= \gamma^\sigma \, \left( \begin{matrix}\wt f_1(z) \\\wt f(z) 
\end{matrix}\right)=
\left( \begin{matrix}a^\sigma \wt f_1(z) + b^\sigma \wt f(z) \\ c^\sigma \wt f_1(z) +d^\sigma \wt f(z) \end{matrix}\right)=
\left( \begin{matrix}(a^\sigma \p(z) + b^\sigma)\wt f(z) \\ (c^\sigma \p(z) +d^\sigma ) \wt f(z) \end{matrix}\right).$$
Again, the second row implies that $\wt f$ is a twisted modular form for $\G$ of
weight $(0,1)$.
\par
Holomorphicity of $f$ and $\wt f$ in the interior of $\HH$ is obvious by the definition of a period. To show that 
they are holomorphic at the cusps, we may assume without loss of generality that $z_0=\infty$ and $t_0=0$.
There, it follows from the definition of the monodromy representation that $\alpha$ (resp.\ $\alpha^\sigma$) is characterized in 
$\langle \wt\omega\rangle^\perp$ (resp.~in $\langle \omega\rangle^\perp$) as the elements invariant under the local 
monodromy group. The period of a stable form along such a cycle is finite.
\par
To establish the last statement of the proposition, we just need to assign to every holomorphic twisted modular form 
of weight $(1,0)$ (resp.\ weight $(0,1)$) a section of $\cLL$ (resp.~of~$\wt\cLL$). This is well-known in the 
untwisted case and was done in both the untwisted and twisted cases in Section~\ref{sec:twistedDE} of~Part~I.
\end{proof}

We end with a remark on Galois conjugation and spin. We defined $\Gamma$ to be monodromy group of the local system~$\LL$. 
Then, of course, the monodromy group of the Galois conjugate $\wt \LL$ is $\Gamma^\sigma$.
We will see in the next section in an example, and at the end of the paper in general, that
the solutions $y$ and $\wt y$ also have coefficients in the field~$K$. However,
the Galois conjugate solution $y^\sigma$ is neither equal to $\wt y$ nor to any other solution of $\wt \LL$. 
In fact, $y^\sigma$ is naturally a solution of a differential operator associated with the Galois conjugate
\Teichmuller\ curve $W^\sigma$. For $D \equiv 1 \mod 8$ this is the curve
with the other spin invariant (see Theorem~\ref{thm:topoTeich}~iv)).  
For $D \not\equiv 1 \mod 8$ this Galois conjugate curve is isomorphic to the 
original curve by McMullen's classification recalled in Theorem~\ref{thm:TMclassg2}. 
The equation of this curve for $D = 13$ is given explicitly in~\cite{BM07}.

\section{Example: the curve $W_{17}^1$ and its associated differential equations}
\label{sec:BouwMoeller}

Our running example, from now until the end of Part~II, will be the \Teichmuller\ curve
$W_{17}^1$ on the Hilbert modular surface $X_{17}$. In this section we gather
the known results for this curve, computing the Veech group and summarizing the 
construction from~\cite{BM07} to compute the equation of 
the universal family and the corresponding Picard-Fuchs differential equations.


\subsection{The Veech group for $D=17$ and spin~1.} \label{generators}
For small values of $D$ the Veech groups $\G_D^\w = \G(W_D^\w)$ can be calculated using the algorithm in~\cite{McM03}. \changed{This is sufficient for our purposes, but we 
emphasize that a general algorithm to compute the Veech group for any Veech surface
has been developed and implemented by Mukamel~\cite{muk12}.} 
We describe this in detail for the case $D=17$, $\w=1$, i.e., for the \Tei{} curve of non-trivial spin.
McMullen's algorithm gives a subgroup of $\SL{\O_D}$, so the group that we will get (which we will denote
simply by $\G$, or by $\G_{17}$ when needed for clarity, with quotient~$W=\H/\G$) is actually the conjugate 
$\Delta\G_{17}^1\Delta\i$ of $\G_{17}^1$ by $\Delta=\sm100\sD$, and for the same reason the function
$\varphi$ used to make the modular embedding will go from $\HH$ to~$\HH^-$ rather than from~$\HH$ to~$\HH$.
Later, when we use this modular embedding explicitly to compare the twisted modular forms on~$W$ with
standard Hilbert modular forms for~$\Q(\sqrt{17})$, we will conjugate back to make the comparison easier.

We denote by $\a=(1+\sq)/2$ the standard generator of $\O=\O_{17}$ over~$\Z$, and for ease of reading will sometimes
use the abbreviated notation $[m,n]$ for $m+n\a\in\O$.  The group~$\G$ can be embedded into~$\SL\R$ by the standard
embedding of $\O$ into~$\R$ and then acts discretely.  (Note that the other embedding of~$\O$ into~$\R$ would lead
to a non-discrete subgroup of~$\SL\R\,$!)  A fundamental domain for this 
action is shown in Figure~\ref{cap:Ltable}\,(b),
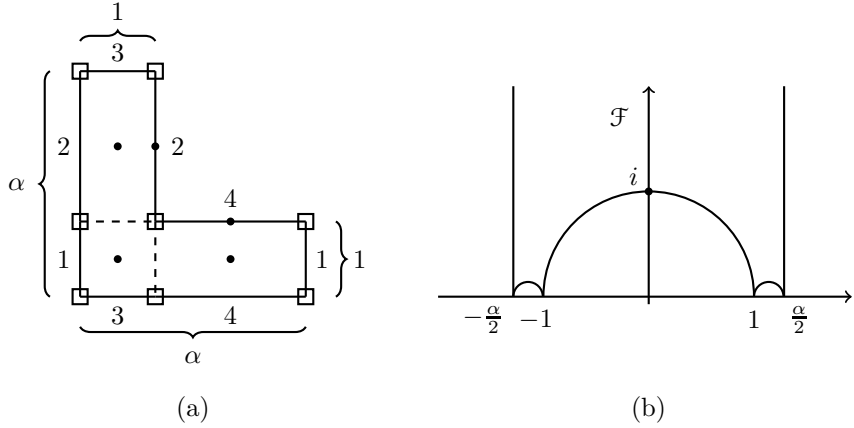
\begin{figure}[h]  
\begin{center}  
\begin{tikzpicture}[shift={(-1,0)}]
\fill (.5,.5) circle (1.5pt);
\fill (.5,2) circle (1.5pt);
\fill (1,2) circle (1.5pt);
\fill (2,.5) circle (1.5pt);
\fill (2,1) circle (1.5pt);
\draw (2,1.3) node {$4$};
\draw (1.3,2) node {$2$};
\draw[thick] (0,1) -- (0,3) node[midway,left] {$2$};
\draw[thick] (0,0) -- (0,1) node[midway,left] {$1$};
\draw[thick] (0,0) -- (1,0) node[midway,below] {$3$};
\draw[thick] (1,0) -- (3,0) node[midway,below] {$4$};
\draw[thick] (0,3) -- (1,3) node[midway,above] {$3$};
\draw[thick] (1,3) -- (1,1);
\draw[thick] (1,1) -- (3,1);
\draw[thick] (3,1) -- (3,0) node[midway,right] {$1$};
\draw[thick,dashed] (0,1) -- (1,1);
\draw[thick,dashed] (1,1) -- (1,0);
\draw[thick] (-.1,2.9) rectangle (.1,3.1);
\draw[thick] (-.1,.9) rectangle (.1,1.1);
\draw[thick] (-.1,-.1) rectangle (.1,.1);
\draw[thick] (.9,2.9) rectangle (1.1,3.1);
\draw[thick] (.9,.9) rectangle (1.1,1.1);
\draw[thick] (.9,-.1) rectangle (1.1,.1);
\draw[thick] (2.9,.9) rectangle (3.1,1.1);
\draw[thick] (2.9,-.1) rectangle (3.1,.1);
\draw [thick,decorate,decoration={brace,amplitude=4},xshift=-0.4cm,yshift=0cm] (0,0) -- (0,3)
 node [black,midway, left, xshift=-.2cm, yshift=0cm] {$\alpha$};
\draw [thick,decorate,decoration={brace,amplitude=4},xshift=0cm,yshift=-0.4cm] (3,0) -- (0,0)
 node [black,midway, below, xshift=0cm, yshift=-.2cm] {$\alpha$};
\draw [thick,decorate,decoration={brace,amplitude=4},xshift=0cm,yshift=0.4cm] (0,3) -- (1,3)
 node [black,midway, above, xshift=0cm, yshift=0.15cm] {$1$};
\draw [thick,decorate,decoration={brace,amplitude=4},xshift=0.4cm,yshift=0cm] (3,1) -- (3,0)
 node [black,midway, right, xshift=0.1cm, yshift=0cm] {$1$};
\draw (1.5,-1.5) node {(a)};
\end{tikzpicture}
\qquad
\begin{tikzpicture}
\draw[thick,->] (0,0) -- (5.5,0);
\draw[thick,->] (2.8,-0.1) -- (2.8,2.8);
\draw[thick,-] (4.6,0) -- (4.6,2.8);
\draw[thick,-] (1,0) -- (1,2.8);
\draw[thick] (1.4,0) arc (180:0:1.4);
\draw[thick] (4.2,0) arc (180:0:0.196);
\draw[thick] (1,0) arc (180:0:0.196);
\draw (.6,-.3) node {$-\frac{\alpha}{2}$};
\draw (1.3,-.3) node {$-1$};
\draw (4.2,-.3) node {$1$};
\draw (4.8,-.3) node {$\frac{\alpha}{2}$};
\draw (2.6,1.6) node {$i$};
\draw (2.4,2.4) node {$\mathcal{F}$};
\fill (2.8,1.4) circle (1.5pt);
\draw (2.8,-1.5) node {(b)};
\end{tikzpicture}

\end{center}
    \caption{(a) Flat surface generating the \Teichmuller\ curve $W_{17}$ when $\alpha=(1+\sq)/2$. The
     square is the double zero of $\omega$, the black points are the remaining $5$ Weierstrass points.
\newline
(b) A fundamental domain for $\Gamma$. }
 \label{cap:Ltable}
\end{figure} 
while Figure~\ref{cap:Ltable}\,(a) shows the explicit ``L-shaped region'' needed to apply the algorithm and obtain the fundamental domain. 
The group~$\G$ has three cusps, at~$z=\infty$,~1, and~$\a/2$, and an elliptic fixed point of order~2 at~$z=i$, where~$z$ is the
coordinate in~$\H$.  The stabilizers of the cusps are the infinite cyclic groups generated by the parabolic elements
$$  M_\infty=\mat1\a01, \;\,  M_1 \=\mat{-2\a-2}{2\a+3}{-2\a-3}{2\a+4},\;\, M_{\a/2}\=\mat{-2\a-3}{3\a+4}{-2\a-2}{2\a+5}\,,  $$
and the stabilizer of~$i$ is generated by the element $M_i=\sm01{-1}0$. The presentation of the group~$\G$ given by
McMullen's algorithm is then
$$ \G \= \bigl\langle M_\infty\,,\;M_i\,,\; M_1\,,\;M_{\a/2}\;\bigl|\; M_i^2\,=\,-1,\; M_\infty M_iM_1M_{\a/2}\,=\,1\bigr\rangle\;.$$

It will be useful in the sequel to deal not only with the Fuchsian group~$\G$ but also with a certain index~4 subgroup 
$\Pi$ of it, already mentioned in~\S\ref{sec:dimvan}. This group is more convenient for purposes of
calculation because it is free and also because the universal genus~2 curve over $\overline{\H/\Pi}$ has a stable model. 
We want that $\wt\Pi=\Pi\times\{\pm1\}$ has index two in $\G$. (This already implies that $\Pi$ has no torsion and
is thus free.) This group is not unique. We fix the choice 
$$ \Pi \= \langle M_\infty,\; M_{\a/2},\; M_1^2,\; -M_1^{-1} M_\infty M_1 \rangle\,.$$
The group $\wt\Pi=\Pi\times\{\pm1\}$ is the kernel of the homomorphism from 
$\G \to \{\pm1\}$ sending $M_\infty$ and $M_{\a/2}$ to~1 and $M_i$ and $M_1$ to~$-1$.  

Both curves $\H/\G$ and $\H/\Pi$ have genus~0, so there are modular functions $s(z)$ and $t(z)$ on~$\G$ and $\Pi$
giving isomorphism between their compactifications and~$\P^1(\C)$ (hauptmodules).  We can normalize them so that the
involution induced by $M_i\,:\;z\mapsto-1/z$ corresponds to $t\mapsto1/t$ and that the three cusps of $\H/\Pi$ are at~0, 1
and~$\infty$. Then the quotient map $t\mapsto s$ is given by
 \be\label{sandt}  s \=  -\frac{4\,\kappa_0\,t}{(t-1)^2}\,,\qquad\text{where}\quad \kappa_0=\frac{-895+217\sq}{256} \ee
and the values of~$s$ and $t$ at the cusps and elliptic fixed points are given, according to the calculations in~\cite{BM07}
(where a different parameter on $\H/\Pi$ was used), by the table 

\be\label{VALUES} \begin{tabular}{|c||c|cccccc|} \hline   & $z$ & $\infty$ & 0 &  1 & $\a/2$ & $-2/\a$ & $i$ \\
\hline $\Pi$ & $t=t(z)$ & $0$ & $\infty$ & 1 & $\la\i$ & $\la$ & $-1$ \\ 
\hline $\G$ & $s=s(z)$ & 0 & 0 & $\infty$ & $1$ & $1$  & $\kappa_0$ \\
\hline \end{tabular} \ee
\smallskip\noindent  where 
\be \label{deflambda} \la\,=\,t\Bigl(-\frac2\a\Bigr)\,=\,\frac{31-7\sq}2\,, \qquad \kappa_0\,=\,s(i)\,=\, -\frac{(\la-1)^2}{4\la}\,. \ee
For later use we emphasize that both $t$ and $s$ are local parameters of the Teich\-m\"uller curve
at the cusp $z=\infty$.  The whole situation is summarized by the following diagram.

$$\xymatrix{ \overline{W}_\Pi = \overline{\H/\Pi} \ar[r]^{{\phantom x}t} 
\ar[d]_{2:1\;} & \P^1(\C) \ar[d]^{\;s\,=\,-\frac{4\kappa_0t}{(1-t)^2}}\\
\overline{W} = \overline{\H/\G} \ar[r]^{{\phantom x}s}   & \P^1(\C)  \\}  $$

\subsection{The universal family over $W$.} \label{UnFa}
The modular curve $W=\HH/\G$ parametrizes a family of genus 2 curves with real 
multiplication by $\O=\Z+\Z\a$ on their Jacobians.  This family, and its associated Picard-Fuchs differential
equations, was determined explicitly in~\cite{BM07}.  In this subsection we review these results, and also give 
somewhat simpler equations by making suitable changes of coordinates. \changed{We remark
that meanwhile equations of more Teichm\"uller curves have been computed by
a different method by Kumar and Mukamel (\cite{KuMu14}).}

The explicit equation in \cite{BM07} was actually given for the family over the double cover $\HH/\Pi$
of~$W$, with the parameter~$t$, and has the form
 \ba \label{BMform}  & Y^2 \= P_5(X,t) \= \bigl(X\+(At+B)\bigr)\,\times\,\bigl(X\+(Bt+A)\bigr)\,\times \\
    & \quad\times\,\bigl(X^3\+C(t+1)X^2\+(D(t+1)^2+Et)X\+F(t+1)^3+Gt(t+1)\bigr)\,, \ea
with coefficients $A,\dots,G$ given (with our above notation $[m,n]=m+n\a$) by
\ba \label{BMcoeffs} & A=5\,[2,1]\,,\quad B=-2\,[5,3]\,,\quad C=[3,1]\,,\quad D=-\frac14\,[827,529]\,,\\
  & E=2^4\cdot17\,[3,2]\,,\quad F=-\frac12\,[4597,2943]\,, \quad G=2\cdot17\cdot[271,173]\,.\ea
(These coefficients are not quite as bad as they look since they all factor into small prime factors, 
e.g.~$D=-\2^{11}/4\ve^3$ and $F=-\pi_2^{14}\ve^4/2$ where $\ve=4+\sq$ is the fundamental unit of~$\Q(\sq)$ 
and $\2=(3+\sq)/2$ one of the prime factors of 2.)  We explain briefly how this equation is derived.

We can represent the fiber $C_t$ over~$t$ as a hyperelliptic curve $Y^2=P_6(X,t)$,
where $P_6$ is a polynomial of degree~6 in~$X$ whose roots correspond to the 
Weierstrass points of $C_t$.  From the action of the \changed{subgroup 
$\Pi\subset\G$ in} the Veech groups on these points (which can be
analyzed by looking at the Figure~\ref{cap:Ltable}{A}, in which the Weierstrass points are indicated by black points), 
we see that they break up into three orbits of size~1 and one of size~3, with two of the 1-element orbits being 
interchanged by the symmetry $t\mapsto t\i$. Placing the other 1-element orbit, the singularity of~$\om$,
at infinity, we get a new equation of the form $Y^2=P_5(X,t)$ where
$P_5$ factors into two linear and one cubic polynomial, and by degree computations together with the
symmetry under $t\mapsto t\i$ we find that these factors must have the form given in \eqref{BMform}
for some constants $A,\dots,G$. They are not unique, since we can make a change of variables $X\mapsto\a X+\b(t+1)$,
but become unique up to scaling if we assume that the two eigendifferentials $\om=dv_1$ and $\tom=dv_2$ are $dX/Y$ 
and $X\,dX/Y$, respectively.  To determine them, we note that at each cusp $t=c\in\{0,\,1,\,\la\}$ this polynomial must
acquire two double roots (the degeneracy of the genus~2 curve at infinity cannot consist of just two roots
coming together, because the real multiplication forces the subspace of $H_1(C;\Z)$ that collapses to be an $\O$-module 
and hence to have even rank over~$\Z$), so we have $P_5(X,c)=(X-X_0^c)(X-X_1^c)^2(X-X_2^c)^2$.  (The corresponding 
factorizations at the two other cusps $t=\infty$ and $\la\i$ are then automatic because of the $t\mapsto t\i$ symmetry.) 
These conditions do not yet suffice to determine the equation, but there is one further condition at each cusp.
This comes from the fact that a singular curve of the form $Y^2=(X-X_0)(X-X_1)^2(X-X_2)^2$ has genus~0. 
A parametrization with parameter $T$ is 
given by $X=(X_1-X_0)T^2+X_0$ and $Y=(X_1-X_0)^{5/2}T(T^2-1)(T^2-\rho^2)$, where 
$\rho^2=(X_2-X_0)/(X_1-X_0)$ denotes the cross-ratio of $X_0$, $X_1$, $X_2$ and $\infty$.  
The differential form $dX/Y$ corresponds under this map to a multiple of the differential form
$\Bigl(\frac1{T-1}-\frac1{T+1}+\frac{\rho\i}{T-\rho}-\frac{\rho\i}{T+\rho}\Bigr)\,dT$ on $\P^1$ having four simple poles 
with residues summing to~0 in pairs, and with the ratio of the non-paired residues being $\pm\rho$.
But for the cusps of the \Teichmuller\ curves we know {\it a priori} that these ratios of residues must equal the 
ratio of the top and bottom sides of the $L$-shaped region as it degenerates. From the horizontal sides
Figure~\ref{cap:Ltable}{A} we read off the value $\rho_0=\a$. Redrawing this figure decomposed into cylinders
in the direction of slope one (corresponding to the cusp $t=1$) and slope $\alpha/2$ 
(corresponding to the cusp $t=\la$ by the table in the preceding subsection) we find  
$\rho_1=\a/2$, $\rho_\la=(1+\a)/2$.  This information now suffices to determine all of the unknown
coefficients, up to the ambiguity already mentioned (in particular the second eigendifferential form $\tom=X\,dX/Y$ 
automatically has the correct ratio of residues, namely, the Galois conjugates of the ones for $\om$, so that there 
are no extra restrictions on the coefficients coming from this condition), and carrying out the calculation we
find the values given in~\eqref{BMcoeffs}.

We remark that equation \eqref{BMform} can be simplified considerably by substituting $(1+t)(1+X\sq)/4$ for $X$,
in which case $P_5(X,t)$, up to a factor $(\sq(1+t)/4)^5$, takes on the much simpler form 
$$ F_5(X,u)\=\bigl((X-1)^2\,-\,[4,5]\,u\bigr)\,\bigl((X+\sq)(X+1)^2\,-\,8u(2X+[9,5])\bigr)$$
with $u=4\ve\,\bigl(\frac{1-t}{1+t}\bigr)^2=\frac{4\ve}{1-s/\kappa_0}$.  This gives an 
explicit and relatively simple equation for the family of genus~2 curves over the \Teichmuller\ curve $\HH/\G$.

\subsection{The Picard-Fuchs equations for $W$ and their solutions.} \label{sec:BMDE}

As already discussed in Section~\ref{sec:dimvan}, even though we are considering only the single curve $W=\HH/\G$, 
there are {\it two} Picard-Fuchs differential equations, corresponding to the variation of the periods of the two 
eigendifferentials $\om$ and $\tom$ for the
action of $\O$ on the space of holomorphic differentials of the fibers.  It will be crucial for our
calculations to have both of them, since together they will tell us explicitly how the \Teichmuller\
curve $W$ is embedded in the Hilbert modular surface $X_{17}$.

Obtaining the Picard-Fuchs differential equations satisfied by the periods of the two eigendifferentials
$\om$ and $\tom$ is straightforward once the equation of the family of curves has been obtained.  
One has to find differential operators $L$ and $\wt L$ mapping the one-forms $\omega$ and $\wt\omega$ to exact forms.
The result, given in~\cite{BM07}, is a pair of differential operators of the same form as in~\eqref{diffeq10}, namely  
\be \label{BMDE} L \= \frac d{dt}\, A(t)\,\frac d{dt}\+B(t)\,,\quad \tL \= \frac d{dt}\, \tA(t)\,\frac d{dt}\+\tB(t)\,,  \ee
where $A(t)$ and $B(t)$ are the polynomials given by
\be \label{AB} \begin{aligned}
   A(t)& \=t\,(t-1)\,(t-\la)(t-\la\i)\=t^4\,-\,\b t^3\+\b t^2\,-t\,,\\ 
   B(t)&\=\frac34\, \bigl(3t^2\,-\,(\b+\g)\,t\+\g\bigr) \,,\end{aligned}  \ee
with $\la=(31-7\sq)/2$ as in~\eqref{deflambda} and $\b$ and $\g$ defined by
\be \label{defbetgam} \b=\la\+\la\i+1=\frac{1087-217\sq}{64}\,,\quad  \g=\frac{27-5\sq}4\;,   \ee
and where $\tA(t)$ and $\tB(t)$ are the rational functions
\ba \tA(t)&\=A(t)\Bigl/\Bigl(t^2\+\frac{137-95\sq}{128}\,t\+1\Bigr)\,,\\ 
  \tB(t)&\= \Bigl(\frac14\,t^4\+\frac{1113-399\sq}{512}\,t^3\,-\,\frac{260375-69633\sq}{16384}\,
  t^2 \\&\,-\,\frac{1387-301\sq}{128}\,t\+\frac{23-5\sq}8\Bigr)\Bigl/\Bigl(t^2\+\frac{137-95\sq}{128}\,t\+1\Bigr)^2\,.\ea
The differential operator $L$ has five singularities, at infinity and at the roots of $A(t)$.  The differential operator
$\tL$ has seven singularities, these five and two more at the poles of $\tA$, but these last two are only 
apparent singularities of the differential equation, i.e., all solutions of the equation are holomorphic at these points.

The unique solutions in $1\+t\,\CC[[t]]$ of the differential equations $Ly=0$ and $\tL\ty=0$ can easily be calculated
recursively. The first few terms are given by
  \bas y &\= 1 \+ \tfrac{81 - 15\sq}{16}\,t \+  \tfrac{4845- 1155\sq}{64}\,t^2  \+ \tfrac{3200225 - 775495\sq}{2048}\,t^3 \+\,\cdots \\
       & \;\approx\; 1 \+ 1.197\,t \+ 1.294\,t^2 \+ 1.356\,t^3 \+ 1.402\,t^4 \+ 1.439\,t^5 \+ \,\cdots  \\
   \ty &\= 1 \+ \tfrac{23 - 5\sq}{8}\,t \+ \tfrac{5561- 1343\sq}{128}\,t^2 \+ \tfrac{452759 - 109793\sq}{512}\,t^3 \+\,\cdots \\
       & \;\approx\;  1 \+ 0.2981\,t \+ 0.1849\,t^2 \+ 0.1384\,t^3 \+ 0.1131\,t^4 \+ 0.0973\,t^5  \+\,\cdots  \eas
There are also unique power series $y_1$ and $\ty_1$ without constant term such that $y\,\log(t)+y_1$ and $\ty\,\log(t)+\ty_1$ 
are solutions of the same differential equations as $y$ and $\ty$, respectively.  These series begin
   \bas y_1 &\= \tfrac{439 - 97\sq}{64}\,t \+ \tfrac{563089 - 135575\sq}{4096}\,t^2 
             \+ \tfrac{200641639 - 48642353\sq}{65536}\,t^3 \+ \,\cdots  \\
            &\;\approx\; 0.6103\,t \+ 1.001\,t^2 \+ 1.283\,t^3 \+ 1.504\,t^4 \+ 1.687\,t^5 \+ \, \cdots \\
      \ty_1 &\=  \tfrac{1575- 369\sq}{128}\,t \+ \tfrac{1749337- 423695\sq}{8192}\,t^2 
             \+ \tfrac{1764480419 - 427927381\sq}{393216}\,t^3 \+ \,\cdots\\
   &\;\approx\;  0.4185\,t \+ 0.2927\,t^2 \+ 0.2305\,t^3 \+ 0.1958\,t^4 \+ 0.1748\,t^5 \+ \, \cdots    \eas

We have given the numerical values of the first coefficients of each of these four power series 
to emphasize that they are quite small (and the same is true of the first few hundred, which we have computed).
In fact, the coefficients in each case grow like $\la^n$, where $\la=1.06913\cdots$ is the number defined 
by~\eqref{deflambda}, since the radius of convergence is the absolute value of the nearest singularity $t\ne0$, and
the singularities are at $t=0$, $\la\i$, $1$, $\la$ and $\infty$.  The growth in each case is quite regular,
with the coefficient of~$t^n$ being asymptotic to a constant times~$\la^n/n$.  It is perhaps worth mentioning that if we
took the Galois conjugates~$L^\s$ and $\tL^\s$ of the differential operators~$L$ and $\tL$, which give the Picard-Fuchs 
equations for the other \Tei\ curve $W_{17}^0$ (cf.~\cite{BM07} or \S\ref{sec:ThetaDer}), then the power series $y,\dots,\ty_1$
would also be replaced by their Galois conjugates and would look algebraically very similar to those above, but would
have completely different real coefficients and growth, e.g., the expansion of $y^\s$ begins
 \bas y^\s &\= 1 \+ \tfrac{81+15\sq}{16}\,t \+  \tfrac{4845+1155\sq}{64}\,t^2  \+ \tfrac{3200225+775495\sq}{2048}\,t^3 \+\,\cdots \\
       & \;\approx\; 1 \+ 8.928\,t \+ 150.11\,t^2 \+ 3123.9\,t^3 \+ 71667\,t^4 \+ 1738907\,t^5 \+ \,\cdots\,,  \eas
now with coefficients growing like $(\la^\s)^n/n$ with $\la^\s=29.93086\cdots$.  The corresponding Fuchsian group, although isomorphic 
to~$\Pi$ as an abstract group, is not conjugate to it in $\SL\R$, and the quotients of the upper half-plane by these two
groups, which are the curves $W_D^0$ and $W_D^1$, represent different points of the moduli space~$\M_{0,5}$.
 
Another very striking property of the expansions of~$y$ and $\ty$ given above (and then of course also 
of their conjugates $y^\s$ and $\ty^\s$) is that the only denominators one sees are powers of~2, i.e., 
the first few coefficients of these power series all belong to the ring $\O[\tfrac12]$.  A calculation to
higher accuracy shows that the same holds for the first few hundred coefficients, and in fact it is a 
theorem, proved in~\cite{BM07}, that it holds for all coefficients.  We will return to this question at
the end of the next section because it is has a very interesting aspect that was in fact the point of departure
for our whole investigation.  

\section{Arithmetic properties of modular forms for $W_{17}^1$}\label{sec:arithmeticW17} 

With the preparations in the preceding sections we can now compute in \S\ref{sec:DEmodular}
the modular embedding~$\varphi$ in the example $D=17$. In the process we compute
the Fourier expansions of some modular forms and later, in \S\ref{sec:Twisted} we 
completely determine the ring of twisted modular forms in this specific example. 
This arithmetic of the coefficients reveals two surprising phenomena, a transcendental
constant needed for the correct choice of the $q$-parameter and the integrality
statement mentioned above and proved in~\cite{BM07}, which cannot be explained using one modular
$q$-variable. We discuss these in \S\ref{sec:DEintegral}, and provide the explanations in Section~\ref{sec:eqfromdiffeq}.

\subsection{Modular parametrization of the differential equations} \label{sec:DEmodular}
We have already mentioned that the differential equations $Ly=0$ and $\tL\ty=0$ have the same form
as the differential equation~\eqref{diffeq10} satisfied by ordinary or twisted modular forms with respect
to a hauptmodule.  This is of course not a coincidence:  we have
$$  y(t(z)) \= f(z)\,, \;\quad \ty(t(z)) \= \tf(z) \qquad \text{for $\Im(z)$ large, $|t(z)|$ small} $$
where $t:\H/\Pi\to\C$ is the map defined in~\S\ref{generators} and $f(z)$ and $\tf(z)$ are
a modular form of weight~1 and a twisted modular form of weight~(0,1), respectively, on the same group~$\Pi$.
In this subsection we will work out this statement in more detail, obtaining in particular
a way to calculate the expansion \eqref{eq:phiabsexpansion} of the function $\p:\H\to\Hm$ whose graph gives 
the embedding of $\H/\Pi$ into $\H\times\Hm/\SL{\O_{17}}$ as discussed in Section~\ref{sec:hme}.  (Here we need $\Hm$,
rather than~$\H$ as in Section~\ref{sec:hme}, because we have conjugated the original Veech group by~$\D=\sm100{\sq}$ to embed it into 
$\SL{\O_D}$ and the Galois conjugate of~$\D$ has negative determinant.)  We will also calculate the $q$-expansion
of $\tf(z)$, obtaining our first explicit example of a twisted modular form. 
In Section~\ref{sec:eqfromdiffeq}
we will use this information to determine completely the rings of twisted modular forms for~$\G$ and $\Pi$ and the
algebraic description of the \Tei\ curve $W$ inside the Hilbert modular surface $X_{17}$.

We begin with the functions $y$ and~$f$.  As stated in \S\ref{generators}, the cusp at infinity 
for either~$\G$ or~$\Pi$ has width~$\a$, i.e., its stabilizer is generated by the transformation $z\mapsto z+\a$, 
where~$\a=(1+\sqrt{17})/2$ is our standard generator of~$\O_{17}$, so any modular function or modular form on 
either group can be written as a power series in the variable $q=e^{2\pi iz/\a}$.  On the other hand, we know that
the space of solutions of the differential equation satisfied by any weight~1 modular form $f(z)$ with respect 
to any modular function on the same group is spanned by $f(z)$ and~$zf(z)$.  Our first thought is thus that~$q$
coincides with the ``mirror parameter"
 \be \label{defQ}  Q \= Q(t) \= t\exp(y_1/y)\,, \ee
where $y=y(t)$ and $y_1=y_1(t)$ are the two power series in~$t$ defined in~\S\ref{sec:BMDE}.  This is indeed what happens 
in the case of the Ap\'ery or Ap\'ery-like differential equations (see~\cite{Z09}), at least if one normalizes the 
hauptmodule correctly.  Here, however, it is not quite true. We can see this numerically as follows.
The function $Q(t)$ has a Taylor expansion beginning 
 \be \label{Qwrtt} Q(t) \= t \+ \tfrac{439 - 97\sq}{64}\,t^2 \+ \tfrac{249125 - 60195\sq}{2048}\,t^3 \+ \cdots\,. \ee
We can invert this power series to obtain
 \be\label{twrtQ} t\= t(Q) \= Q \,-\, \tfrac{439 - 97\sq}{64}\,Q^2 \+ \tfrac{103549-24971\sq}{2048}\,Q^3 \+ \cdots \ee
and then substitute this into the expansion of $f(z)=y(t)$ to express $f(z)$ as a power series
 \ba \label{fwrtQ} f(z) &\=  1 \+ \tfrac{81-15\sq}{16}\,Q \+ \tfrac{8613-2019\sq}{512}\,Q^2 \+ \tfrac{726937-175823\sq}{16384}\,Q^3 \+ \cdots \\
  &\;\approx\; 1 + 1.197\,Q + 0.563\,Q^2 + 0.122\,Q^3 + 0.0082\,Q^4  -0.0011\,Q^5 - \,\cdots \ea
in the new local parameter~$Q$ at infinity.  Looking at the first few numerical coefficients in this
expansion, we see that they seem to be tending to~0 rapidly, suggesting that the radius of convergence of 
this power series is larger than~1, which is the value it would have to have if we were expanding with respect to~$q$.  
The point is that, although the function $\,\log Q=\log t\+y_1/y$ has the same behavior at infinity as~$2\pi iz/\a$,
namely, that it is well-defined up to an integer multiple of~$2\pi i$, this property determines it only up to
an additive constant. Therefore $q$ and $Q$ are related by 
 \be\label{eq:defA} Q \= A\,q \= A\,e^{2\pi iz/\a} \ee
for some constant $A\ne0$ that has no reason to be equal to~1.  The radius of convergence of the series in~\eqref{fwrtQ}
is then equal to the absolute value of this constant. 

We can use this idea, or a modification of it, to calculate~$A$ numerically.  First, by computing a few hundred coefficients of 
the series in~\eqref{fwrtQ} and calculating its radius of convergence by the standard formula $R=\liminf|a_n|^{-1/n}$,
where $a_n$ denotes the $n$th coefficient, we find that $|A|$ is roughly equal to 7.5.  However, this direct approach has
very poor convergence (because the coefficients of the expansion of $f(z)$ in~$Q$, unlike those of the same function when
written as a power series $y(t)$ in~$t$, do not behave in a regular way), and anyway gives only the absolute 
value of the scaling constant~$A$.  To find the actual value to high precision, we apply a simple trick.
From the data in the table~\eqref{VALUES}, we know that the value of $t(z)$ at $z=i$ equals~$-1$ and that this value
is taken on with multiplicity~1 (because $t$ is a hauptmodule for a group with no elliptic fixed points) and is
not taken on at any point in the upper half-plane with imaginary part bigger than~1 (because $i$ and its translates
by multiples of~$\a$ are the highest points in the $\Pi$-orbit of~$i$).   It follows that the function $1/(t(z)+1)$ has 
a simple pole at $z=i$ and that if we express this function as a power series in~$q$ (resp.~$Q$), then its 
singularity nearest the origin is a simple pole at $q_0=e^{-2\pi/\a}$ (resp.~$Q_0=Aq_0$).  In other words,
$1/(1+t)$ is the sum of $c/(1-Q/Q_0)$ for some non-zero constant~$c$ and a function holomorphic in a disc 
of radius strictly larger than~$|Q_0|$. This implies that if we expand  $1/(1+t)=\sum b_nQ^n$, then the coefficients $b_n$
are given by $b_n=cQ_0^{-n}(1+\text O(a^{-n}))$ for some~$a>1$, and hence that the quotients $b_{n+1}/b_n$ tend to~$Q_0$
with exponential rapidity.  Calculating a few hundred of the coefficients~$b_n$ numerically, we find from this the value
  $$  A \;\approx\; -7.48370822991173536914114556623211$$ 
to very high precision.  After some trial and error we can recognize this number ``in closed form" as
  \be \label{eq:A}  A \;\overset?=\; -2\,\bigl(3+\sq\,\bigr)\;\Bigl(\frac{5-\sq}{2}\Bigr)^{(\sq-1)/4}\,, \ee
and we will see later that this guessed value is indeed the correct one.  

Equations~\eqref{twrtQ}--\eqref{eq:A} now give as many terms as desired of the $q$-expansions of the 
modular function~$t(z)$ and modular form~$f(z)$.  We can (and of course did) then use this to check the correctness of
these equations numerically to high accuracy by verifying the invariance of $t(z)$, and the invariance of $f(z)$ up to 
an automorphy factor $cz+d$, under modular transformations $z\mapsto(az+b)/(cz+d)$ in the group~$\Pi$.
 Similarly, by inverting~\eqref{Qwrtt} we can also give the inverse of the uniformizing map $\H\to W_\Pi$ explicitly as
 \ba z(t) &\= \frac\a{2\pi i}\,\log \frac QA
 \= \frac\a{2\pi i}\,\Bigl(\log t  \+ \frac{y_1(t)}{y(t)} \,-\,\log A \Bigr)  \\
  &\=  \frac\a{2\pi i}\,\bigl(\log \frac tA  \+ \tfrac{439-97\sq}{64}\,t \+ \tfrac{321913 - 77807\sq}{4096}\,t^2 \+\cdots\bigr)\,. \ea

Exactly the same considerations apply to the second differential operator~$\tL$, with the difference that here the mirror parameter 
$$ \tQ \= t\,e^{\ty_1/\ty}\= t + \tfrac{1575- 369\sq}{128}\,t^2 \+ \tfrac{4814915- 1166773\sq}{16384}\,t^3 \+\cdots$$
is related to the variable $z$ in the upper half-plane by 
$$  \tQ \= \tA\,\tq \qquad\text{with}\qquad \tq \= e^{2\pi i\p(z)/\a^\s}\,, $$
where $\,\p:\H\to\Hm$ is the twisting map and~$\tA$ is some constant.  A calculation like the one for~$A$ gives the numerical value
$$  \tA \;\approx\;-40.9565407890298922716044572957685\,,$$ 
which we can recognize as the ``conjugate-in-the-exponent'' of the value in~\eqref{eq:A}:
\be\label{eq:tA}  \tA \;\overset?=\; -2\,(3+\sq)\,\Bigl(\frac{5-\sq}{2}\Bigr)^{(-\sq-1)/4}\;. \ee
We will show later that also this formula is indeed correct. 

We can now calculate the Fourier expansions of both the twisted modular form
$\tf(z)=y_1(t)$ and the twisting map $\varphi:\HH\to\HH^-$ as
$$ \tf(z) \=  1 \+ \tfrac{23-\sq}8\,A\,q \+ \tfrac{1951- 473\sq}{256}\,A^2\,q^2 \+ \tfrac{184453- 44739\sq}{8192}\,A^3\,q^3 \+ \cdots $$
and
\ba \label{phiExp} \p(z) &\= \frac{\a^\s}{2\pi i}\,\log \frac\tQ\tA
 \= \frac{\a^\s}{2\pi i}\,\Bigl(\log t \+ \frac{\ty_1(t)}{\ty(t)\,-\,\log \tA}\Bigr)  \\
  &\=  \frac{\a^\s}{2\pi i}\,\bigl(\log \frac t\tA \+ \tfrac{1575- 369\sq}{128}\,t 
      \+ \tfrac{1208617-292799\sq}{8192}\,t^2 \+\cdots\bigr)\\
 &\= \tfrac{-9+\sq}8z + \tfrac{1-\sq}{4\pi i}\Bigl(\tfrac{-\sq}2 \log(\tfrac{5-\sq}2) \+ \tfrac{697- 175\sq}{128}\,A\,q \\
 & \qquad \+\tfrac{-29767+7249\sq}{8192}\,A^2\,q^2\+\tfrac{3091637-749587\sq}{393216}\,A^3\,q^3\+\cdots\Bigr)\,.  \ea
Again these Fourier expansions, unlike the expansions of the same functions as power series in $t(z)$, converge
exponentially rapidly for all~$z$ in the upper half-plane and can be used to compute the functions $\tf(z)$
and $\p(z)$ numerically and to verify the modular transformation properties~\eqref{eq:phiequivar} and 
$\tf\bigl(\frac{az+b}{cz+d}\bigr)=(c^\s \p(z)+d^\s )\,\tf(z)$ numerically to high accuracy, giving us the
first explicit example of a non-classical twisted modular form on a \Tei\ curve.

\subsection{Modularity and integrality} \label{sec:DEintegral}
At the end of~\S\ref{sec:BMDE} we stated that all the coefficients of the expansions of~$y$ and $\ty$
as power series in~$t$ belong to the ring $\O[\tfrac12]$.  This integrality has a rather puzzling aspect, 
which we discuss here and resolve in~\S\ref{sec:puzzles}.

If we write $y$ as $\sum c_nt^n$, then the differential equation $Ly=0$ translates into the recursion
$$  (n+1)^2c_{n+1} \,=\, \bigl(\b(n^2+n)+\tfrac34\gamma\bigr)\,c_n \,-\,\bigl(\b(n^2-\tfrac14)+\tfrac34\gamma\bigr)\,c_{n-1} 
  \+ (n-\tfrac12)^2c_{n-2}  $$
for the coefficients $c_n$, where $\b$ and $\g$ are given by~\eqref{defbetgam}.  The integrality (away from~2)
of the $c_n$ is far from automatic from this recursion, because at each stage one has to divide a linear combination 
of previous coefficients by $(n+1)^2$, so that {\it a priori} one would only expect $n!^2c_n$ to be 2-integral. 
Divisibility properties of this type are familiar from well-known recursions like the recursion
  \be\label{apery} (n+1)^2\,A_{n+1} \= (11n^2+11n+3)\,A_n - n^2\,A_{n-1} \ee
used by Ap\'ery in his famous proof of the irrationality of~$\zeta(2)$, or the similar one he used
in his even more famous proof of the irrationality of~$\zeta(3)$.  However, they are extremely rare.
For instance, in~\cite{Z09} it was found that of the first~100,000,000 members of the three-parameter family 
of recursions obtained by varying the coefficients ``11,'' ``3'' and ``$-1$'' in~\eqref{apery}, only 7 (if one 
excluded certain degenerate families, and up to scaling) had integral solutions.  

Ap\'ery proved the integrality of the solution of his recursion~\eqref{apery} by giving the
explicit closed formula $A_n=\sum_{k=0}^n\binom nk^2\binom{n+k}n$.  We do not know a corresponding expression
in our case.  However, soon after Ap\'ery's original proof, a more conceptual explanation was found by 
Beukers~\cite{Beuk87}, who saw that the differential equation corresponding to Ap\'ery's recursion has a modular
parametrization $y=f(z)\in M_1(\G)$, $t=t(z)\in M_0^{\text{mer}}(\G)$ of the type discussed in~\S\ref{sec:twistedDE},
the group~$\G$ in this case being $\G_1(5)$, and this implies the integrality because we have $f(z)\in\Z[[q]]$, $t(z)\in q+q^2\Z[[q]]$
and hence~$y\in\Z[[t]]$. Similar statements hold for all seven of the ``Ap\'ery-like'' equations mentioned above, 
leading to the conjecture (which was made explicitly in~\cite{Z09}) that the integrality property for recursions
of this type occurs {\it precisely} when the corresponding differential equation has a modular parametrization.  

The surprise is now this.  In our case, just as in the seven ``Ap\'ery-like'' ones, the differential equation
(at least for~$y$) {\it is} modular, and {\it does} have power series solutions with integral coefficients 
(away from~2), but now the modularity does not explain the integrality in the same way as above, because here
the relevant Fuchsian group is not arithmetic and the $q$-expansions of $t$ and $y$ are not integral.
Indeed, as we saw in~\S\ref{sec:DEmodular}, the coefficients of these $q$-expansions are not even algebraic
numbers, since they involve powers of the scaling constant~$A$, which according to~\eqref{eq:A} and Gelfond's
theorem is a transcendental number.  But even if we rescale by replacing $q$ by $Q=Aq$, then, although the first
few coefficients as listed in equations~\eqref{twrtQ} and~\eqref{fwrtQ} have denominators that are powers of~2, 
this property fails if we compute more coefficients.  For example, the coefficient of $Q^{11}$ in $f(z)$ equals 
$$\frac{16063132006911958155776129 - 3895881761337356780171815\sq}{2^{53}\cdot3^3\cdot 5\cdot7}\,,$$
and calculating further we find that the first~100 coefficients contain in their denominators all primes less than~100 
that do not split in~$\Q(\sq)$, and similarly for~$t(z)$. Thus, although our differential equation~\eqref{BMDE} does
not actually contradict the hypothetical statement
$$ \text{``integrality occurs only when the differential equation is modular"} $$
mentioned above, the mechanism
\bas \text{$y(t)$ modular}\;\, &\Rightarrow \;\, \text{$y$ and $t$ both have integral $q$-expansions}\\
   &\Rightarrow \;\,\text{$y$ has an integral $t$-expansion} \eas
which previously explained that statement now breaks down completely.  This puzzle, which was in fact the original
motivation for the investigation described in this paper, will be solved in Section~\ref{sec:eqfromdiffeq}, where 
we will provide a purely modular explanation of the integrality property by expanding~$y$ and~$\ty$ with respect 
to {\it both} $q$ and~$\tq$, using Hilbert modular forms rather than modular forms in one variable.

\subsection{The ring of twisted modular forms for $W$ and \Wcov.} \label{sec:Twisted}

We can now calculate the rings $M_{*,*}(\G,\p)$ and $M_{*,*}(\Pi,\p)$ of twisted modular
forms on the \Tei\ curve $W= \H/\G$ and its double cover $W_\Pi =\H/\Pi$. 
This information will be used in the following section to embed the curve $W_{17}^1$ 
into the Hilbert modular surface~$X_{17}$.
\par
We already know two twisted modular forms on~$\Pi$, namely $f(z)=y(t(z))$ in $M_{1,0}(\Pi)$
and $\tf(z)=\ty(t(z))$ in $M_{0,1}(\Pi)$.  (From now on we omit the ``$\p$".) Any holomorphic
or meromorphic twisted modular form of weight $(k,\l)$ on~$\Pi$  is then equal to $f^k\tf^\l$ 
times a rational function of~$t=t(z)$. The next proposition tells us which ones are holomorphic.

\begin{Prop} \label{Prop71} For $k,\,\l\ge0$ the vector space of twisted modular forms of weight $(k,\l)$ on~$\Pi$ is given by
\be \label{eq:Mkl} M_{(k,\l)}(\Pi) \= \bigl\langle \,f(z)^k \tf(z)^\l\,t(z)^c \mid 3k+\l\geq 2c \ge0 \,\bigr\rangle_\C\;.\ee
\end{Prop}
\par
\begin{proof}
The group $\Pi$ has no elliptic fixed points and five cusps, one of which is irregular, so
Theorem~\ref{thm:dimmodWD} implies the dimension formula
\be \label{eq:dimMkl} \dim M_{(k,\l)}(\Pi) = 1\+ \left\lfloor \frac{3k+\l}2\right\rfloor\;. \ee
Since the right-hand side of this equals the number of monomials $f^k\tf^lt^c$
in~\eqref{eq:Mkl}, it suffices to prove that each of these monomials is holomorphic or
equivalently, that $f(z)$ and $\tf(z)$ are holomorphic everywhere (including at the cusps)
and vanish to orders $1/2$ and $3/2$, respectively, at the cusp $t=\infty$, where the
order is measured with respect to the local parameter~$1/t$. 
The holomorphy at the cusps is a special case of Proposition~\ref{Prop5.6},
since the construction of the defining equation~\eqref{BMform} of~$W$ given in~\S\ref{UnFa} was
based on choosing the coefficients in such a way as to make the two differential forms
$\om=dX/Y$ and $\wt\om=X\,dX/Y$ stable at all of the cusps of~$W$.  We therefore only need to check that
the order of vanishing of~$f$ and~$\wt f$ at~$z=0$ (corresponding to~$t=\infty$) are at most, and hence
exactly, equal to 3/2 and 1/2, respectively.  We will give two arguments to see this.
\par
The first way is to use the action of the element~$S=M_i$ of~$\G$, which corresponds to the involution $t\mapsto 1/t$ on the base of
the family \eqref{BMform}. This involution extends via $X \mapsto X/t$ and $Y \mapsto Y/t^{5/2}$ to 
an involution $\iota$ of the whole family, with $\iota^* \omega(t) = t^{3/2} \omega(t)$
and $\iota^* \wt \omega(t) = t^{1/2} \wt \omega(t)$. Near~$t=0$, the section $f$
is the period of $\omega$ along the unique cycle (up to scale) $\beta_0$
that is orthogonal to $\wt \omega(t)$ and extends across $t=0$. Near $t=\infty$, it
is the period of $\omega$ along the unique cycle (up to scale) $\beta_\infty$
that is orthogonal to $\wt \omega(t)$ and extends across $t=\infty$.
From these defining properties it follows that $i^* \beta_0$ is proportional
to $\beta_\infty$. The same argument applies for $\wt f$.  In modular terms, 
this translates into the statement that the function~$f(z)$ transforms via 
\be\label{Sonf}   \frac1z\,f\Bigl(-\frac1z\Bigr) \= f(z)\,t(z)^{3/2}   \ee
for some appropriate choice of the square-root of~$t(z)^{3/2}$ (which we need
only make at one point since this function has no zeros or poles in~$\H$),
and similarly 
\be\label{Sontf}   \frac1{\varphi(z)}\,\tf\Bigl(-\frac1z\Bigr) \= \tf(z)\,t(z)^{1/2}\,.   \ee
Equations~\eqref{Sonf} and~\eqref{Sontf} clearly imply the statement that the modular
forms $f^2t^3$ and $\tf^2t$ at~$\infty$ are holomorphic everywhere, as claimed. 
\par
The other approach, not using the accidental fact that the cusps~0 and~$\infty$ of~$\HH/\Pi$ happen to be 
interchanged by an element in the normalizer of~$\Pi$ and therefore applicable in other situations, is based on the equation
\be\label{tderiv} \frac\a{2\pi i}\,t'(z) \= f(z)^2\,A(t(z))\,, \ee
where $A(t)$ is the 4th degree polynomial given in~\eqref{AB}, see~\eqref{eq:tprime}.  
Since the polynomial~$A(t)$ is divisible by~$t$, we find that $f(z)^2$ multiplied by a cubic polynomial
in~$t(z)$ is equal to the logarithmic derivative of the modular function~$t(z)$, and hence
is holomorphic at the cusp~$t=\infty$, so the order of~$f$ at~$t=\infty$ is $\le3/2$, as desired. 
A similar argument, this time using~\eqref{eq:tprimephi}, applies also to $\tf$, with the
polynomial~$A(t)$ replaced by the rational function~$\wt A(t)$, which grows like~$t^2$ as~$t\to\infty$.
 \end{proof}

{\bf Remarks. i)} We make some comments about the half-integer order of $f$ and~$\wt f$ at~$\infty$
and about the appearance of the function $\sqrt{t(z)}$ in~\eqref{Sonf} and~\eqref{Sontf}. 
By the discussion in~\S\ref{generators}, the image $\bar\Pi$ of~$\Pi$ in 
$\bar\G=\G/\{\pm1\}$ has index~2 and is hence normal, so~$S\bar\Pi S^{-1}=\bar\Pi$
and hence $F|_kS$ must be a modular form on~$\Pi$ for any modular form~$F$ (twisted
or not) of~{\it even} weight~$k$ on~$\Pi$. But the subgroup $\Pi$ of~$\G$ is not
normal and is not normalized by~$S$, so the intersection of~$\Pi$ with 
$\Pi'=S\bar\Pi S^{-1}$ is a proper subgroup (of index~2) in~$\G$. The space
of holomorphic modular forms of weight~1 on this group is 4-dimensional, spanned by
the functions $ft^{j/2}$ with $0\le j\le3$, with the spaces of modular forms of weight~1
on $\Pi$ and $\Pi'$ separately being spanned by $(f,ft)$ and by~$(ft^{1/2},ft^{3/2})$,
respectively.  Notice that the group $\Pi\cap\Pi'$ does not contain the stabilizer~$M_\infty$
of~$\infty$ in~$\G$ (or~$\Pi$), but only its square, so that it has width~2 and hence
a local uniformizer~$q^{1/2}$ at~$\infty$. This group has genus~0, with $t(z)^{1/2}$ as a hauptmodule.

{\bf ii)} In the above proof we gave an implicit estimate of the period integrals 
that define~$f$ and~$\tf$ in the neighborhood of any cusp.  These periods are given by 
integrating $\omega = dX/Y$ and $\wt \omega = XdX/Y$ over a linear combination of paths 
that are invariant under the local monodromy around the given cusp. It is perhaps worthwhile 
giving a more explicit proof in the special case at hand, since this makes the argument 
clearer and also shows how to find the full expansion, and not just the order of vanishing, 
of $f$ and~$\tf$ at every cusp. We will just give the main formulas, without complete 
details. At each cusp we choose a local parameter~$\ve=\ve_j$. Then we can find the 
expansions of~$y(t)=f(z)$ and $\ty(t)=\tf(z)$ near~$t=t_j$ by looking at the
explicit form of the degenerations of the differentials $\om$ and~$\wt\om$ there,
as explained at the end of the discussion in~\S\ref{UnFa}.   
We consider the cusp $t=\infty$ here and the other cusps even more briefly in~iii) below.
Near $t=\infty$ we make the substitutions $(X,t)=\bigl(\frac{cT^2-A}\ve,\,\frac1\ve\bigr)$, with $c=\frac{17-3\sq}2$ and 
$A=5\,\frac{5+\sq}2$ as in~\eqref{BMcoeffs}, and where $\ve$ tends to~0.  Then by direct computation we find
$$P_5(X,t) \= \bigl(c/\ve\bigr)^5\,\bigl[T\,\bigl(T^2-\la_1^2\bigr)\,\bigl(T^2-\la_2^2\bigr)\bigr]^2
  \+ \text O\bigl(\ve^{-4}\bigr)\,,$$
with $\la_1 = \tfrac{ 5+\sqrt{17}}2$ and $\la_2 = \tfrac{3+\sqrt{17}}4$. The fact that the 
leading coefficient of the right-hand side as a Laurent series in~$\ve$ is a
square corresponds to the degeneration of the fiber over the cusp to a rational curve, and lets us
compute the differential form $\om=dX/\sqrt{P_5(X,t)}$ as
$$ \om \= \frac{2\,\ve^{3/2}}{c^{3/2}}\,\Biggl[ \frac1{\bigl(T^2-\lambda_1^2\bigr)\,\bigl(T^2-\lambda_2^2\bigr)} 
 \;+\; \cdots\,\Biggr]\,dT\,,$$
where the omitted terms contain higher powers of $\ve$ with coefficients that are rational functions of $T$
having poles only at $\pm\la_1$, $\pm\la_2$ that can easily be found explicitly 
with a suitable mathematical software program.  The factor $\ve^{3/2}$ gives the vanishing order 
we claimed, and the  rest of the expansion gives us the complete expansion of $f(z)$ near $z=0$.  
Specifically, the homology of 
$\PP^1\smallsetminus\{\pm\la_1,\pm\la_2,\infty\}$ is spanned by the four small loops $\gamma_i$
around the four poles $\pm\la_1$ and $\pm\la_2$, and the integral of the above form around
each such loop is given simply by the residue of the form at that pole, so that we can easily 
calculate the periods around each $\gamma_i$ to any order in~$\ve$.  When one does this 
calculation, one finds that these integrals are given, up to a constant, by
$$ \int_{\gamma_1,\gamma_2,\gamma_3,\gamma_4}\omega(\ve) \;\doteq\;\bigl(1,-1,\tfrac{1+\sd}2,-\tfrac{1+\sd}2\bigr)\,
  \Bigl(1 \+ \tfrac{81-15\sd}{16}\ve\+\tfrac{4845-1155\sd}{64}\ve^2\+\cdots \Bigr)\,.$$
The surprising observation that they are all proportional is explained by the fact the loops $\gamma_i$ 
correspond to the core curves (in both directions) of parallel cylinders in the generating flat surface, here concretely
the vertical cylinders in Figure~\ref{cap:Ltable}~(a). 
More precisely, these curves stay parallel in a neighborhood of
the cusp of the \Teichmuller\ curve by definition of these curves, and hence the periods remain proportional.
Another striking property of the above expansion, namely the integrality (up to powers of two) of its
coefficients as a power series in the local parameter $\ve=1/t$ at the cusp $t=\infty$, is obvious both 
from the proof in~\cite{BM07} and from the one that we will give in Section~\ref{sec:eqfromdiffeq}. 
\par

{\bf iii)} We now also indicate briefly how to find the Fourier expansions of~$f$ and~$\wt f$ at a cusp~$t_j$ other 
than~$t=0$.  To define them, we must first choose a matrix $M_j$ mapping the point~$z_j$ with~$t(z_j)=t_j$
to~$\infty$. For definiteness's sake we choose $M_1=$Id and $M_j=\sm0{-1}1{-z_j}$~for $j\ne1$.
The {\it width} $w_j$ of the cusp is defined as the smallest positive number~$w$ with 
$M_j^{-1}\sm1w01M_j\in\G$, and we define $q_j=e^{2\pi iM_j(z)/w_j}$ and $Q_j=A_jq_j$, where~$A_j$
is chosen so that $Q_j=\ve+\text O(\ve^2)$ as~$\ve\to0$.  Notice that both the width~$w_j$ and
the value of the scaling constant~$A_j$ depend on our choices of~$\ve_j$ and~$M_j$; with the
choices given above, they are given as in Table~\ref{tab:cusps} for representatives of $\Gamma$-equivalence 
classes of cusps. At each cusp we expand the integral as a power series in the local parameter, observing that each coefficient
is the integral of a rational function on a punctured Riemann sphere, and proceed just as we did above for $t\to\infty$. 
We omit the calculations and give only the results (for $f$; those for~$\wt f$ can be obtained in the same way):
\bas
f|_1M_2 &\= \frac18 (3+\sq)^{5/2} (4+\sq)^{1/2}\,\,  \Bigl(1 \,-\, \tfrac{3}{4} Q_3 
\+ \tfrac{3807+915\sq}{128} Q_3^2 \+ \cdots \Bigr)\;, \\
f|_1M_3 &\= \frac{-i}{2^{29/2}} (3+\sq)^{11} (4+\sq)^{-4}\, \, \Bigl(1 \+ \tfrac{-255+1959\sq}
{1024}Q_4 \+ \cdots \Bigr)\;.
\eas
 
  \begin{table}  \begin{tabular}{|c|c|c|c|c|}
  \hline  $j$ & $t$ & $z$ & $w_j$ & $A_j$ \\     \hline
  1 & 0 & $\infty$ & $\alpha= \tfrac{1+\sq}{2}$ & $-2\,(3+\sq)\,
\Bigl(\tfrac{5-\sq}{2}\Bigr)^{\tfrac{\sq-1}{4}}$ \\
  2 & 1 & 1 &  $8 + 2 \sq$ & $4 \,(3+\sq)\, (4+\sq)^{\tfrac{5+\sq}{4}}$ \\
  3 & $\lambda^{-1}$ & $\alpha/2$ & $\tfrac{-3+\sq}{2}$ & $2^{11}\, (4+\sq)^9\, 
\Bigl(\tfrac{3+\sq}{2}\Bigr)^{-\tfrac{39+\sq}{2}}$ \\
  \hline   \end{tabular} \caption{Cusps} \label{tab:cusps}  \end{table}
\par
\medskip

This concludes our discussion of the modular forms and twisted modular forms on the group~$\Pi$.
Proposition~\ref{Prop71} also allows us to describe the twisted modular forms on the 
\Tei\ curve itself, i.e.\ for the group $\G$. Since this group contains $-I$, there are no 
twisted modular forms of weight $(k,\ell)$ for $k+\ell$ odd. Theorem~\ref{thm:dimmodgeneral} 
and the arguments in the preceding proof immediately imply the following statement.
\begin{Prop} \label{prop:dimGamma}
The ring of weight modular forms of even total weight $k +\l$ for~$\G$ has the graded dimensions
$$\dim M_{(k,\l)}(\G) \=  1 \+ \Bigl\lfloor \frac{3k+\l}{4} \Bigr\rfloor\;.$$
It consists of the invariants of $M_{(k,\l)}(\Pi)$ as given in~\eqref{eq:Mkl}
under the involution
$$ t \mapsto t^{-1}, \quad f \mapsto f t^{3/2}, \quad \wt f \mapsto \wt f t^{1/2}.$$
\end{Prop}
\par
As a corollary, we see that that the ring of {\it parallel weight} twisted modular forms on~$\G$ is freely generated by the two forms 
\be \label{xiandeta} \xi \,=\,(1-t)^2\,f\tf\,, \qquad \eta\,=\,\frac{45-11\sd}8\,\,t\,f\tf \,=\,-\frac{19+5\sd}4\,s\,\xi\,, \ee
with $s$ as in~\eqref{sandt}.  (The numerical factors in the definition of~$\eta$ have been included for later convenience.)
In particular, when we embed the \Teichmuller\ curve into the Hilbert modular surface, then
the restriction of a Hilbert modular form of weight~$(k,k)$  is a homogeneous polynomial of degree~$k$
in $\xi$ and~$\eta$, or equivalently equals~$\xi^k$ times a polynomial in~$s$ of degree at most~$k$.  
We will use this in the next section to find an explicit description of this embedding.

\section{The Hilbert modular embedding of $W_{17}^1$}
\label{sec:eqfromdiffeq}

We continue to study the example of the special \Tei\ curve $W=W_{17}^1$, using the information given in the
previous section to give a complete description of the embedding of~$W$ into the Hilbert modular surface~$X_{17}$.

\subsection{Modular embedding via Eisenstein series.} \label{sec:MEviaES}
We can use our explicit knowledge of the twisting map $\p:\H\to\Hm$ and of the twisted modular forms on~$\Pi$
and on~$\G$ to embed $\H/\Pi$ and $\H/\G$ into $X_{17}$ by expressing the restrictions
of Hilbert modular forms as polynomials in~$f,\,\tf$, and $t$ (or simply in~$\xi$ and~$\eta$
if we restrict our attention to forms of parallel weight).  We begin by recalling the main points of the
theory of Hilbert modular forms, using $\SL{\O}$ rather than $\SLO$ since it is more familiar, and
then later transform our results back to~$\SLO$ for~$D=17$, using the isomorphism of the two groups in this case. \changed{A textbook reference for the notions
and claims in this section in e.g.~\cite{vG87}.}
\par
A {\em Hilbert modular form} of weight $(k,\l)$ on the full modular group 
$\SL{\O_D}$ is a holomorphic function $F:\HH^2\to\CC$ satisfying the transformation 
law $F(\g z_1,\,\g^\s z_2)=(cz_1+d)^k(c^\s z_2+d^\s)^\l F(z_1,z_2)$ for $g\=\sm abcd\in\SL{\O_D}$. 
If $k=\l$ (``parallel weight"), we call the weight simply~$k$. We denote the space of such forms by 
$M_{k,\l}(\SL{\O_D})$, or simply $M_k(\SL{\O_D})$ if~$k=\l$, and more generally write $M_{k,\l}(\G,\chi)$ 
for the corresponding space for forms with respect to a subgroup $\G$ of $\SL{\O_D}$ and character~\changed{$\chi:\SL\O_D\to\C^*$.}
On $M_k(\SL{\O_D})$ we have the involution induced by the symmetry $\iota:(z_1,z_2)\mapsto(z_2,z_1)$ of~$\HH^2$, 
so we can split this space into the direct sum of the subspaces $M_k^\pm(\SL{\O_D})$ of symmetric and 
antisymmetric Hilbert modular forms.

A Hilbert modular form~$F$ on~$\SL{\O_D}$ has a Fourier expansion of the form
  \be\label{HMFExp} F(\zz)\= F(z_1,z_2) \= b_0 \+ \sum_{\nu\in\O_D^\vee,\,\;\nu\gg0} b_\nu \; \e\bigl(\tr(\nu\zz)\bigr) \ee
where $\O_D^\vee$ is the inverse discriminant and $\tr(\nu\zz)$ for $\zz=(z_1,z_2)\in\HH^2$ means $\nu z_1+\nu^\s z_2$.
If we choose a $\Z$-basis for $\O_D^\vee$ and write $X$ and $Y$ for the corresponding exponential functions
$ \e(\tr(\nu \zz))$, then the right-hand side of~\eqref{HMFExp} becomes a Laurent series in~$X$ and~$Y$
(or even a power series if the basis is chosen appropriately).  In practice it is sometimes more convenient
to choose only a $\Q$-basis for $\O_D^\vee$, in which case we work with power series with congruence
conditions on the exponents of~$X$ and~$Y$. The simplest choice is
\be\label{defXY} X \= X(\zz) \= \e\Bigl(\frac{z_1+z_2}2\Bigr)\,,\qquad Y \= Y(\zz) \= \e\Bigl(\frac{z_1-z_2}{2\sD}\Bigr)\,,\ee
in which case $\,\bigl\{\e(\tr(\nu\zz))\}=\{X^mY^n\mid m\equiv n\pmod2\}\,$ and~\eqref{HMFExp} becomes
\be\label{XYexp}  F(\zz) \= \sum_{m=0}^\infty\,\Biggl(\sum_{|n|\le|m|\sD\atop n\equiv m\!\!\!\pmod2 }
   b_{m,n}\,Y^n\Biggr)\,X^m\quad\in\;\C[Y,Y\i][[X]]\;. \ee
\par
For Hilbert modular forms the same remarks as in \S\ref{sec:HMGHMS} for Hilbert modular surfaces apply 
concerning the (less standard) Hilbert modular groups $\SLOA{\O^\vee}$ or functions on $\HH \times \HH^-$, 
the only difference in the latter case being that the condition $\nu\gg0$ in~\eqref{HMFExp} must
be replaced by~$\nu>0>\nu^\s$.

\par
The simplest Hilbert modular forms to construct are the Eisenstein series of weight~$(k,k)$ ($k=2,\,4,\,6,\dots$), 
with Fourier expansion given by\footnote{For general discriminants there are several Hilbert-Eisenstein series for 
each value of~$k$, with sum~$\GG_k^D$. If the class number of~$D$ is~one, as is the case for $D=17$, there is only one.}
  \be \GG_k^D(\zz)\= \frac{\z_K(1-k)}4 \+ \sum_{\nu\in\O_D^\vee,\,\;\nu\gg0} \s_{k-1}^K\bigl(\nu\sD\bigr) \, \e(\tr(\nu \zz))\,. \ee
Here $\z_K(s)$ denotes the Dedekind zeta function of $K=\Q(\sD)$ and $\s_{k-1}^K(\nu\sD)$ for $\nu\in\O_D^\vee$ with $\nu\gg0$ is given by
$$ \s_{k-1}^K\bigl(\nu\sD\bigr) \;:=\; \sum_{\fracb|\nu\sD} N(\fracb)^{k-1} 
  \= \sum_{d|\nu\sD} d^{k-1}\,\s_{k-1}\Bigl(\frac{D\nu\nu'}{d^2}\Bigr)\,,$$
where the first sum runs over integral ideals~$\fracb$ of $K$ dividing the integral ideal $\nu\sD$ and the 
second sum (whose equality with the first is proved in~\cite{Z76}, Lemma on p.~66) runs over positive 
integers~$d$ such that~$d^{-1}\nu\in\O_D^\vee$, and where $\s_{k-1}(m)$ for $m\in\N$ has its usual meaning as
the sum of the $(k-1)$st powers of the (positive) divisors of~$m$. 
In particular, for $D=17$ the first three Eisenstein series $\GG_k=\GG_k^{17}$ have Fourier expansions beginning 
  \bas  \GG_2(\zz) &\= \frac1{12} \+ \bigl(3Y^3+7Y+7Y^{-1}+3Y^{-3}\bigr)\,X  \+ \bigl(Y^8+21Y^6+14Y^4\\ &\qquad\quad
       +45Y^2+18+45Y^{-2}+14Y^{-4}+21Y^{-6}+Y^{-8}\bigr)\,X^2 \+ \cdots\,, \\
  \GG_4(\zz) &\= \frac{41}{120} \+ \bigl(9Y^3+73Y+73Y^{-1}+9Y^{-3}\bigr)\,X \+ \bigl(Y^8+657Y^6 \\ &\qquad \quad
   +2198Y^4+5265Y^2+4914+5265Y^{-2}+\cdots+Y^{-8}\bigr)\,X^2 \+ \cdots\,, \\ 
   \GG_6(\zz) &\= \frac{5791}{252} \+ \bigl(33Y^3+1057Y+1057Y^{-1}+33Y^{-3}\bigr)\,X \+ \bigl(Y^8+34881Y^6 \\ &\qquad \quad
     +371294Y^4 + 1116225Y^2 + 1419858+\cdots+Y^{-8}\bigr)\,X^2 \+ \cdots\,, \eas
with $X$ and $Y$ as in~\eqref{defXY}. As a check, we can verify that if we set $X=q$ 
and $Y=1$, \changed{corresponding to the natural embedding 
$z \mapsto (z,z)$ of $\HH/\SL\ZZ$
into $\HH^2/\SL\O$}, then these
Fourier expansions agree to the accuracy computed (several hundred terms) with those of the classical $\SL\Z$
modular forms $\frac1{12}E_4$, $ \frac{41}{120}E_4^2$, and $ \frac{44}3E_4^3 + \frac{2095}{252}E_6^2$, respectively.
\par
We now compute the restrictions of these Eisenstein series to the \Teichmuller\ curve $W=\HH/\G$ (or rather, of its
double covering $\HH/\Pi$) that we studied in~\S\ref{sec:BouwMoeller}.  As explained there, the 
algorithm used produced~$\G$ as a subgroup of~$\SL\O$, rather than~$\SLO$, so that the twisting function~$\p(z)$ 
maps~$\HH$ to~$\HH^-$ (cf.~eq.~\eqref{phiExp}).  Hence we must use the 
embedding of~$W$ into $\H^2/\SL\O$ by
\be \label{TeiEmb} z\quad\mapsto\quad(z_1,z_2)\,=\,\bigl(\ve z,\,\ve^\s\p(z)\bigr) \qquad(z\in\HH)\,, \ee
where $\ve=4+\sd$ is the fundamental unit of~$\Q(\sd)$.  Using the expansions of~$Q$ 
and~$\tQ$ as power series in~$t$ that 
were given in~\S\ref{sec:DEmodular}, we find that the monomial $X^mY^n=\e(\tr(\nu\zz))$ for 
$\nu=\frac m2+\frac n{2\sd}\in\O^\vee$ has the $t$-expansion 
\begin{equation} \label{eq:NewXYexp}
\begin{split} X^mY^n &\= \e\bigl(\ve\nu z\+\ve^\s\nu^\s\p(z)\bigr) 
  \= (Q/A)^{\ve\a\nu}\,(\tQ/\tA)^{(\ve\a\nu)^\s}  \\
 &\= \bigl(-2(3+\sd)\bigr)^{-\frac{21m+5n}2}\,\biggl(\frac{5-\sd}2\biggr)^{-8m-2n}\, t^{\frac{21m+5n}2} \\
  &\;\quad \times\biggl(1\+\frac{16597m+3985n-(3827m+919n)\sd}{128}\,t\+\cdots\biggr)\,.
\end{split} 
\end{equation}
(More precisely, this is true under the assumption that the constants~$A$ and $\tA$ relating $q=\e(z/\a)$ and $\tq=\e(\p(z)/\a^\s)$ 
to $Q=t+\cdots$ and $\tQ=t+\cdots$ indeed have the values given in~\eqref{eq:A} and~\eqref{eq:tA}; we will return to this point below.)
Inserting the expansion~\eqref{eq:NewXYexp} into the Fourier development of~$\GG_k(\zz)$ or any other Hilbert modular form, we can compute its 
restriction to~$W$ as a power series in~$t$.  On the other hand, as we saw above, the restriction of any $F\in M_k(SL(2,\O))$ to~$W$
is a linear combination of monomials $\xi^i\eta^j$ with $i+j=k$, where ~$\xi$ and~$\eta$ are the functions defined 
in~\eqref{xiandeta}, whose expansions in~$t$ are known, so we can find the coefficients of this linear combination by 
linear algebra.  The result of the computation for the first three Eisenstein series is 
\[ \begin{split}
& 12\,\GG_2\bigr|_{W}  \= \xi^2 \,-\, \frac{11+\sq}4\,\xi\eta \+ \eta^2\;, \\
& 120\,\GG_4\bigr|_{W} \=41\,\xi^4 \,-\, \frac{1855+365\sd}2\,\xi^3\eta \+ \frac{18245+3979\sd}8\,\xi^2\eta^2 \\
& \qquad\qquad\qquad \+ \frac{151+35\sd}4\,\xi\eta^3 \+ 41\,\eta^4 \;,\\
& 252\,\GG_6\bigr|_{W}  \=  5791\,\xi^6 \,-\, \frac{867831+173541\sd}4\,\xi^5\eta \+ \frac{11350461+2429643\sd}{16}\,\xi^4\eta^2 \\
&  \qquad\qquad\qquad \+ \frac{1883335+652933\sd}{16}\,\xi^3\eta^3 \+  \frac{69270195+16881483\sd}{64}\xi^2\eta^4 \\
&  \qquad\qquad\qquad  \+ \frac{1983525+452397\sd}{32}\xi\eta^5 \+ 5791\,\eta^6 \;. 
\end{split} \]
(It was to simplify the coefficients in these polynomials that we introduced the factor~$c$ in~\eqref{xiandeta}. The
coefficient of $\xi^k$ here is just the constant term of~$\GG_k$, and in particular rational, and the 
coefficient of~$\eta^k$ has the same value because the matrix~$\sm1101$ belongs to $\SL\O$, though not to~$\G$, 
so that the constant terms at the two cusps $s(0)=0$ and $s(1)=\infty$ of the restriction of any Hilbert modular form
to~$W$ are the same up to scaling.)  Now by elimination we find a polynomial
\[ \begin{split} P(\GG_2,\,\GG_4,\,\GG_6) &\= \bigl(3465994203567 - 840620808790\,\sq\bigr)\,\GG_6^4 \\
  & \qquad +\; \cdots\;  \+ \bigl(7395484320944244318526129490625 \\
    &  \qquad -\,1711627845603248913114298550625\,\sq\bigr)\,\GG_2^{12} \end{split} \]
(in which we have omitted seventeen equally gigantic intermediate terms) whose restriction to $W$ vanishes.  
We have thus obtained an explicit algebraic equation cutting out the \Teichmuller\ curve $W$ on the Hilbert
modular surface $X_{17}$, but it is too big in the sense that its  vanishing locus is reducible and~$W$ is only
one of its components.  Indeed, we know from the results of Bainbridge~\cite{Ba07}
that there must be a symmetric Hilbert modular form of weight~12 which vanishes precisely
on the \Teichmuller\ curve and its image under the involution $\iota:(z_1,z_2)\mapsto(z_2,z_1)$, whereas
the above equation has weight~24.  The reason for this is twofold:
\begin{itemize}
\item the Eisenstein series generate only a subring of the full ring of symmetric Hilbert modular forms, and there
is no reason that the minimal defining equation of~$W$ should belong to this subring; and
\item we did not even use all the Eisenstein series, but only $\GG_2$, $\GG_4$ and $\GG_6$.
\end{itemize}
The second point can be dealt with by expressing the restriction of each $\GG_k$ to~$W$ as a polynomial in $\xi$ and $\eta$
and looking for the first weight in which some linear combination of monomials in these restricted Eisenstein series vanishes.
This weight, however, still turns out to be~14 rather than~12, and the answer is not even unique: there is a 2-dimensional
space of linear combinations of $\GG_2^7,\dots,\GG_{14}$ that vanish on~$W$, all having huge coefficients so that we do not
reproduce them here.  (As a side remark, it is actually surprising that there should be even one relation in such a low weight,
let alone two, since there are 15 monomials of weight~14 in $\G_2,\dots,\G_{14}$ and also 15 monomials of weight~14 in $\xi$ and~$\eta$,
so that one would not expect the former to lie in a non-trivial subspace, let alone a subspace of codimension~2. The first weight
in which there has to be a relation for dimensional reasons is~16.  Similarly, it is surprising that we found a relation among $\GG_2$,
$\GG_4$ and $\GG_6$ in as low a weight as~24, since {\it a priori} the first weight in which the number of monomials in these three
forms is larger than the number of monomials in~$\xi$ and~$\eta$ is~38.  This suggests that the restrictions of the Eisenstein series
to~$W$ have some non-generic property, but we do not know what it is.)  To address the first point, we need to have a full set of
generators of the ring of Hilbert modular forms.  Finding such a set of generators in general is a comparatively difficult (though always
algorithmically solvable) problem, but in the case $D=17$ the result has been obtained by Hermann~\cite{He81}.  We will describe
his results in the next subsection and use them to determine the symmetric Hilbert modular form of minimal weight~12 vanishing on~$W$. 

 Another pertinent remark is that, although we have so far only used Hilbert modular forms of ``parallel'' weight $(k,k)$,
whose restriction to~$W$ is a polynomial in~$\xi$ and~$\eta$, one can equally well consider Hilbert modular
forms of mixed weight~$(k,\l)$ with $k\ne\l$, in which case the restrictions become polynomials in~$f$, $\tf$ and 
$t^{\pm1}$. For example, the Rankin-Cohen bracket $2E_2'(\zz)E_4(\zz)-E_2(\zz)E_4'(\zz)$, where the prime denotes differentiation
with respect to the first variable~$z_1$, is a Hilbert modular form of weight~(8,6), and its restriction to $W$ could be computed 
explicitly as $f(z)^8\tf(z)^6$ times a Laurent polynomial in~$t(z)$. 
Such mixed weight forms will play a role in Part~III of this paper, e.g.~in Section~\ref{sec:ThetaDer}, where
we will use a different construction to find explicitly a non-symmetric Hilbert modular form of non-parallel
weight $(3,\,9)$ which vanishes precisely on $W$, again in accordance with the general results of Bainbridge.

\subsection{Hilbert modular forms for the discriminant 17.}
\label{sec:HMFdisc17}
We continue to work with the Hilbert modular group $\G_D=\SL\O$ for the case $D=17$, i.e., for
$\O = \ZZ[\a]$ with $\a=(1+\sq)/2$.  We will describe the structure of the ring of
symmetric Hilbert modular forms of parallel even weight, following Hermann~\cite{He81}, who obtains these 
modular forms by restriction of Siegel modular forms of genus~$2$.  Later we will look also at non-symmetric
 Hilbert modular forms and Hilbert modular forms of odd or non-parallel weight on~$\G_{17}$.

Hermann begins with the sixteen genus 2 Siegel theta series
\be \label{SiegelTheta} \Theta_{m,m'}(Z,v) \= \sum_{x \in \ZZ^2 + m} \e\Bigl(\tfrac12 xZx^t + x(v+m')^T \Bigr)\,. \ee
Here $m,\,m'\in\{0,\tfrac 12\}^2$ and the independent variables $Z$ and~$v$ are in the Siegel half-space $\HH_2$
and in~$\CC^2$, respectively.  Ten of these (those for which the theta characteristic $(m,m')$ is even,
i.e. $4 m\cdot m'\equiv0\pmod2$) are even functions of~$v$ and hence give Siegel modular forms of weight~$\tfrac 12$
after restricting to~$v=0$.  (The other six are odd and hence give~$0$ on restriction, but their derivatives
with respect to~$v$ give non-trivial restrictions that will play a crucial rule in the constructions of
Part~III of this paper.)  Using a modular embedding from $\HH^2$ to $\HH_2$ like the one described in the
previous sub-section, we get ten Hilbert theta series, all of weight~$\tfrac 12$ 
with respect to a suitable subgroup of~$\G_D$.

It is convenient to re-index the sixteen theta characteristics in a way that makes the action of $\G_D$ more transparent. 
This part works for any $D\equiv1\;(8)$, i.e., for $D$ such that the prime~$2$ splits as $\pi_2\pi_2^\s$ in 
$K=\QQ(\sqrt{D})$ for some prime ideal $\pi_2\ne\pi_2^\s$ (in our case, the principal ideal generated by $1+\a$).
Define sets $S$ and $\widehat S$ by   
$$ S \=   \{0,1,\infty\}\,, \qquad \widehat S \= S\cup \{X\} $$ 
where $X$ is a symbol, and let $\G_D$ act on $\widehat S\times\widehat S$ by fixing $X$ and identifying $S \times S$ 
with $\P^1(\O/\changed{\pi}_2) \times \P^1(\O/\s(\changed{\pi}_2))$. 
We match the usual indexing by characteristic with these symbols by
\begin{table}[h]
\begin{tabular}{|c|cccc|}
\hline
$m\,\bigl\backslash \,m' $ & X & 0& 1 & $\infty$ \\
\hline
X & $1111$ & $1011$ & $1010$ & $1110$ \\
0 & $0111$ & $0011$ & $0010$ & $0110$ \\
1 & $0101$ & $0001$ & $0000$ & $0100$ \\
$\infty$ & $1101$ & $1001$ & $1000$ & $1100$ \\
\hline
\end{tabular}\medskip
\caption{Reindexing of theta characteristics} \label{tab:reindex}
\end{table}
Table~\ref{tab:reindex}, in which the even characteristics correspond to the pairs $(a,b)\in{\widehat S}^2$
in which $a$ and $b$ are either both equal to or both different from $X$, and the odd characteristics to 
the pairs $(a,b)$ in which exactly one of $a$ and $b$ is equal to~$X$.  This gives us by restriction 10 
Hilbert modular forms $\th_{X,X}$ and $\th_{a,b}$ ($a,\,b\in S$) of weight~1/2.

Set $\Th=\th_{X,X}^2$ and for each permutation~$\pi$ of the set~$S$ set $\eta_\pi=\pm\prod_{s\in S}\th_{s,\pi(s)}$, where the sign
is~$+1$ if $\pi$ is the identity and~$-1$ otherwise.  Up to powers of~$i$, the Hilbert modular group preserves~$\Th$
and permutes the~$\eta_\pi$, preserving the parity of~$\pi$, so that if we set
$$  \eta_\pm \= \sum_{\text{$\pi$ even}} \eta_\pi \;\pm\; \sum_{\text{$\pi$ odd}} \eta_\pi  $$
then $\Th$, $\eta_+$ and $\eta_-$ are Hilbert modular forms, with multiplier systems, for the full Hilbert modular group.  
More precisely, we have
\be \label{hermann}  \Theta \in M_1(\G_D,v_0), \;\quad \eta_\pm^2 \in M_3(\G_D, v_0), \;\quad \eta_+\eta_- \in M_3(\G_D,v_0^{-1})\,, \ee
where 
\be \label{eq:defv0}
v_0:\G_D\to\mu_4, \quad \sm0{-1}10\mapsto -1, \quad \sm1x01\mapsto i^{\tr(x)} \,\, \text{for}\;\, x\in\O
\ee
is a character of order~4. The form~$\th_{X,X}$ is antisymmetric with respect 
to the involution~$\iota$ and hence vanishes on the diagonal 
 $\HH/\SL\ZZ\subset\HH^2/\SL{\O_D}$.
Moreover, in the case $D=17$ this is its full vanishing locus, so that any Hilbert modular form vanishing on the diagonal is
divisible by~$\th_{X,X}$ and any symmetric Hilbert modular form vanishing on the diagonal is divisible by~$\Th$. For example, 
since the restrictions of both $\eta_+^2$ and $\eta_-^2$ are proportional to~$\sqrt\D$ (where we use $E_4$, $E_6$ and
$\D=(E_4^3-E_6^2)/1728$ to denote the standard generators of $M_*(\SL\Z)$), some linear combination of them, which
turns out to be $\eta_-^2-4\eta_+^2$, vanishes on the diagonal and hence is divisible by~$\Th$.  This gives us the
following five symmetric Hilbert modular forms of even weight and trivial character:
\bes  G_2 = \frac{\eta_-^2 - 4\eta_+^2}{\Theta}, \quad G_4 = \eta_+ \eta_- \Theta, \quad H_4 = \Theta^4, 
\quad G_6 = \eta_-^2 \Theta^3, \quad H_6 = \eta_-^3 \eta_+ \,,\ees
where the index of each form indicates its weight. 

\begin{Thm}[Hermann, \cite{He81}] \label{thm:hermann}
The ring $M^+_{\rm ev}(\G_{17})=\bigoplus_{k \geq 0} M^+_{2k}(\G_{17})$ of symmetric Hilbert modular forms of even 
weight for $D=17$ is generated by the five Hilbert modular forms $G_2$, $G_4$, $H_4$, $G_6$ and $H_6$, with the relations
$$   G_4G_6 \= H_4H_6\,, \;\quad G_4^3 \= \frac14 G_6(H_6 + G_2G_4)\,,\;\quad  G_6^2 \= H_4\,(4G_4^2-G_2G_6)\,. $$
In particular, $M^+_{\rm ev}(\G_{17})$ is a free module of rank~$4$ over the algebra $\CC[G_2,H_4,H_6]$, with basis~$\{1,G_4,G_6,G_4^2\}$.
\end{Thm}
\begin{proof} [Sketch of proof (following Hermann)] The relations among Hermann's five forms, like any relations among modular forms,
can be verified algorithmically by looking at a finite part of the Fourier expansions of the functions involved, so
we only have to show that these forms generate the whole ring.  Let $F=F_0\in M_k(\G_{17})$ be a symmetric Hilbert modular form of
even weight~$k$ with trivial character.  Then the restriction of $F$ to the diagonal has weight~$2k$ divisible by~4, and
since the ring of modular forms on $\SL\Z$ of weight divisible by~4 is generated by the forms~$E_4$ and~$\D$, which are
multiples of the restrictions to the diagonal of~$G_2$ and $H_6$, there is a weighted homogeneous polynomial $P_0(G_2,H_6)$
in $G_2$ and $H_6$ whose restriction to the diagonal coincides with that of~$F_0$. By what we said above, we then have
$F_0=P_0(G_2,H_6)+\Th F_1$ for some Hilbert modular form $F_1\in M_{k-1}(\G_{17},v_0^{-1})$. The restriction of~$F_1$ to the
diagonal has the character of~$\sqrt\D$ and weight congruent to~2 modulo~4, so by the same argument as before coincides
with the restriction of~$\eta_+\eta_-P_1(G_2,H_6)$ for some weighted homogeneous polynomial $P_1(G_2,H_6)$. This implies
in turn $F_1=\eta_+\eta_-P_1(G_2,H_6)+\Th F_2$ for some $F_2\in M_{k-2}(\G_{17},v_0^2)$. A similar argument shows that
$F_2$ has the same restriction to the diagonal as $\eta_+^2\eta_-^2P_2(G_2,H_6)$ for yet a third polynomial~$P_2$,
so $F_2=\eta_+^2\eta_-^2P_2(G_2,H_6)+\Th F_3$ for some $F_3\in M_{k-3}(\G_{17},v_0)$, and a final iteration gives a
fourth polynomial $P_3$ such that $F_3=\eta_-^2P_3(G_2,H_6)+\Th F_4$ for some Hilbert modular form~$F_4$ of weight~$k-4$,
now again with trivial character.  Combining these successive identities we have written $F$ as
$P_0+G_4P_1+G_4^2P_2+G_6P_3+H_4F_4$ where each $P_i$ belongs to $\C[G_2,H_6]$ and $F_4\in M_{k-4}^+(\G_{17})$,
and now iterating the whole argument we see that $F$ is a linear combination of~1, $G_4$, $G_4^2$ and $G_6$ with coefficients 
in ~$\C[G_2,H_4,H_6]$ as claimed.
\end{proof}
\noindent{\bf Example.} The Fourier expansions of $\eta_+$,  $\eta_-$ and $\Th$ begin 
\[
\begin{split} 
 \eta_+ &\= -4\,X^{1/4}\,\bigl((Y+Y^{-1}) \+ (13Y^4-19Y^2-19Y^{-2}+13Y^{-4})\,X \+\cdots\bigr)\,, \\
 \eta_- &\= 16\,X^{1/4}\,\bigl(1\,-\,(Y^5+3Y^3-Y-Y^{-1}+3Y^{-3}+Y^{-5})\,X  \+\cdots\bigr)\,, \\
 \Th   &\= 4\,X^{1/2}\,\bigl((Y^2-2+Y^{-2})\,-\,2(Y^5-Y-Y^{-1}+Y^{-5})\,X \+ \cdots\bigr)\,.
\end{split}
\]
(As a check, if we set~$Y=1$ and $X=q$ then these reduce to~$-8\eta^6$, $16\eta^6$, and~0.)
Comparing with the expansions of the first three Eisenstein series $\GG_k$ given above, we find
that these forms are given in terms of Hermann's generators of $M^+_{\rm ev}$ by
\[
\begin{split}
 -192\,\GG_2 &\= G_2\,, \qquad  640\,\GG_4 \= \frac{41}{48}\,G_2^2\,-\,39\,G_4 \,-\, 57\,H_4\,, \\
  14336\,\GG_6 &= -\frac{5791}{72}\,G_2^3 \+ 8571\,G_4 G_2 \+ 11463\,H_4 G_2 \,-\, \frac{32865}4\,G_6 \,-\,6285\,H_6\,.
\end{split}
\]

\begin{Cor} \label{HermCor} The function field field of the symmetric Hilbert modular surface with $D=17$ 
is the rational function field is generated by the two functions
  \be\label{DefUV}  U \= \frac{H_4}{G_4} \quad \biggl(\= \frac{\Th^3}{\eta_-\eta_+}\biggr)\,,\qquad 
  V \= \frac{H_6}{H_6 -G_2G_4}\quad \biggl(\= \frac{\eta_-^2}{4\eta_+^2}\biggr)\,. \ee
\end{Cor}
 \begin{proof} The relations among Hermann's generators imply
 $$ \Bigl(\frac{G_4}{G_2^2}\,,\;\frac{H_4}{G_2^2}\,,\;\frac{G_6}{G_2^3}\,,\;\frac{H_6}{G_2^3}\Bigr) 
  \= \Bigl( \frac{UV}{4(V-1)^2}\,,\;\frac{U^2V}{4(V-1)^2}\,,\;\frac{U^2V^2}{4(V-1)^3}\,,\;\frac{UV^2}{4(V-1)^3}\Bigr)\,.$$
 so the corollary follows immediately from the theorem.
\end{proof}

\subsection{The equation of the \Teichmuller\ curve.}  \label{sec:puzzles}
The corollary just given tells us that the Hilbert modular surface
for~$D=17$ is rational, with coordinates~$U$ and~$V$.  In particular, the image of the \Teichmuller\ curves 
$W_{17}^\pm$ on this surface must be given by polynomial equations in these coordinates.  In this subsection we will give 
these equations, which turn out to be several orders of magnitude simpler than the previously obtained equation 
$P(\GG_2,\GG_4,\GG_6)=0$.  We will also describe better systems of Fourier coordinates and will resolve two questions 
that we raised earlier by showing that the values of~$A$ and $\tA$ given in~\eqref{eq:A} and~\eqref{eq:tA} on the basis of 
numerical computations are indeed correct and by giving a purely modular proof of the integrality (away from the prime~2) of the 
Taylor expansions of~$y$ and~$\ty$ as power series in~$t$. 

\begin{Thm} \label{eq:TeicheqUV}
On the (rational) symmetric Hilbert modular surface with coordinates
$U$ and $V$, the \Teichmuller\ curves $W_{17}^1$ and $W_{17}^0$ are given by the equations
\be\label{Eq1} W_D^1: \quad  V \+ \frac{5+\sq}2\,U^2 \+ 3\,\frac{7+\sq}8\,U \+ \frac{1-\sq}8 \= 0 \ee
and 
\be\label{Eq2} W_D^0: \quad  V \+ \frac{5-\sq}2\,U^2 \+ 3\,\frac{7-\sq}8\,U \+ \frac{1+\sq}8 \= 0\;. \ee
\end{Thm}
\begin{proof} 
Since we gave the expressions for the Hilbert-Eisenstein series~$\GG_2$, $\GG_4$ and $\GG_6$ in the Hermann generators in the last 
subsection, we could derive~\eqref{Eq1} from the results of~\S\ref{sec:MEviaES} giving the restrictions to the curve~$W_{17}^1$ of 
these Eisenstein series.  However, it is much simpler to work directly with Hermann's generators, obtaining their Fourier expansions 
from those of~$\eta_\pm$ and~$\Theta$ as given above and then using~\eqref{eq:NewXYexp} to obtain the $t$-expansions of their 
restrictions to~$W$.  The results of the computations are given in the following table, in which we have used the results from the end 
of~\S\ref{sec:Twisted} to write the expansion of the restriction of each of $G_4$, $H_4$, $G_6$ and $H_6$ to~$W$ (that of $G_2=-192\GG_2$ 
was already given above) as a power of~$\xi$ times a polynomial in~$s$:
\bes \begin{tabular}{|c||c|c|c|c|} \hline $F$ & $\ve G_4/2^4\pi_2^8$ & $\ve H_4/2^4\pi_2^6$ & $\ve^2 G_6/2^6\pi_2^{12}$ & $\ve^2 H_6/2^6\pi_2^{14}$  \\
\hline $\xi^{-k}\,F|W$ & $s(s-1)(s-\kappa_1)$ & $s(s-1)$  & $s^2(s-1)^2$  & $s^2(s-1)^2(s-\kappa_1)$   \\
\hline \end{tabular} \ees
\smallskip\noindent  
Here $\kappa_1= \pi_2^{-2} = \frac{13-3\sd}8$.  From the definitions of the Hilbert modular functions $U$ and $V$ it then follows 
that their restrictions to~$W$ are given by
\be \label{UVonW} U\bigr|_W \ = \frac{\kappa_1}{s\,-\,\kappa_1}\,, \qquad V\bigr|_W \= \frac{-1+\sd}8\,\frac{s(s-1)}{(s-\kappa_1)^2}\,,  \ee
and equation~\eqref{Eq1} follows immediately.  Equation~\eqref{Eq2} is proved in a similar way using the Galois 
conjugate differential equation, as discussed in the remark at the end of \S\ref{sec:GMandPF}; the resulting expansions
are the Galois conjugates of those for~$W_{17}^1$ and since the Hermann generators have rational Fourier coefficients the
final equation is necessarily also the Galois conjugate of that of~$W_{17}^1$.  In fact, this Galois conjugation property
holds for all~$D$, as was already recalled in Theorem~\ref{thm:topoTeich}\,(iv) of~\S\ref{sec:topoTeich}. 
\end{proof}
 
We end this section by discussing four points related to the equations given in Theorem~\ref{eq:TeicheqUV}.

{\bf 1.} Equation~\eqref{Eq1} describes a Hilbert modular function that vanishes precisely on the curve~$W_{17}^1$. We can 
also ask for the holomorphic Hilbert modular form of smallest weight with the same property. If we 
multiply the left-hand
side of~\eqref{Eq1} through by $\eta_+^2\eta_-^2$, then by equation~\eqref{DefUV} the result is
\be \label{eq:F171}
F_{17}^1 \=  \frac14\,\eta_-^4 \+ \frac{5+\sq}2\,\Th^6 \+ 3\,\frac{7+\sq}8\,\Th^3\eta_+\eta_- \+ \frac{1-\sq}8\,\eta_+^2\eta_-^2\,, 
\ee
and by Theorem~\ref{eq:TeicheqUV} this vanishes precisely on $W_{17}^1$.
According to~\eqref{hermann}, $F_{17}^1$ is a holomorphic Hilbert modular form of weight~6 on the full Hilbert modular group, but 
with quadratic character~$v_0^2$, where $v_0$ is defined as in~\eqref{eq:defv0}.  If we further multiply 
it by~$\Th^2$, then
we get a Hilbert modular form of weight~8 on the full Hilbert modular group and with trivial character, given in terms of 
the Hermann generators by
$$ \Th^2F_{17}^1\= \frac14\,G_2G_6 \+ \frac{5+\sq}2\,H_4^2 \+ 3\,\frac{7+\sq}8\,G_4H_4 \+ \frac{9-\sq}8\,G_4^2\,, $$ 
but this form now vanishes not only on~$W_{17}^1$, but also (doubly) on the diagonally embedded modular curve $\H/\SL\Z\subset\H^2/\SL\O$.
On the other hand, if we multiply $F_{17}^1$ by its Galois conjugate~$F_{17}^0=(F_{17}^1)^\s$,
then the product $F_{17}$ vanishes precisely on the full \Teichmuller\ locus $W_{17}=W_{17}^0\cup W_{17}^1$, 
and this is now a holomorphic Hilbert modular form on the full modular group~$\SL\O$ and with trivial 
character, given in terms of the basis of $M_{12}(\SL\O)$ from Theorem~\ref{thm:hermann} by 
$$ F_{17} \= \text{(explicit polynomial in $G_2,\,G_4,\,H_4,\,G_6,\,H_6$ with rational coefficients)}\,. $$
We do not write out the polynomial, since it is a bit complicated, but observe that it involves only 11 of 
the 16 generators of $M_{12}(\SL\O)$.  The fact that here there is a single Hilbert modular form of 
weight~12 whose vanishing locus is precisely the union of the \Teichmuller\ curves on~$X_D$ is a special 
case of the theorem of Bainbridge,
already mentioned in~\S\ref{sec:MEviaES}, stating that such a form $F_D$ exists for every~$D$.  
We will give a different proof of this in Part~III by constructing~$F_D\in M_{12}(\SLO)$ in general as a product of twelve derivatives of theta series 
of weight $(\h,\frac32)$ or~$(\frac32,\h)$.
\par
{\bf 2.} The next point concerns the choice of coordinates for our Fourier expansions. We replace 
the previously used coordinates $X$ and~$Y$ from~\eqref{defXY} by the new Fourier variables 
$$X_1\= X^{-3}Y^{13}, \quad Y_1\= X^5Y^{-21}\,.$$ 
This has several advantages. First of all they form a $\Z$-basis of the group of Fourier 
monomials $\e(\tr(x\zz))$, whereas $X$ and~$Y$ generated a subgroup of index~2. Secondly, symmetric
Hilbert modular forms of even weight are symmetric in~$X_1$ and $Y_1$, as one sees using 
the action of~$\ve\,$. For instance the $(X_1,Y_1)$-expansions of the first two Hermann generators
begin
\bas
-\tfrac1{192}\, G_2 & \= \tfrac1{12} \+ Y_1X_1 \+ (9Y_1^2 + 3 Y_1^3)\,X_1^2
\+ (3Y_1^2 + 10Y_1^3+15 Y_1^4)\,X_1^3 \+ \cdots\,, \\ 
\tfrac1{256}\, G_4 & \= \phantom{\tfrac1{12}} \,-\, Y_1X_1 \+ (-9Y_1^2 + Y_1^3)\,X_1^2
\+ (Y_1^2 -14 Y_1^3 + Y_1^4)\,X_1^3 \+ \cdots\,, \\ 
\eas
in which the coefficients of $Y_1^3X_1^2$ and $Y_1^2X_1^3$ are the same.
Thirdly, and most important, both are holomorphic near the cusp of~$W$ 
and hence have power series expansions in~$t$, with valuations~1 and~0 there rather than $21/2$ 
and~$5/2$ as for~$X$ and~$Y\,$. This has to do with the  Hirzebruch resolution of the
cusp singularities, according the which different $\Z$-bases of the group just mentioned
are good coordinates in different parts of the resolution cycle. (We will discuss this in 
much more detail in Section~\ref{sec:HirzBain}.) Here we have to choose the coordinates
that are adapted to the point of the cusp resolution through which $W$ passes. 
Explicitly, these expansions are
  \ba \label{eq:X1Y1}
X_1 &\= X^{-3}Y^{13} \= -\tfrac{11+3\sd}{64}\,t \+ \tfrac{403-229\sd}{2048}\,t^2 \+\cdots\,, \\
  \qquad Y_1 &\= X^5Y^{-21} \= \tfrac{21-5\sd}2  - \tfrac{895-217\sd}8\,t \+ \cdots\,. \ea
(For comparison, the leading terms of $X$ and $Y$ are $\sqrt{\frac{524445-220267\sd}{2^{80}}}t^{21/2}$ 
and $\sqrt{-\frac{3+11\sd}{2^{20}}} t^{5/2}$, respectively.) It is then very easy to restrict 
a Hilbert modular form~$F$ given in its $(X_1,Y_1)$-expansion to~$W$: to get the expansion up 
to order~$t^n$ it suffices to expand~$F$ up to $X_1^n$ in $\C[Y_1][[X_1]]$.
\par
{\bf 3.} The third point concerns the correctness of the values of~$A$ and $\wt{A}$ in~\eqref{eq:A} 
and~\eqref{eq:tA}. We have been assuming throughout (for instance in~\eqref{eq:X1Y1})
that these guesses were correct, and now finally prove it. We write
\ba \label{eq:X1Y1orig}
X_1 & \= -\tfrac{11+3\sd}{64} \,c_1\, t\,\exp\Bigl(\ve\a\lambda \frac{y_1}{y}
+ (\ve\a\lambda )^\sigma \,\frac{\ty_1}{\ty}\Bigr)\,, \\
Y_1 & \= \phantom{-}\tfrac{21-5\sd}{2} \,c_2\, t\,\exp\Bigl(\ve\a \mu\, \frac{y_1}{y}
+ \bigl(\ve\a \mu\bigr)^\sigma \,\frac{\ty_1}{\ty}\Bigr)\,, \\
\ea
where $\lambda = \tfrac{13-3\sd}{2\sd}$ and $\mu = \tfrac{-21+5\sd}{2\sd}$ 
(the factor $\ve$ in the exponent takes care of the identification of $\HH\times\HH$ 
and $\HH\times\HH^-$, the factor~$\a$ is the cusp width, and the factors~$\la$ and $\mu$
come from the passage from $(X,Y)$ to $(X_1,Y_1)$) and
where $c_1$ and $c_2$ are numerical factors that are both equal to one if (and only if)
the formulas for $A$ and $\wt{A}$ in~\eqref{eq:A}  and~\eqref{eq:tA} are correct.
(One could---and we originally did---also do the whole calculation without including the prefactors
$-\tfrac{11+3\sd}{64}$ and $\tfrac{21-5\sd}{2}$ in~\eqref{eq:X1Y1orig}, but then the numbers
in the calculation would be even worse.) We substitute these expressions into $G_2$,
divide by $\xi^2$, with $\xi$ as in~\eqref{xiandeta}, and write the quotient as a power
series $\sum_i C_i s^i$, with $s = \tfrac{4t}{(1-t)^2}$ as usual. By Proposition~\ref{prop:dimGamma},
we know that this power has to reduce to a quadratic polynomial if $c_1$ and~$c_2$ have
the correct values. If we instead treat $c_1$ and $c_2$
as unknowns, then the coefficients of this power series are polynomials in $c_1$ and $c_2$
(with huge coefficients), and we have to show that the infinite system of polynomial
equations $C_i(c_1,c_2)=0$ for all $i>2$ has the unique solution $c_1=c_2=1$.
By computer calculation we find that the g.c.d.~of the resultants ${\rm Res}_{c_1}(C_3,C_4)$ 
and ${\rm Res}_{c_1}(C_4,C_5)$ is equal to $-2^{-111}c_2^8(c_2-1)$, and since $c_2$ cannot be~0,
it must be~1. Then substituting $c_2=1$ into the g.c.d.~of $C_3$ and $C_4$ gives $c_1=1$.  
\par
{\bf 4.} The final point concerns the integrality (away from the prime~2) of our Fourier expansions. The 
coefficients of the expansion of any Hilbert modular form in $X_1$ and~$Y_1$ with rational coefficients 
always has bounded denominators, like in the examples for $G_2$ and $G_4$ above. The same is true also 
for Hilbert modular functions, e.g.
\bas
U +1 &\= (Y_1 - Y_1^2)\,X_1 \+ (-Y_1 -15Y_1^2 + 17Y_1^3 -Y_1^4)\,X_1^2 \+ \cdots \,, \\
\tfrac14\, V & \=  Y_1\,X_1 \,-\,(22Y_1^2 + 2Y_1^3)\,X_2 \+ 
(-2Y_1^2 + 289Y_1^3 + 12 Y_1^4 + Y_1^5)\,X_1^3 \+ \cdots 
\eas
for the generators of the field of symmetric Hilbert modular functions given in Corollary~\ref{HermCor}
above.  We introduce new generators of this function
field, namely $U_1 =  1-\tfrac {4(U+1)}{V}$ and $V_1 = \tfrac14 V$. (Other choices would be
equally good.)
\par
\begin{Prop} In the new set of generators the equality 
$$\ZZ[[X_1,Y_1]]^{\rm sym}  \= \ZZ[[U_1,V_1]]$$ 
of power series rings holds.
\end{Prop}
\par
\begin{proof} We have $\ZZ[[X_1,Y_1]]^{\rm sym}=\ZZ[[S,P]]$, where $S = X_1+Y_1$ and~$P=X_1Y_1$.
We have already seen that $U$ and $V$, and hence also $U_1$ and $V_1$, are symmetric in $X_1$ and~$Y_1$. 
We can thus express $U_1$ and $V_1$ as a power series in $S$ and $P$. Concretely, these expansions start
\bas
U_1 &\= S \+ (-7+3S+3S^2)\,P \+(13-65S+ 37S^2+ 6S^3)P^2 \+  \cdots\,, \\
V_1 &\= P \+ (-22-2S)\,P^2  \+ (289+12S+S^2)\,P^3 \+ \cdots\,.
\eas 
They have integral expansions, so $\ZZ[[U_1,V_1]] \subseteq \ZZ[[S,P]]$.  Conversely, since the expansions
begin $S+\text O(P)$ and $P+\text O(P^2)$, we can recursively compute $S$ and~$P$ as power series 
in $U_1$ and~$V_1$, and these again have integral coefficients:
\bas
S &\=  U_1 \+ (7+3U_1+3U_1^2)\,V_1 \+ (120-20U_1-82 U_1^2+6U_1^3)\,V_1^2  \+ \cdots\,, \\
P &\= V_1 \+ (22+2U_1)\,V_1^2 \+ (693+158U_1+U_1^2)\,V_1^3  \+ \cdots\,.
\eas 
It follows that $\ZZ[[U_1,V_1]] \supseteq \ZZ[[S,P]]$, as desired.
\end{proof}
\par
\begin{Cor} The restriction to~$W_\Pi$ of any symmetric Hilbert modular form of even weight 
with integral Fourier coefficients belongs to $R[[t]]$, where $R = \ord_{17}[\tfrac12]$.
\end{Cor}
\par
\begin{proof} 
Since $\kappa_0 \in R^\times$, eq.~\eqref{sandt} implies that $R[[t]]=R[[s]]$.
Since $\kappa_1 \in R^\times$, eq.~\eqref{UVonW} implies that the restrictions
of $U_1$ and $W_1$ to $W$ belong to this ring. Since any Hilbert modular form
has a Fourier expansion with exponents in a cone strictly contained in the positive 
quadrant (explicitly, if $F = \sum_{r,s} a_{r,s}X_1^rY_1^s$, then $a_{r,s} =0$
unless $\tfrac{9-\sq}{8}r \leq s \leq \frac{9+\sq}{8}r\bigr)$, it contains only
finitely many monomials contributing to any fixed power of~$t$ in the 
$t$-expansion of its restriction to $W$. (Explicitly, $X_1^rY_1^s + X_1^s Y_1^r$
is divisible by~$P^{\text{min}(r,s)}$, and $P\bigl|_W = \text{O}(t)$.) The
corollary follows.
\end{proof}
\par
\begin{Prop}
The power series  $y(t)$  and $\ty(t)$ have $\ord_{17}$-integral expansions
up to denominator~$2$, i.e.\ $y(t), \ty(t) \in R[t]$, where $R = \ord_{17}[\tfrac12]$.
\end{Prop}
\par
\begin{proof} The differential operators $D_1 = \tfrac{\sq}{2\pi i}\frac{\partial}{\partial z_1}$
and  $D_2 = \tfrac{\sq}{2\pi i}\frac{\partial}{\partial z_2}$ can be written as
\bes
 D_1 \=  \lambda\, X_1\, \frac{\partial}{\partial X_1} \+ \mu 
\,Y_1\, \frac{\partial}{\partial Y_1} \,, \quad
 D_2 \= \lambda^\sigma\, X_1\, \frac{\partial}{\partial X_1} \+ \mu^\sigma 
\,Y_1\, \frac{\partial}{\partial Y_1} \,,
\ees
where $\lambda = \tfrac{13-3\sq}{2}$ and $\mu = \tfrac{-21+5\sq}{2}$, and
hence map $R[[U_1,V_1]]$ to itself. On the other hand, they send Hilbert modular functions
to meromorphic Hilbert modular forms of weight~$(2,0)$ and~$(0,2)$ respectively. By
Proposition~\ref{prop:dimGamma}, the quotients of the restrictions to $W_\Pi$ of 
these derivatives by~$y^2$ (resp.~$\ty^2$) must be rational functions of $t$. We have to 
make the right choices of these functions in order not to introduce unwanted denominators.
Since $1/U$ and $V/U^2$ restrict to polynomials in~$s$ (cf.~eq.~\eqref{UVonW}), we choose these
as the Hilbert modular functions to be differentiated, finding
$$y^2 \= c\,\frac{(1-t)^4}{P(t)}\,\cdot\,D_1\Bigl(\frac{V}{U^2}\Bigr)\biggr|_{W_\Pi}\,,
\qquad \ty^2 = \tilde{c}\,\frac{(1-t)^2}{t(1+t)}\,\cdot\,D_2\Bigl(\frac{1}{U}\Bigr)\biggr|_{W_\Pi}\;,$$
where $c= \frac{29+7\sq}{2}$, $\tilde{c} = \tfrac{7+\sq}{2}$ and $P(t)$ is the polynomial
$$ P(t) \= t\,(1+t)\,\bigl(1-\tfrac{31-7\sq}{2}t\bigr)\,\bigl(1-\tfrac{31+7\sq}{64}t\bigr)
\,\bigl(1-\tfrac{647-153\sq}{8}t + t^2\bigr)\,.$$
Since $c$ and $\tilde{c}$ belong to $R$ and the polynomials in the denominators
are in $1+tR[t]$, it follows that $y^2$ and~$\ty^2$, and hence also $y$ and~$\ty$,
are in $R[[t]]$.
\end{proof}
\par
Although we have been working with $D=17$ all the time and using concrete generators of the field of
Hilbert modular functions, it is clear from the proof that the basic principle---the use of integral 
coefficients for Hilbert modular forms and the base change in two variables---can be applied for any $D$, 
giving the $\ord_D$-integrality (up to finitely many primes in the
denominator) of solutions $y(t)$ and $\ty(t)$ of the corresponding
Picard-Fuchs equations for any $D$.

\newpage

\part*{Part III\,: Modular embeddings via derivatives of theta functions}

In Part~II we deduced the equation of the \Teichmuller\ curve in the
Hilbert modular surface for $D=17$ and of a Hilbert modular form cutting out
the \Teichmuller\ curve starting from the differential
operators $L$ and $\wt L$, which in turn were deduced from the 
explicit algebraic model of the family as given in~\eqref{BMform}.
In Part~III we will show that there is a general construction of
a Hilbert modular form of mixed weight that cuts out the \Teichmuller\ curve.
The construction, given in Section~\ref{sec:ThetaDer}, uses derivatives of theta functions. 
The short proof depends on the description of \Teichmuller\ curves using eigenforms
for real multiplication with a double zero  (see Theorem~\ref{thm:TMclassg2}).
We also verify that in the case $D=17$ we get the same equation for
the \Teichmuller\ curve as the one already obtained in Section~\ref{sec:eqfromdiffeq}.
\par
Next, in Sections~\ref{sec:cuspTheta0} and~\ref{sec:disjointred}, we develop the theory of 
\Teichmuller\ curves in genus two ``from scratch'' starting from the new
definition as vanishing loci of theta derivatives. In particular we give new proofs from 
this point of view of the cusp classification and of the facts that these
curves are Kobayashi geodesics and are disjoint from the reducible locus
and hence are \Teichmuller\ curves. We do not know how to
reprove the irreducibility from the viewpoint of theta functions.
\par
Along the way, in Section~\ref{sec:HirzBain} we show that Bainbridge's 
compactification of Hilbert modular surfaces using the moduli space of
curves is indeed a toroidal compactification. Recall that this property 
of Hirzebruch's compactification was the model on which the notion of 
toroidal compactifications was developed.
\par

\section{\Teichmuller\ curves are given by theta derivatives}
\label{sec:ThetaDer}

Bainbridge has shown in \cite{Ba07}, \changed{Theorem~10.2,} 
that the \Teichmuller\ curves $W_D$ 
defined in Section~\ref{sec:defTeich} \changed{are} 
given as the vanishing locus of a modular form $(3,9)$ for all $D$.
We determine this form explicitly. It turns out to be a product of 
derivatives of theta series
restricted from the Siegel half space to Hilbert modular varieties. 

\subsection{Theta functions and their restrictions to Hilbert modular varieties}
\label{sec:recalltheta}

We recall the definition of the classical theta-functions and properties
of their derivatives. Although we are ultimately interested in $g=2$ only, 
we can keep $g$ general without effort when setting up the definitions.
\par
For $m, m' \in (\tfrac{1}{2}\ZZ)^g$ (considered as row vectors) we define
the (Siegel) theta function
$$\Theta \Tchi{m}{m'}: 
\left\{ \begin{array}{lcl} \CC^g \times \HH_g &\to& \CC \\
(v,Z)  &\mapsto & 
\displaystyle \sum_{x \in \ZZ^g + m} \e\left(\tfrac{1}{2}xZx^T + x(v+m')^T
\right). \end{array} \right.
$$
with characteristic $(m,m')$. The evaluation of a theta-function 
at $v=0$ is called a {\em theta constant}. The theta-function (and the
characteristic $(m,m')$) is called {\em odd} if $4m(m')^T$ is odd and {\em even} otherwise.
Odd theta-constants vanish identically as functions of~$Z$. Up to sign, $\Theta \Tchi{m}{m'}$
depends only on $m$ and~$m'$ modulo~$\ZZ^g$.
\par

The theta constants $\Theta \Tchi{m}{m'}(0,Z)$ are modular forms of
weight $\tfrac 12$ for some subgroup (in fact $\Gamma(4,8)$, see e.g.~\cite{Ig72}) 
of  $\Sp(2g,\ZZ)$. The partial derivatives with respect to any $v_i$ are 
not modular, but if we restrict to $v=0$ and consider the gradient (as column vector)
$$ \gr \Bigl(\Theta \Tchi{m}{m'}(0,Z)\Bigr) \= \left( \frac{\d}{\d v_i} \Theta \Tchi{m}{m'}(v,Z)
|_{v=0}\right)_{i=1,\ldots,g}\,, $$
then we get a vector-valued modular form. That is, if $\Theta \Tchi{m}{m'}(0,Z) =0$, one calculates that for any
$M = \left(\begin{smallmatrix} A & B \\  C & D \end{smallmatrix} \right)\in \Gamma(4,8)$
the gradient transforms as
$$  \gr\Bigl(\Theta \Tchi{m}{m'}\Bigr)(0,M \cdot Z) \= \zeta_8 \, 
\det(CZ+D)^{1/2} \, (CZ+D)\,  
 \gr\Bigl(\Theta \Tchi{m}{m'}\Bigr)(0,Z),
$$
where $\zeta_8$ is an $8$-th root of unity depending on $M$.
\par
If~$K$ is a totally real number field of degree~$g$ over~$\Q$ with ring of integers~$\O$, 
then just as in the special case~$g=2$ we can define a $g$-dimensional Hilbert modular variety
$X_K=\HH^g/\SLOA{\O^\vee}$ and a Siegel modular embedding $(\Psi, \psi)$ of~$X_K$ into
$\HH_g/{\rm Sp}(2g,\ZZ)$, given by a matrix $B \in {\rm GL}(g,\RR)$ as in~\eqref{eq:SMEA}. 
Recall that this means that $\Psi: \SLOA{\O^\vee} \to {\rm Sp}(2g,\ZZ)$ is a homomorphism 
and that $\psi(\zz) = B^T \text{diag}(z_1,\ldots,z_g) B$ is a map that is equivariant
with respect to $\Psi$.
We then denote by $\theta\Tchi{m}{m'}(\zz) = \Theta \Tchi{m}{m'}(0,\psi(\zz))$ 
and $\gr \theta \Tchi{m}{m'}(\zz) = \gr \Theta  \Tchi{m}{m'}(0,\psi(\zz))$ 
the restriction of the theta functions and their gradients to~$\HH^g$. 
We also write $\theta\TchiP{m}{m'}{\Psi}$ if we want to emphasize the dependence on the
modular embedding. The modularity of the Siegel theta functions imply that the (Hilbert) theta constants
$\theta\Tchi{m}{m'}(\zz)$ are modular forms with a character of order~$8$ of weight $(\tfrac 12,\ldots, \tfrac 12)$
for a subgroup of finite index of the Hilbert modular group $\SLOA{\O^\vee}$.
The theta constants are non-zero if and only if $(m,m')$ is even, while the theta
gradients are modular if and only if $(m,m')$ is odd.
\par
The modular transformation of the derivative of theta constants for $(m,m')$ odd now reads 
$$  \gr \theta \Tchi{m}{m'} (\gamma \cdot \zz ) \= 
\zeta_8 \, \det(\hat{c} \psi(\zz)+ \hat{d})^{1/2} \, B^{-1} J(\zz,\gamma) B\, \gr \theta \Tchi{m}{m'} (\zz) $$
for $\gamma = \sm abcd$ in the subgroup $\Psi^{-1}(\Gamma(4,8)) \subset \SLOA{\O^\vee}$,
where $\hat{e}$ for $e \in K$ denotes the diagonal matrix with entries
$\sigma_j(e)$ given by the different real embeddings $\sigma_j$ of $K$ and 
$J=J(\zz,\gamma)$ is the diagonal matrix with entries $\sigma_j(c)z_j+\sigma_j(d)$.
Consequently, the vector-valued modular form $B\,\gr\theta\Tchi{m}{m'}$ transforms with the automorphy
factor $\zeta_8 \det(\hat{c} \psi(\zz)+ \hat{d})^{1/2} J(\zz,\gamma)$, which is also a diagonal matrix. 
We will calculate the root of unity~$\z_8$ for $K=\Q(\sqrt D)$ with $D \equiv 1 \pmod 8$ in detail below.
\par
To summarize, the $i$-th entry $D_i\theta\Tchi{m}{m'}(\zz)$  of the column vector  
$B \, \gr \theta_\psi$ is a Hilbert modular form of 
multi-weight $(\tfrac 12,\ldots,\tfrac 12, \tfrac 32, \tfrac 12,\ldots, \tfrac 12)$
and a character stemming from the $8$-th root of unity.
\par
We can also express the functions $D_i\theta \Tchi{m}{m'}$ as derivatives in certain 
eigendirections. For this purpose replace the original coordinates $\vv = (v_1,\ldots,v_g)\in\C^g$
by the ``eigendirection coordinates'' $\uu = B \vv= (u_1,\ldots,u_g)$. Then we may write  
$$ D_i\theta \Tchi{m}{m'}(\zz) = \frac{\d}{\d u_i} \Theta \Tchi{m}{m'}(\vv,\psi(\zz)) |_{\uu=0}.$$
\par

\subsection{$W_D$ is the vanishing locus of the theta-derivatives}
\label{sec:WD_theta}

With this preparation we can now determine Bainbridge's  modular form
for general $D$. 
\par

\begin{Thm} \label{thm:TeichviaDTH}
The function 
\be  \label{eq:defofDtheta}
 \PTD(\zz) \= \prod_{(m,m')\;\; {\rm odd}} D_2\theta\Tchi{m}{m'}(\zz) 
\ee
is a modular form of weight $(3,9)$ (with character) for the full Hilbert modular group $\SLOD$.
Its vanishing locus is precisely the \Teichmuller\ curve~$W_D$.
\end{Thm}
\par
We will give \changed{additional properties of the modular form~$\PTD$} if the prime $(2)$ splits, i.e., if $D \equiv 1 \pmod 8$. 
In that case, as already mentioned in~\S\ref{sec:HMFdisc17}, the discussion of the numbering
of theta characteristics given there for $D=17$ always holds. In particular, 
one has a quartic character $v_0$ given by~\eqref{eq:defv0}. Here,
$\PTD$ is a product of two modular forms and its character is $v_0^2$. 
Indeed, using the shorthand notation introduced in \S\ref{sec:HMFdisc17} we calculate the action
of generators of $\SLOA{\O^\vee}$ on the functions $D_2\theta\Tchi{m}{m'}$. 
From Table~\ref{tab:reindex} in~\S\ref{sec:HMFdisc17}, we deduce that the products of theta derivatives 
\begin{equation} \label{eq:defDLDR}
D_L \= D_2\theta_{0X} \cdot D_2\theta_{1X} \cdot D_2\theta_{\infty X}
\quad \text{and} \quad D_R \= D_2\theta_{X0} \cdot D_2\theta_{X1} \cdot D_2\theta_{X\infty}.
\end{equation}
are modular forms for the full group $\SLOA{\O^\vee}$ of weight $(\tfrac 32, \tfrac 92)$ with 
a character of order $8$ according to Table~\ref{cap:onD2} below, in which
$\alpha = \tfrac{1 + \sqrt{D}}2$ as usual. The function $\PTD$ then equals
$D_L D_R$ and has character $v_0^2$ of order~$2$.
\begin{table}[h]
\begin{tabular}{|c|ccc|}
\hline
$f(z_1,z_2)$ & $f(z_1+1,z_2+1)$ & $f(z_1+\alpha,z_2+\s(\alpha))$ 
& $z_1^{-3/2} z_2^{-9/2} f(1/z_1,1/z_2)$ \\
\hline
$D_L$ & $-D_L$ & $ -\zeta_8      D_L$  & $-D_L$ \\
$D_R$ & $-D_R$ & $ \zeta_8^{-1} D_R$  & $-D_R$ \\
\hline
\end{tabular} \medskip
\caption{ The action of $\SL\O$ on the $D_2\theta_L$ and $D_2\theta_R$ } \label{cap:onD2}
\end{table}
This computation also shows that if $D \equiv 1 \pmod 8$, then the \Teichmuller\ curve $W_D$
has two components, the vanishing loci of $D_L$ and of $D_R$. This is our
equivalent of McMullen's spin invariant that we referred to after
Theorem~\ref{thm:TMclassg2}.
\par
Before proceeding to the proof of Theorem~\ref{thm:TeichviaDTH}, we write out the
Fourier expansions of the theta derivatives explicitly.
Fix a basis $(\omega_1,\omega_2)$ of $\O_D$ as we did in~\eqref{eq:SMEA}
for the construction of a Siegel modular embedding, say
$(\omega_1,\omega_2) = (1,\gamma_0 = \tfrac{D+\sqrt{D}}{2})$. Then the 
first row the matrix $B$ is the basis of $\O$ 
and the second row consists of the Galois conjugates. 
Given $m \in \frac{1}{2}\ZZ^2$, we let 
$$\rho(x) \= (x+m)\, \cdot\,  (1,\gamma_0 )^T$$ 
for $x \in \ZZ^2$. Then for $(m,m')$ odd the theta derivative is given by
\begin{equation} \label{eq:thetaexpl}
D_2\theta  \Tchi{m}{m'}(z_1,z_2) \= i^{4m(m')^T}\sum_{x \in \ZZ^2}
(-1)^{2x\cdot(m')^T} \s(\rho(x)) q_1^{\rho(x)^2/2} q_2^{\s(\rho(x))^2/2}\,,
\end{equation}
where $q_j = \e(z_j)$ as usual. 
\par
\begin{proof}[Proof of Theorem~\ref{thm:TeichviaDTH}]
The modularity of $\PTD$ for some level subgroup follows from the
corresponding property of the Siegel theta functions, as shown in 
\S\ref{sec:recalltheta}. It follows from the calculations above for 
$D \equiv 1 \pmod 8$ that the factor group permutes 
the set of theta characteristics, preserving their parity. Moreover 
we read off from the table that $\PTD$ has character $v_0^2$ in this case.  In
other cases \changed{$D \not\equiv 1 \pmod 8$, one checks similarly that
the factor group still permutes theta characteristics, preserving their parity.
In fact, classical calculations of Igusa show the full 
symplectic group $\Sp(4,\Z)$ preserves the parity.
Consequently,} $\PTD$ is still a modular form, with some character, for 
the full Hilbert modular group.
\par
For the second statement, let $C$ be a curve of genus two with period matrix~$Z$.
It is cut out in its Jacobian as the vanishing locus of 
$\Theta \Tchi{0}{0}(v,Z)$. \changed{See e.g.\
\cite{FaKra}, Chapter~VI and Chapter~VII.1 for standard properties of theta
functions and Weierstrass points.} 
\changed{The} two-torsion points in $\Jac(C)$ are
$Zm^T + (m')^T$ for $m,m' \in \left(\frac{1}{2}\ZZ\right)^2$.
The Weierstrass points of $C$ are precisely the $6$ points $Zm^T + (m')^T$
for $(m,m')$ odd. Moreover, a holomorphic differential 
form $\omega$ has a double zero on $C$ if and only if $\omega$ vanishes 
(automatically doubly) at a Weierstrass point.  
\par
Suppose the point $[C] \in\M_2$ lies on $W_D$. Then $C$ has real multiplication, 
so $Z = \psi(\zz)$ for some fixed Siegel modular embedding $\psi$.
We identify the universal covering $V$ of the Jacobian $J(C)$ with the 
dual to $H^0(C,\Omega^1_C)$. \changed{We claim that the characterizing 
condition on the existence of an eigenform with a double zero 
(see Theorem~\ref{thm:TMclassg2}) is equivalent to the vanishing of the
derivative of the Riemann theta-function in the second eigendirection $u_2$ 
at a Weierstrass point. In fact, the definition of the Abel-Jacobi
map $C \to J(C)$ implies that its projectivized tangent map is given
by $p \mapsto (\omega_1(p):\omega_2(p))$ in the chosen eigenbasis
of $H^0(C,\Omega^1_C)$ and so~$\omega_1$ vanishes at~$p$ if and
only if the $u_2$-derivative of the defining equation of~$C$ in~$J(C)$
vanishes at~$p$.}  Since we defined $\uu = B\vv$  
this just means that
\be \sum_{j=1}^2  \omega^\sigma_{j} \left(\frac{\d}{\d v_j} \label{eq:thetacond} 
\Theta \TchiP{0}{0}{D}(\vv,\psi(\zz))\right)|_{\vv = \psi(\zz)m^T + (m')^T} \= 0
\ee
for some even $(m,m')$.
\par
\changed{By the product rule, the derivative with respect to $v_i$ of the defining equation 
\bes 
\Theta \Tchi{0}{0}(\vv+Zm^T+(m')^T,Z)) \= \e(-\pi i mZm^T-2\pi i m(\vv+m')^T) \cdot\Theta\Tchi{m}{m'}(\vv,Z)
\ees
for $Z = \psi(\zz)$ has two summands. The factor given by differentiating the
exponential term vanishes by definition for every pair of a point 
$\zz \in \HH^2$ and $v \in \Jac(C_\zz)$. Since $\e(\cdot)$ is never zero, the vanishing
of the derivative~\eqref{eq:thetacond} is equivalent to $\zz$ being in the 
vanishing locus of $\PTD$.}
\par
To show that the modular form $M_{W_D}$  vanishes nowhere else there are
several options. The first is to remark that the above argument can
be inverted for Jacobians of smooth curves. So one has just to show
that the vanishing locus of $\PTD$ is disjoint from the reducible
locus. We give two proofs of this fact that do not rely on any \Teichmuller\ theory 
in Section~\ref{sec:disjointred}. Yet another way to conclude 
is to compare the weight of $\PTD$ with the modular form 
that cuts out $W_D$ in Bainbridge's theorem from~\cite{Ba07}, \changed{Theorem~10.2}. 
They are both of weight~$(3,9)$.
\end{proof}
\par

{\bf The example $D=17$ revisited.}  If we calculate the theta series and their derivatives for~$D=17$
as was done in Section~\ref{sec:HMFdisc17}, then we can verify that the product of $D_R(z_1,z_2)$ 
with $D_R(z_2,z_1)$ is indeed proportional to the function~$F_{17}^1$ given in~\eqref{eq:F171},
and similarly that the product of $D_L(z_1,z_2)$ with $D_L(z_2,z_1)$ is proportional to
the Galois conjugate function~$F_{17}^0$.

\section{Cusps and multiminimizers}
\label{sec:cuspTheta0}

Fix an invertible $\O_D$-ideal $\mathfrak{a}$. Our aim in this section is to list the
branches of the theta-derivative vanishing locus $W_D$ through
the cusp of the Hilbert modular surface $X_D$ determined by the 
class of $\mathfrak{a}$. Our Ansatz is to describe the
branch of $W_D$ defined by a Hilbert modular form in $(q_1,q_2)$ 
by $q_1 = q^\alpha$ and $q_2 = q^{\s(\alpha)}(1+P)$, where $q$ a suitable local parameter and $P$ 
a power series with positive valuation in~$q$. This will lead us to consider a minimization
problem on quadratic forms. The solutions are given by so-called ``multiminimizers.''

The main result of this section will be the following characterization of
cusps of $W_D$, derived from Fourier expansions only. We call an indefinite quadratic 
form $Q=[a,b,c]$ {\em \Pred}\footnote{There seems to be no standard terminology 
for these quadratic forms and quadratic irrationalities. They appear 
implicitly in~\cite{ChZa93} 
and~\cite{Ba07}. } if $a > 0 > c$ and $a + b+ c <0$. To a quadratic form $Q$
we associate the quadratic irrationality $\la_Q = \frac{-b + \sqrt{D}}{2a}$.
The quadratic form is \Pred\ if and only if $\la=\la_Q$ satisfies
\be \label{eq:defstd} \la \, > \, 1 \, > \, 0 \, > \la^\s \;. \ee
A quadratic irrationality $\la$ satisfying~\eqref{eq:defstd} will also be called {\em \Pred}.
\par
\begin{Thm} \label{thm:cuspsWD}
Suppose that $D$ is a fundamental discriminant or, more generally, that $\fraca$ is an invertible $\O_D$-ideal. 
Then there is a bijection between the cusps of $W_D$ mapping to the cusp $\fraca$ of $X_D$
and the set of pairs $(Q,\overline{r})$ consisting  of a \Pred\ quadratic form $Q=[a,b,c]$ of
discriminant $D$ such that $[\langle 1, \la_Q \rangle] = [\mathfrak{a}]$ together with a residue class
$\overline{r} \in \ZZ/(a,c)\ZZ$.
\end{Thm}
\par
This will be proved in~\S\ref{sec:cuspHMS}.  We will also comment on the case of non-invertible
ideals after a discussion of cusps of Hilbert modular surfaces in the same subsection.
\par
Given that $W_D$ is a \Teichmuller\ curve, as we showed in 
Theorem~\ref{thm:TeichviaDTH} using the eigenform definition and
will show again in the next section using the theta viewpoint only, 
this characterization reproves the list of cusps of~\cite{McM05}. 
(See also~\cite{Ba07}, \changed{Theorem~8.7(5)}.) 
For comparison, we will briefly sketch the approach based on flat surfaces in~\S\ref{sec:Cuspviaflat}.
\par
The proof of Theorem~\ref{thm:cuspsWD} can be applied verbatim to prove the following result
on the reducible locus~$P_D$ (see~\S\ref{sec:redloc}) from the theta function viewpoint. This
result was also proven by Bainbridge using the flat surface viewpoint on cusps.
\par
\begin{Thm} \label{thm:cuspsPD}
Let $D$ and $\fraca$ be as above. 
Then there is a bijection between the cusps of $P_D$ mapping to the cusp $\fraca$ of $X_D$
and the set of pairs $(Q,\overline{r})$ as in Theorem~\ref{thm:cuspsWD}.
In particular, for every given $\fraca$ the numbers of cusps of $P_D$ and $W_D$
mapping to the cusp $\fraca$ of $X_D$ coincide. 
\end{Thm}
\par
We will end the section by giving an algorithm in \S\ref{sec:compMMM} to compute
multiminimizers using continued fractions.
\par
\subsection{Multiminimizers}
Let $F(x,y)=Ax^2+Bxy+Cy^2$ be a positive definite binary quadratic form with real coefficients.
On each of the cosets of $2\Z^2$ in $\Z^2$ this form assumes its minimum a finite number
of times.  Of course on the trivial coset the minimum is~0 and is assumed exactly once,
while on each of the other cosets the minimum is generically attained exactly twice,
by some non-zero vector and its negative.  We call the form $[A,B,C]$ {\it multiminimizing}
if on at least one of the three non-trivial cosets the minimum is attained more than twice.
These forms are classified by the following proposition.

\begin{Prop}  \label{prop:MMchar1}
A positive definite binary quadratic form is multiminimizing if and
only if it is diagonalizable over~$\Z$.  In this case, there is exactly one coset of $2\Z^2$
in $\Z^2$ on which the form has a multiple minimum; this minimum is assumed exactly twice
(up to sign) and is the sum of the minima in the other two cosets. 
\end{Prop}

\begin{proof} This is proved using reduction theory. Suppose that the form $F$ is multi\-minimizing. 
Since this property is obviously $\SL\ZZ$-invariant, we can assume that $F$ is reduced,
i.e. $F=[A,B,C]$ with $C\ge A\ge|B|$.  By Cauchy's inequality we have $|Bxy|\le A(x^2+y^2)/2\,$, and hence 
$$F(x,y)\;\ge\;\frac A2\,x^2\+\Bigl(C-\frac A2\Bigr)\,y^2\,, $$
for all $(x,y)\in\R^2$.  In particular,
\bas  |x|\ge1,\;|y|\ge2\quad&\Rightarrow\quad F(x,y)\;\ge\; 4C\,-\,3A/2 \;\,>\;A\,, \\ 
|x|\ge2,\;|y|\ge1\quad&\Rightarrow\quad F(x,y)\;\ge\; C\+3A/2 \quad >\; C\,, \\ 
|x|\ge3,\;|y|\ge1\quad&\Rightarrow\quad F(x,y)\;\ge\;C\+4A \qquad >\; A+|B|+C\,, \\ 
|x|\ge1,\;|y|\ge3\quad&\Rightarrow\quad F(x,y)\;\ge\; 9C\,-\,4A \quad\;\,>\;A+|B|+C\,. \eas
These equations show that the smallest value of $F$ on the coset ``(odd,\,even)'' is attained only at
$(\pm1,0)$ and equals $A\,$, that the smallest value of $F$ on the coset ``(even,\,odd)'' is attained only at
$(0,\pm1)$ and equals $C\,$, and that the two smallest values of $F$ on the coset ``(odd,\,odd)'' are attained
only at $(\pm1,\pm1)$ and are equal to $A-B+C$ and $A+B+C\,$.  In particular, $F$ is multi-minimizing
if and only if $B$ vanishes, in which case the only coset on which it attains its minimum more than once is
``(odd,\,odd)'' and the minimum there is attained exactly twice (up to sign) and is
the sum of the unique minima in the other two cosets.  \end{proof}

Let $\fraca$ be a fractional $\O_D$-ideal  
in a real quadratic field~$K$ and $\xi$ a non-zero coset of $\fraca$ in $\frac12\fraca$. 
(Thus there are three possibilities for~$\xi$ given~$\fraca$.) We denote by 
$\widetilde{\rm{MM}}(\fraca, \xi)$ the set of non-zero $\a\in K$ such that the quadratic 
form  $F_\a(x) = \tr(\alpha x^2)$ is positive definite and assumes its minimum value on 
the coset $\xi$ more than twice (up to sign), and call the elements of this set {\it multiminimizers} for $\xi$.  
We denote by $\widetilde{\rm{MM}}(\fraca)$ the set of all {\em multiminimizers for $\fraca$}, 
i.e.\ the union of the sets $\widetilde{\rm{MM}}(\fraca, \xi)$ for all three cosets~$\xi$.
Clearly, $\widetilde{\rm{MM}}(\fraca)$ is invariant under multiplication by positive rational numbers and 
by the squares of elements of the group $U_D$ of units $\ve$ of $\O_D$. We set
\be
\label{eq:defMM}
{\rm{MM}}(\mathfrak{a}) \= \widetilde{\rm{MM}}(\mathfrak{a}) /\bigl(\QQ_+^\times \cdot U_D^2\bigr)\;.
\ee
Later we will often use the representatives of $\alpha \in {\rm{MM}}(\mathfrak{a})$
that are primitive in $(\fraca^2)^\vee$. They are unique up to multiplication by $U_D^2$.
\par
\begin{Prop} \label{prop:MMchar2} Let $\fraca$ be a fixed fractional $\O_D$ 
ideal of $K$. Then there is a bijection between ${\rm{MM}}(\mathfrak{a})$ 
and the set of \Pred\ quadratic forms $Q = [a,b,c]$ with $b^2 -4ac = D$ in the wide ideal class of~$\fraca$. 
\par
Moreover, given $\alpha \in \widetilde{\rm{MM}}(\fraca)$, there is a
unique basis  $(\om_1,\om_2)$ such that $\omega_2 > \omega_1 >0$ of $\fraca$ 
with respect to which the form $F_\a$ is diagonal, and the coset $\xi$ 
is then $\frac12(\om_1+\om_2)+\fraca$. 
\end{Prop}
\par
\begin{proof} Given $Q = [a,b,c]$ choose $\mu \in K^\times$ positive such that $\mu \langle 1, 
\la_Q \rangle = \mathfrak{a}$ and take 
\be \label{eq:def_alphaMM}
\alpha \= \frac{-ac}{\mu^2 \la_Q \sqrt{D}} \= \frac{a}{\mu^2}\cdot\frac{b+\sqrt{D}}{2\sqrt{D}}\,.
\ee 
This is positive definite if $Q$ is standard.
Since $Q$ is simple the basis $\omega_1 = \mu $, $\omega_2 = \mu \la_Q$ satisfies
the conditions stated. Moreover, in this basis the quadratic form $F_\alpha$ is 
$$F_\alpha = [A,0,C], \qquad A = \tr(\alpha \omega_1^2) = a, \quad C =  \tr(\alpha \omega_2^2) = -c$$
since $\tr(\alpha \omega_1 \omega_2) = \tr(-ac/\sqrt{D}) =0$. 
Thus $\alpha$ is a multiminimizer.
\par
Conversely, if $\a$ is a representative of a class in  
$\rm{MM}(\fraca)$ then it follows from 
Proposition~\ref{prop:MMchar1} that there exists a unique basis 
$(\om_1,\om_2)$ of $\fraca$ (up to interchanging the $\om_i$ and changing their signs)
of~$\fraca $ with respect to which $F_a(x)$ is diagonal, and the coset $\xi$ is then $\frac12(\om_1+\om_2)+\fraca$. 
This also proves the last statement. In this basis $F_\alpha = [A,0,C]$ with $A = \tr(\alpha \omega_1^2)$
and $C = \tr(\alpha \omega_2^2)$. We choose signs and order the basis such that
$\la = \omega_2 /\omega_1 > 1$ and claim that then $\s(\la)<0$, so that $\lambda$ is the 
root of a \Pred\ form. 
In fact, we have $\alpha = q/(\omega_1 \omega_2 \sqrt{D})$ for
some $q \in \QQ^\times_+$. Thus $C = \tr(q \la/\sqrt{D})$ and
$$A = \tr(q/(\la \sqrt{D})) = - N(\la) C.$$ Since $F_\alpha$ is positive
definite (by definition of a multiminimizer), the numbers $A$ and $C$ are positive, so this implies 
that $N(\la)<0$ as claimed. One checks immediately that $\la$ does not depend
on the representative of the multiminimizer in $\rm{MM}(\fraca,\xi)$
we have chosen. We take $Q = [a,b,c]$ so that $\lambda$ satisfies $a\lambda^2 + b\lambda + c = 0$ 
with $a>0$ and $a,\,b,\,c$ coprime integers, and since $\langle 1, \lambda \rangle$ is an 
invertible $\O_D$ module, we then have $b^2-4ac=D$.
\par
Obviously, the composition $\lambda_Q \mapsto \alpha(\la_Q) \mapsto \lambda(\alpha)$ 
is the identity. In the other direction, note that $\alpha$ is determined
by $(\omega_1,\omega_2)$ up to a positive rational number and that
$\mu$ with  $\mu \langle 1, \la_Q \rangle = \mathfrak{a}$
is determined up to a unit. Consequently, each of the distinguished 
basis elements $\omega_i$ is determined up to a unit and $\a$ is
determined up to a square of this unit. Since multiminimizers were
defined in \eqref{eq:defMM} by these two equivalence relations, 
this shows the bijection we claimed. 
\end{proof}

\subsection{Cusps of Hilbert modular surfaces} \label{sec:cuspHMS}
Classically, cusps of the Hilbert modular surface $X_D$ are defined to 
be the equivalence classes of points in $\PP^1(K)$ under the action of $\SLO$.
Equivalently, we may define a cusp as an exact sequence 
$$ 0 \to \fraca^\vee \to \O_D^\vee \oplus \O_D \to \fraca \to 0 $$
of torsion-free $\O$-modules up to the action of $\SLO$ on $\O^\vee \oplus \O$ 
and its sub $\O$-modules. The modules $\fraca$ arising in this way are
{\em quasi-invertible}, i.e.\ invertible $\O_E$-module for some order 
$\O_E \supseteq \O_D = \O$.
Yet another equivalent viewpoint to define a cusp is by the 
class of an invertible $\O_E$-ideal $\fraca$ together with an element
$r \in \ZZ/\sqrt{\tfrac DE}\ZZ$.
\par
We briefly recall how to see the equivalence of these definitions. For the equivalence 
of the first two definitions, intersect the line $L \subset K^2$ determined by a 
point in $\PP^1(K)$ with a fixed embedding of  $\O^\vee \oplus \O$ in $K^2$ to get 
the exact sequence and conversely tensor the exact sequence with $K$ over $\O$. 
That such an extension class is determined by $r \in \ZZ/\sqrt{\tfrac DE}\ZZ$ can be 
deduced from the calculation in \cite{Ba07}, Proposition~7.20 or \cite{BM12}, Theorem~2.1.
\par
\medskip
As a preparation for the proof of Theorem~\ref{thm:cuspsWD} we determine the 
Fourier series of $\PTD$ at a given cusp $\fraca$ of $X_D$.
For a basis $\bfom = (\omega_1,\omega_2)$ of 
the ideal $\fraca$ let  $\rho_\bfom(x) \= (x+m)\cdot \bfom^T$ for $x \in \ZZ^2$, 
with the dependence on $m$ suppressed in the notation.
Then for $(m,m') \in (\tfrac12 \ZZ)^2$ odd we define
\begin{equation} \label{eq:thetafraca}
D_2\theta  \TchiP{m}{m'}{\bfom}(z_1,z_2) \= \sum_{x \in \ZZ^2}
(-1)^{2x\cdot(m')^T} \s(\rho_\bfom(x)) q_1^{\rho_\bfom(x)^2/2} q_2^{\s(\rho_\bfom(x))^2/2} \, , 
\end{equation}
where $q_i = \e(z_i)$. Note that $D_2\theta  \TchiP{m}{m'}{\bfom}$
depends on the chosen basis $\bfom$,  but a base change can be
compensated for by letting the base change matrix also act \changed{on} the
characteristic $(m,m')$. Consequently, the product 
\be \label{eq:defPTDa}
\PTDa = \prod_{(m,m')\;\; {\rm odd}}  D_2\theta  \TchiP{m}{m'}{\bfom}
\ee
is of the form $\sum_{y \in \fraca} c(y) \exp(\tr(y^2z))$
and thus invariant under an upper triangular matrix in $\SLA$. 
\par
\begin{Lemma} \label{le:DthetaOtherCusp}
The Fourier expansion of~$\PTD$ at the cusp~$\fraca$
is proportional to $\PTDa$.
\end{Lemma}
\par
\begin{proof} To avoid the generally hard problem of finding the Fourier development
of a modular form at a different cusp we use the fact that the vanishing locus
of $\PTD$ has an intrinsic formulation in terms of eigenforms. In Theorem~\ref{thm:TeichviaDTH} we
proved that this vanishing locus corresponds to the set of principally polarized 
abelian varieties with real multiplication such that the first eigenform has a double zero. 
This proof works for any Siegel modular embedding, for example the one given
at the end of~\S\ref{sec:HMGHMS}, where the
locus of real multiplication is $X_{D,\fraca}$ and $\psi$ is constructed with the
help of the matrix $B$ as in~\eqref{eq:SMEA} having $\bfom = (\omega_1,\omega_2)$ as its first column.
The restriction of the Siegel theta function with characteristic $(m,m')$ 
via this modular embedding is just $D_2\theta\TchiP{m}{m'}{\bfom}$. To 
complete the proof, 
we note that the cusp at $\infty$ of $X_{D,\fraca}$ is just the cusp $\fraca$
of $X_D$. To see this, take a matrix $\sm{\fraca^{-1}}{\fraca^\vee}
{(\fraca^\vee)^{-1}}{\fraca} \cap \SL{K}$. It 
conjugates $\SLO$ into $\SLA$, since $\fraca^\vee = \fraca^{-1} \O^\vee$ and the line
at infinity in $\PP^1_K$ intersects $\fraca^\vee \oplus \fraca$ in the submodule
$\fraca^\vee$, as required. 
\end{proof}
\par
\begin{proof}[Proof of Theorem~\ref{thm:cuspsWD}]
We first determine the cusps of the vanishing locus of $\PTDa$, as
defined in~\eqref{eq:defPTDa},
that map to the cusp $\infty$
of the Hilbert modular surface $X_{D,\fraca}$. 
Because functions in a neighborhood of the cusp 
$\fraca$ have the form $\sum_{\nu} c_\nu \e(\tr(\nu z))$ with 
$\nu \in (\fraca^\vee \fraca^{\i})^\vee = \fraca^2$, we can choose a local parameter $q$
at a branch of this locus, with $q=0$ at the cusp, of the form
$q = \e(z/\alpha)$ with $\alpha \in (\fraca^2)^\vee$ primitive. 
A lift of the vanishing locus to $\HH \times \HH$ looks locally like
$z_2 = \varphi(z_1)$, where $\varphi(z) = \frac{\alpha^\sigma}{\alpha} z + C 
+ \tfrac{\ve(q)}{2\pi i}$ as $\Im(z) \to \infty$ for some $C \in \CC$ and some 
power series $\ve$ in~$q$ with no constant term.  Then we have
\begin{equation} \label{eq:branch}
\e(\nu z_1 + \nu^\sigma z_2)|_{\text{locus}} \= \e(C \nu^\sigma) \, q^{\tr(\alpha \nu)} e^{\sigma(\nu)\ve(q)} \qquad
\text{for all} \quad \nu \in \fraca^2\,,
\end{equation}
so that the restriction of any Hilbert modular form becomes a power series in~$q$.
Making a different choice of the lifting would change $C$ by  an integral multiple
of $N(\fraca)^2\sqrt{D}/\alpha$, so that the quantity
\be \label{eq:newS}
S \= \e \Bigl( \frac{C\alpha}{N(\fraca)^2\sqrt{D}} \Bigr)\; \in \, \CC^*
\ee
is independent of the choice of the lifting. We will show below that $S$ is in
fact a rational power of an element of~$K^*$.
\par
The resulting $q$-exponents after plugging~\eqref{eq:branch} into~\eqref{eq:thetafraca}
are of the form $\tr(\alpha\rho_\om(x)^2)$. In order for the theta derivative to vanish,
the smallest exponent of~$q$ must occur twice, so this quadratic form has to take
its minimum twice (with $x$ and $-x$ not distinguished).  Hence 
$\alpha$ is a multiminimizer for $\fraca$. Recall that this specifies
$\alpha$ only up to $\QQ_+^\times \cdot U_D^2$, but here the fact that $q = \e(z/\alpha)$
is a local parameter (or equivalently, that $\alpha$ is primitive) eliminates
the $\QQ_+^\times$-ambiguity. The $U_D^2$-ambiguity corresponds to the fact
that this  group (considered as diagonal matrices in the Hilbert modular
group) stabilizes the given cusp $\alpha$ of the Hilbert modular surface. 
Since $U_D^2$ acts transitively on the three non-trivial cosets 
$\zeta$ of $\tfrac12 \fraca/\fraca$  and since a multiminimizer has the multiple
minimum property on exactly one of the three cosets by 
Proposition~\ref{prop:MMchar1}, 
we may suppose from now on that the branch is chosen such that 
$D_2\theta  \TchiP{m}{m'}{\bfom}$ for $m = (\tfrac12, \tfrac12)$ vanishes on that branch.
\par
We next determine $\alpha$ and exhibit for this purpose a convenient 
basis of $\fraca^2$. We suppose that $\omega_1, \omega_2$ was from the
beginning of the discussion the distinguished
basis of $\fraca$ associated with a multiminimizer in Proposition~\ref{prop:MMchar2}.
We let $g= {\rm gcd}(a,c)$, where $ax^2 + bx + c=0$ is the minimal polynomial
of $\lambda =\omega_2/\omega_1$. Choose $s,t \in \ZZ$ such that $sa + tc = g$.
It is easily verified that
$$\fraca^2 \= \langle \omega_1^2, \omega_1\omega_2, \omega_2^2 \rangle 
\= \langle \alpha^*, \beta^* \rangle,$$
where
$$\alpha^* \= \frac ga \omega_1^2 + \frac{tb}{a} \omega_1\omega_2 \quad
\text{and} \quad
\beta^* \= \frac 1g \omega_1\omega_2$$
since ${\rm gcd}(a,b,c)=1$. (In fact, $\fraca^2$ contains 
$\tfrac ca \omega_1^2  + \tfrac ba \omega_1 \omega_2$ using the minimal
polynomial, so $\tfrac bg \omega_1\omega_2$ and finally 
$\tfrac 1g \omega_1\omega_2$ by the gcd condition. The ideal also contains 
$\frac {ct}a \omega_1^2 + \frac{tb}{a} \omega_1\omega_2$ and, since
${\rm gcd}(a,tc) = g$, also $\alpha^*$. The converse inclusion follows from
the line below~\eqref{eq:defFG}.) This basis is chosen such that 
the dual basis is $\{\alpha, \beta\}$ where
\be \label{eq:alphaoverg}
\alpha \= \frac{-ac}{g\omega_1^2 \la \sqrt{D}} \= \frac{a}{g\omega_1^2}\cdot\frac{b+\sqrt{D}}{2\sqrt{D}}\,.
\ee
As local coordinates on the Hilbert modular surface we now pick 
$X = q_1^{\alpha^*}q_2^{\sigma(\alpha^*)}$ and $Y =  q_1^{\beta^*}q_2^{\sigma(\beta^*)}$,
which are power series in the local coordinate~$q$ beginning with $q$ and $S$, respectively.
In these coordinates the factors of $\PTDa$ have the expansion  
\begin{equation} \label{eq:thetafracaFG}
D_2\theta  \TchiP{m}{m'}{\bfom}(z_1,z_2) \= \sum_{x \in \ZZ^2}
(-1)^{2x\cdot(m')^T} \rho_\bfom(x)^\s\,X^{F(\tx_1, \tx_2)/2} Y^{G(\tx_1, \tx_2)/2}  \, , 
\end{equation}
where $\tx_i = x_i + \h$ and where 
\be \label{eq:defFG}
F \= \Bigl[\frac ag, 0, \frac {-c}g \Bigr], \quad 
G \= \Bigl[-bt , 2g, bs \Bigr]\,.
\ee
(This follows from $\tfrac ag \alpha^* -bt \beta^* = \omega_1^2$, 
$2g \,\beta^* = 2\omega_1\omega_2$ and $-\tfrac cg \alpha^* - bs \beta^* 
= \omega_2^2$.) 
Then $D_2\theta  \TchiP{m}{m'}{\bfom}/X^{1/8}Y^{1/8}$ has an
expansion in integral powers of $X$ and $Y^{1/2}$, but we actually need the product 
$\cFF=D_2\theta\TchiP{(\h,\h)}{(\h,0)}{\bfom}D_2\theta\TchiP{(\h,\h)}{(0,\h)}{\bfom}$,
and this product, divided by $X^{1/4}Y^{1/4}$, has integral powers of~$Y$ as well.

\par
Finally we need to show that map from cusps to multiminimizers is onto
and that the fibers have cardinality $g$. 
Analyzing the lowest order coefficient in $q$ of~$\cFF$, to which precisely
the summands $\tx_1 = \pm \tfrac 12$ and $\tx_2 = \pm \tfrac 12$
contribute, and noting that $G(\tfrac12, \tfrac12) - G(\tfrac12, -\tfrac12) =g$, 
we find that the $S$ of~\eqref{eq:newS} has to be a solution of
\be  \label{eq:lowesttermeq}
S^g \= \biggl( \frac{\rho_\bfom(\tfrac12, -\tfrac12)}
{\rho_\bfom(\tfrac12, \tfrac12)} \biggr)^\sigma \, .\ee
This equation has precisely $g$ solutions, differing by $g$th roots of unity.
For each such solution of the lowest order term in $q$ there is a unique
power series $\ve(q)$ such that~\eqref{eq:branch} is in the vanishing
locus of $\PTDa$, since the coefficients of $\ve(q)$ are recursively
determined by a triangular system of equations. (We will discuss the arithmetic
properties of $\ve(q)$ in \S\ref{sec:phiarithmetic}.)
\par
This completes the proof for the vanishing locus of $\PTDa$ at 
the cusp $\infty$ and by Lemma~\ref{le:DthetaOtherCusp} also
the proof of Theorem~\ref{thm:cuspsWD}.
\end{proof}
\par
A statement like Theorem~\ref{thm:cuspsWD} can certainly be proven along the same lines
also if $\fraca$ is not an invertible $\O_D$-ideal (which can of course only 
happen for non-fundamental~$D$).  This has an effect in Lemma~\ref{le:DthetaOtherCusp},
and the fact that ${\rm gcd}(a,b,c)>1$ changes the computation of the 
basis~$\{\alpha^*, \beta^*\}$ that was used in the proof. Since our 
aim is just to demonstrate the method of reproving the properties of 
genus two \Teichmuller\ curves using theta functions, we do 
not carry this out in detail.
\par
\begin{proof}[Proof of Theorem~\ref{thm:cuspsPD}]
The reducible locus is the vanishing locus of the product of 
all $10$ even theta functions. Thus branches of the vanishing
locus have to be parametrized as in \eqref{eq:branch} with
$\alpha$ a multiminimizer. For $m = (0,0)$ the forms are never
multiminimizing and in all other cases the proof proceeds as
the proof of Theorem~\ref{thm:cuspsWD}. Only the coefficients of
the equation~\eqref{eq:lowesttermeq} change, but not the exponents.
This does not affect the number of cusps for each multiminimizer.
\end{proof}

\subsection{Cusps of $W_D$  via flat surfaces and prototypes} 
\label{sec:Cuspviaflat}

The cusps of the \Teichmuller\ curve $W_D$ were first determined in~\cite{McM05}
based on the following observations.
It was discovered by Veech along with the definition of \Teichmuller\ curves that
cusps correspond to directions (considered as elements of $\RR^2/\RR^*$) 
of {\em saddle connections} (i.e.\ geodesics for the flat metric $|\omega|$ starting and ending 
at a zero of~$\omega$). Applying a rotation to the flat surface, we may 
suppose that the direction is horizontal, so that the subgroup of $\GL\RR$
stabilizing the direction consists of upper triangular matrices. McMullen discovered 
that flat surfaces parametrized by $W_D$ always decompose in saddle connection 
directions into two cylinders as indicated in Figure~\ref{cap:Prototype}. The action
of the upper triangular group allows us to assume the upper cylinder to be a 
square, while still having the freedom to normalized by the action of  
$T=\left(\begin{smallmatrix} 1 & 1\\0 & 1 \end{smallmatrix}\right)$.
McMullen further deduces from real multiplication that the two cylinders must 
be isogenous with a homothety in a real quadratic field. We thus may suppose that 
$(a,c,q)$ are integral and $\lambda$ is real quadratic. A more careful analysis 
of the isogeny (see~\cite{McM05}, Equation (2.1), or \cite{Ba07}, Proposition~7.20), imposing 
moreover that $\ord_D$ is the exact endomorphism ring of a generic abelian surface 
parametrized by the \Teichmuller\ curve, implies that $\lambda = [a,b,c]$ with $a$ and~$c$ 
as in  Figure~\ref{cap:Prototype} and such that $(a,b,c,r)=1$. 
  \begin{figure}[h]
    \begin{center}  
\begin{tikzpicture}
\draw[thick] (0,0) node(1){} -- ++(60:2cm) node(2){} -- ++(0cm,2cm) node(3){} -- ++(2cm,0cm) node(4){} -- ++(0,-2cm)
node(5){} -- ++(1cm,0cm) node(6){} -- ++(-120:2cm) node(7){} -- ++(-1cm,0cm) node(8){} -- cycle;
\draw[thick,dashed] (2) -- (5) -- (8);
\foreach \x in {1,...,8}
\node[draw=black] at (\x) {};
\foreach \from/\to/\name in {1/2/1,2/3/2}
  \path (\from) -- (\to) node[midway,left] {\name};
\foreach \from/\to/\name in {3/4/3,5/6/4}
  \path (\from) -- (\to) node[midway,above] {\name};
\foreach \from/\to/\name in {6/7/1}
  \path (\from) -- (\to) node[midway,right] {\name};
\foreach \from/\to/\name in {1/8/3,8/7/4}
  \path (\from) -- (\to) node[midway,below] {\name};
\draw (-.8,0) node {$(0,0)$};
\draw (4,0) node {$(-c,0)$};
\fill (2,2.73) circle (1.5pt);
\fill (3,2.73) circle (1.5pt); 
\draw (3.2,2.73) node {2};
\fill (1.5,.87) circle (1.5pt);
\fill (3,.87) circle (1.5pt);
\fill (3.5,1.73) circle (1.5pt);
\draw [thick,decorate,decoration={brace,amplitude=10},xshift=-.2cm,yshift=0cm] (1,1.73) -- (1,3.72)
 node [black,midway, left, xshift=-0.3cm, yshift=0cm] {$-a\lambda^\sigma$};
\node at (0.2,1.73) {$(r,a)$};
\draw [thick,decorate,decoration={brace,amplitude=10},xshift=0cm,yshift=.2cm] (1,3.72) -- (3,3.72)
 node [black,midway, above, xshift=0cm, yshift=0.4cm] {$-a\lambda^\sigma$};
\end{tikzpicture}
    \caption{Prototype of a flat surface parametrizing cusps of the \Teichmuller\ curves $W_{D}$. Sides with the same label are identified.} 
    \label{cap:Prototype}
    \end{center}
  \end{figure}
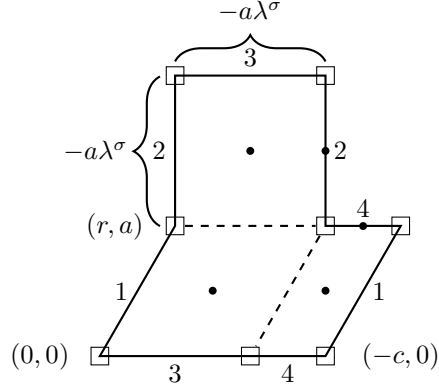 
\par
Obviously $a>0$, $c<0$ and the side length inequality $-a\lambda^\sigma < -c$ is equivalent
to $a+b+c<0$. By the action of $T$ we may reduce mod ${\rm gcd}(a,c)$.
Altogether this implies that cusps of $W_D$ correspond bijectively to pairs
$(\lambda=[a,b,c],q)$ with $\lambda$~\Pred\  
and $q \in \ZZ/{\rm (a,c)}$, as in Theorem~\ref{thm:cuspsWD}.
\par

\subsection{Computing multiminimizers} \label{sec:compMMM}
There are many variants of how to expand real numbers into continued
fractions. One may subtract (slowly) one at each step or group these
steps together, one may subtract one while being greater than one, 
greater than zero, or until becoming negative. We will need two of them, one
here and one in \S\ref{sub:HirzComp}.
\par
The {\em slow plus greater-than-one continued fraction expansion} of $x \in \RR_{>1}$ is defined
by $x_0 = x$ and then inductively by 
$$ x_{n+1} = \left\{ \begin{array}{ccc} x_n -1 \quad & \text{if}& x_n > 2 \\
1/(x_n-1) \quad &  \text{if}& 2 > x_n > 1 \\ \end{array} \right., $$
so that
$$x  = 1+\cdots+1+\left(1 + \frac{1}{1+\cdots+1+\left(1 + 
\frac{1}{\ddots}\right)}\right).$$
\par 
The class of the ideal $[\fraca] = [\langle 1,x_n \rangle]$ is unchanged under this continued
fraction algorithm, since $\langle 1,x_n \rangle = \langle 1,x_n-1 \rangle= (x_n-1) \langle 1, \tfrac1{x_n-1} \rangle$. 
\par
\begin{Lemma}\label{slowLag}  The slow plus greater-than-one continued fraction expansion of any quadratic 
irrational~$\lambda$ is periodic, and is pure periodic if and only if $\lambda$ is \Pred. 
\end{Lemma}
\par
\begin{proof} The classical proof of Lagrange (for the more usual fast continued fraction 
algorithm with plus sign, where $\la$ is sent to~$\la-\lfloor\la\rfloor$ if $\la>1$ and 
to $1/\lambda$ if $0<\lambda<1$) applies here as well. 
The slow algorithm here just introduces intermediate steps and combines two 
steps when $1<\lambda <2$. The second statement is proven in~\cite{ChZa93}.
\end{proof}
\par
This observation gives us an algorithm to compute all multiminimizers for a given
ideal $\fraca$. Write $\fraca = \mu_0 \langle \lambda_0,1 \rangle$ and apply the
continued fraction algorithm to $\lambda_0$ until it becomes pure periodic, say at
a \Pred\ quadratic irrational $\lambda_1$ with  $\fraca = \mu_1 \langle \lambda_1,1 \rangle$.
The first multiminimizer is then the $U_D^2\cdot \QQ^*$-class of 
$\alpha_1 = \tfrac{-1}{\mu_1^2\lambda_1\sqrt{D}}$ as in~\eqref{eq:def_alphaMM}.
Proceeding in this way with $\alpha_k = \tfrac{-1}{\mu_k^2\lambda_k\sqrt{D}}$ 
where $\fraca = \mu_k \langle \lambda_k,1 \rangle$
for $\lambda_2,\ldots,\lambda_n$  along the period of $\lambda_1$ gives 
all the multiminimizers for $\fraca$.
\par
In each of the steps we may choose $\alpha$ in its $\QQ^*$-class to be primitive in
some fixed $\ord_D$-module. We will use this normalization for the
Bainbridge compactification in the next section.

\section{The Hirzebruch and the Bainbridge compactification} \label{sec:HirzBain}

Hirzebruch's minimal smooth compactification is a toroidal compactification.
In this section we reinterpret the Bainbridge compactification of
Hilbert modular surfaces as a toroidal compactification, 
see Theorem~\ref{thm:bainviamultimin}.
This compactification was defined originally via the Deligne-Mumford
compactification of the moduli space of one-forms on Riemann surfaces, 
We will see that both Hirzebruch's and Bainbridge's 
compactification can be computed using continued fraction algorithms.
\par
We start with a review of toroidal compactifications from the viewpoint
of curve degenerations and we recall Hirzebruch's continued fraction
algorithm. The procedure presented in \S\ref{sec:compMMM}
to compute multiminimizers is very similar to Hirzebruch's, but the continued
fraction algorithms are different. The geometry of Bainbridge's compactification 
has been described in~\cite{Ba07}. We recall this material in \S\ref{sec:propbain}, 
since it is part of the main Theorem~\ref{thm:bainviamultimin}. Once
this is proven, all geometric properties will follow from general facts
about toroidal compactifications. (See Proposition~\ref{prop:TorGeneral}.)
The proof of the main theorem will be given in \S\ref{sec:proofBviaMM}.
\par
The description of 
toroidal compactifications will make it easy to determine (in~\S\ref{sec:intFnWDcomp})
where the cusps of modular curves $F_N$ or \Teichmuller\ curves $W_D$
intersect the cusp resolution cycle both of the Hirzebruch compactification
and the Bainbridge compactification.
Finally, in \S\ref{sec:RelBainHirz} we compare the two continued fraction
algorithms governing the two compactifications and give formulas 
for the total lengths or the cusp resolution cycles in both cases.

\subsection{Toroidal compactifications } \label{sec:toroidal}

\changed{The reader may consult the textbook by Ash-Mumford-Rapoport-Tai (\cite{amrt})
as general reference on toroidal compactifications. We give a self-contained
treatment for the case of Hilbert modular surfaces.}
Locally near $(\infty,\infty) \in \HH^2$ a Hilbert modular surfaces is $\HH^2/G(M,V)$, 
where $M$ is a complete submodule of $K$ (i.e., an additive subgroup of $K$ that 
is free abelian of rank two) and $G(M,V) \subset \SLOD$ is the semidirect product 
of $M$ and some subgroup $V$ of totally positive units (see e.g.~\cite{vG87}).
\par
We study the limiting behavior of a complex curve $C$ in $\HH^2$ 
parametrized by $\tau \in \HH$ and 
having the asymptotic form
\begin{equation} \label{eq:curveasym}
 \begin{aligned}
z_1 & \= \,\gamma\,\tau + A_0 + A_1q + A_2 q^2 + \cdots \\
z_2 & \= \gamma^\sigma \tau + B_0 + B_1q + B_2 q^2 + \cdots \\
\end{aligned}
\end{equation}
for $\Im(\tau) \to \infty$, where $\gamma \in K$ is totally
positive, $q = \e(\tau)$ and $A_i, B_i \in \CC$. If this curve
descends to an algebraic curve in the Hilbert modular surface $X_D$,
so that the intersection with the subgroup $M$ of the cusp stabilizer
is not trivial, we can always assume that the leading coefficients
are of that form. \changed{In fact, such a curve is always given
as (part of) the vanishing locus of some Hilbert modular form. Looking
at the $q$-expansion of this Hilbert modular form, we deduce~\eqref{eq:curveasym}
by the same argument as given at the beginning of the proof of 
Theorem~\ref{thm:cuspsWD}.}
\par
To each totally positive $\alpha \in M$ we associate a copy
$\PP^1_\alpha$ of $\PP^1(\CC)$. We have a map
$$X_\alpha: \left\{ \begin{array}{lcl}
\HH^2 / M & \to & \CC^*_\alpha \subset \PP^1_\alpha \\
(z_1,z_2) & \mapsto & \e\left(\tr(\alpha^\vee \ul{z}) \right)  = 
\e\left(\frac{\alpha^\sigma z_1 - \alpha z_2}{N(M)\sqrt{D}}) \right)  \\
\end{array} \right.\,,
$$
where $N(M)$ denotes the norm of $M$. Using this identification, we glue $\PP^1_\alpha$ at the
cusp $(\infty,\infty)$ of $\HH^2$ topologically as follows. 
A sequence of points $\ul{z}= (z_1,z_2) \in \HH^2$ with imaginary
parts $y_1, y_2$ both tending to $\infty$ with limiting value of
$y_1/y_2$ being equal to $t \in \RR_+$ converges to the point $0 \in 
\PP^1_\alpha$ if $t < \alpha/\alpha^\sigma$, to the point $\infty \in 
\PP^1_\alpha$ if $t > \alpha/\alpha^\sigma$ and to a finite point
$X= \e(\theta/N(M)\sqrt{D})$ for some $\theta \in \CC$ if $t = \alpha/\alpha^\sigma$ and if
$\alpha^\sigma z_1-\alpha z_2 = \theta + o(1)$.
\par
Consequently, the curve $C$ meets $\PP^1_\alpha$ at $\infty$ if $\tr(\alpha^\vee \gamma) < 0$,
at $0$ if $\tr(\alpha^\vee \gamma) > 0$, and at a finite, non-zero, point if $\gamma$ is a rational multiple of $\alpha$. 
\par
To any oriented $\QQ$-basis $(\alpha, \beta)$, i.e.\ with
$ \frac{\alpha}{\alpha^\sigma} < \frac{\beta}{\beta^\sigma}$, corresponds
a pair of projective lines $\PP^1_\alpha$ and $\PP^1_\beta$
meeting at one point $(\infty,0) \in \PP^1_\alpha \times \PP^1_\beta$.
Suppose that  $\gamma \in K$ has the property
$$ \frac{\alpha}{\alpha^\sigma} < \frac{\gamma}{\gamma^\sigma} < \frac{\beta}{\beta^\sigma}$$ 
holds, so that the curve $C$ passes through the point $(\infty,0)$. If we 
write $\gamma = p\alpha + q\beta$ with $p,q \in \QQ_+$
then, near the point $(\infty,0)$ the curve \eqref{eq:curveasym} looks like
$$ (X_\alpha^{-1})^p = (X_\beta)^q.$$
\par
We will, of course,  apply this in particular to the curves
$F_N$ and to \Teichmuller\ curves.
\par
A $V$-invariant partial compactification $\HH^2 / M$ of a Hilbert modular 
surface is defined by adding  not just one $\PP^1_\alpha$, but 
an appropriate sequence of $\PP^1_\alpha$'s.
\par
\begin{Defi} \label{def:fan}
A sequence of numbers $(\alpha_n)_{n\in\NN}$ in 
$M$ forms a {\em fan} if i) the $\alpha_n$ are all totally positive, ii)
$V\cdot\{\alpha_n\}_{n\in\NN}=\{\alpha_n\}_{n\in\NN}$ and iii) 
the ratios $\alpha_n^\sigma/\alpha_n$ are a strictly decreasing sequence.
\end{Defi}
\par
Our definition includes that the fan is $V$-invariant. If $V = \langle
\v \rangle$, then ii) is equivalent to the existence of some $k$ such
that $\v \alpha_n = \alpha_{n+k}$. The minimal positive $k$ with 
that property will be called the {\em length} of the fan. 
\par
\begin{Prop} \label{prop:TorGeneral}
A fan $(\alpha_n)_{n\in\NN}$ determines a partial compactification 
$\overline{\HH^2/ M}$ of $\HH^2 / M$ with the following
properties. For each $n$ there is an irreducible curve $\PP^1_{\alpha_n}$ in 
$\overline{\HH^2  / M} \ssm \HH^2 / M$. The curves
$\PP^1_{\alpha_n}$ and $\PP^1_{\alpha_{n+1}}$ intersect in one point. 
For $k>1$ the curves $\PP^1_{\alpha_n}$ and $\PP^1_{\alpha_{n+k}}$ 
are disjoint in  $\HH^2 / M$.
\par
The action of $V$ on $\HH^2 / M$ extends to an action on
$\overline{\HH^2  / M}$. 
Hence a fan determines a partial compactification
$\overline{\HH^2/G(M,V)}$ of $\HH^2/G(M,V)$.
\par
The partial compactification is always smooth along $\PP^1_{\alpha_n}$ 
minus the intersection points with the $\PP^1_{\alpha_{n\pm1}}$. 
At the intersection point of $\PP^1_{\alpha_n}$ and
$\PP^1_{\alpha_{n+1}}$ the compactification is smooth if and only 
$M = \ZZ \alpha_n + \ZZ \alpha_{n+1}$. More generally, if 
$\ZZ \alpha_n + \ZZ \alpha_{n+1}$ has index $k$
in $M$, then $\overline{\HH^2  / M}$ has a cyclic quotient singularity of
order $k$ at that point. In particular, the compactification is normal.
\end{Prop}
\par
\begin{proof} We first recall the general setup of two-dimensional
toric varieties, \changed{e.g.\ the book of Fulton \cite{Fu93} provides
an introduction to this topic.} For each two-dimensional cone $\sigma$ 
given, say,  as the span of $\alpha_n$ and $\alpha_{n+1}$ we let $U_{\sigma}$ be the 
variety with
coordinate ring $\CC[\sigma^\vee \cap M^\vee]$. For each
one-dimensional cone $\tau$ at the boundary of $\sigma_2$, 
say spanned by $\alpha_n$ we let $V_{\tau} \subset U_{\sigma}$ be 
the variety with coordinate ring $\CC[\sigma^\vee \cap M^\vee]$.
If $\tau$ is contained both in $\sigma_1$ and $\sigma_2$ we
may glue $U_{\sigma_1}$ and $U_{\sigma_2}$ along the open
set $V_\tau$. In particular $a_n$ (or $\tau = \langle a_n \rangle$)
determines a rational curve $\PP^1_{\alpha_n} = 
(U_{\sigma_1} \cup U_{\sigma_2}) \ssm
V_\tau$. The zero-dimensional cone ${0}$ corresponds to the
variety $W$ with coordinate ring $\CC[M^\vee]$ and sits as open part
in all the $V_\tau$ and $U_\sigma$. Here $M^\vee$ is the dual $\ZZ$-module, which
we will identify from now on as a submodule of~$K$ using the trace pairing.
\par
In our situation we want to identify $W$ with $\CC^2 / M$.
On the level of local coordinate rings this is done by assigning to
$b \in M^\vee$ the coordinate function 
$$X^\vee_b(z_1,z_2) = \e(bz_1 + b^\sigma z_2).$$ 
(This is the same coordinate as $X_{\alpha^\vee}$ associated above with
$\alpha \in M$, using the identification of $M$ with $M^\vee$ via
$\alpha \mapsto \alpha^\vee = \frac{\alpha^\sigma}{N(M)\sqrt{D}}$. We
refer to $\alpha^\vee$ as the {\em trace dual} of $\alpha$.)
Since  $\HH^2 / M$ sits inside $\CC^2 / M$, the partial compactification of $W$
by the $U_\sigma$ defines the desired partial compactification of $\HH^2 / M$.
\par
Given a cone $\tau$ generated by $\alpha_n$, the element 
\changed{$\alpha_n^\vee$} is the unique (up to sign)
primitive element of $M^\vee$ with $\tr(\alpha_n \alpha_n^\vee)=0$. 
If we complete $\alpha_n^\vee$ to a basis of $M^\vee$ using $\beta_n$, whose
sign we may choose such that $\tr(\alpha_n^\vee \beta_n) > 0$, then
the curve $\{X^\vee_{\beta_n}=0\}$ is independent of the choice of $\beta_n$
and coincides with the curve $\PP^1_{\alpha_n}$ defined in the
text preceding the proposition. 
\par
Given two consecutive elements $\alpha_n$ and $\alpha_{n+1}$ of the
fan, the monotonicity of the ratios implies that 
$$ \tr(\alpha_n^\vee \alpha_{n+1}) < 0 \quad \text{and}
\quad  \tr(\alpha_{n+1}^\vee \alpha_{n}) > 0. $$
Hence $\alpha_n^\vee$ and $-\alpha_{n+1}^\vee$ can play the role
of $\beta_n$ above. Consequently, the intersection point of 
$\PP^1_{\alpha_n}$ and $\PP^1_{\alpha_{n+1}}$ is the point
$\{X^\vee_{\alpha_n^\vee}=0, (X^\vee_{\alpha_{n+1}})^{-1}=0 \}$. Since
$X_{\alpha_j}$ is a coordinate on $\PP^1_{\alpha_j}$,
we retrieve that the intersection 
point of $\PP^1_{\alpha_n}$ and $\PP^1_{\alpha_{n+1}}$
is $(\infty,0)$ in that coordinate system.
\par
The disjointness of $\PP^1_{\alpha_n}$ and $\PP^1_{\alpha_{n+k}}$ 
is obvious from the 
gluing procedure. The statement on the $V$-action is an obvious 
consequence of the $V$-invariance of a fan.
\par
The singularity statements are described in detail in \cite{Fu93}, Section 2.2.
\end{proof}
Replacing the $\QQ$-basis $(\alpha, \beta)$ by $(\alpha + \beta, \beta)$
resp.\ by $(\alpha, \alpha+\beta)$ corresponds to performing a 
sigma-transformation (or blowup) defined by the coordinate changes
$$ (X_\alpha^{-1}, X_\beta) \mapsto  (X_\alpha^{-1}X_\beta^{-1}, X_\beta)
\quad \text{resp.} \quad (X_\alpha^{-1}, X_\beta) \mapsto 
(X_\alpha^{-1}, X_\alpha X_\beta) $$
at that point. It follows that the partial compactifications defined by any two
fans are related by repeated blowing up and blowing down.
\par
\subsection{Hirzebruch's compactification} \label{sub:HirzComp}
We describe the fan of Hirzebruch's compactification. There are many
detailed expositions of this, in particular \cite{Hi73} and Chapter~II of~\cite{vG87}.
The aim is to compare the fans of Hirzebruch and Bainbridge below. 
We will use the letters $A_k$ for Hirzebruch's fan (which is 
consistent with~\cite{vG87}) and subscripts $k$ for the indexing
that produces an increasing sequences of slopes. (Thus if we replace $A_n$ by $A_{k_0-n}$ 
for some $k_0$ we fit exactly the Definition~\ref{def:fan}.)
\par
Suppose we want to compactify the Hilbert modular surface $X_D$ at the 
cusp $\fraca$ or equivalently the cusp at $\infty$ in $X_{D,\fraca}$.
Then, in the notation of \S\ref{sec:toroidal} the module 
$M = \fraca^\vee(\fraca^{-1})= (\fraca^2)^\vee$.
We choose the $A_k$ to be the set of extreme points
in $M^+ = M \cap (\RR_+)^2$, i.e.\ the points lying on the
convex hull ${\rm conv}(M^+)$ of $M^+$ in $(\RR_+)^2$, indexed
by increasing slope then the $A_k$ form a fan 
since ${\rm conv}(M^+)$
is $V$-invariant. This compactification is smooth, because any two adjacent points on the boundary
of ${\rm conv}(M^+)$ form a $\ZZ$-basis of $M$ (\cite{Hi73} or
\cite{vG87}, II)~Lemma~2.1). Since the $A_k$ lie on the boundary of the convex hull, we can write
\begin{equation}\label{eq:Basis}
A_{k-1} \+ A_{k+1} \= p_k A_k \quad \text{with} 
\quad p_k \in \ZZ,\;p\ge2\,. \end{equation}
\par
On this smooth compactification the self intersection number of
$\PP^1_{A_k}$ is $-p_k$ (\cite{vG87}, \S~II.2). Consequently,
the compactification using at a cusp with stabilizer $G(M,V)$ 
the fan given by the boundary points of  
${\rm conv}(M^+)$ is the minimal smooth compactification.
\par
We call $x$ {\em reduced} \changed{(or a {\em reduced quadratic irrationality})} 
(see \cite{vG87}, Section 2.5 or \cite{Z81}) if $x$ is real quadratic and 
\be  \label{eq:defreduced}
x > 1 > x' >0. 
\ee
The {\em fast-minus continued fraction algorithm} of $x \in \RR>1$ is
defined by $x_0 = x$ and then inductively by 
\be \label{eq:deffmcf}
x_{k+1} \= 1/(p_k -x_{k})\,, \quad \changed{\text{where} \quad p_k \= \lceil x_k \rceil\,,}
\ee
so that 
$$ x_k \= p_k \,-\, \cfrac{1}{p_{k+1} - \cfrac{1}{\ddots}}\quad. $$
Note that the narrow class of the ideal $[M] = [\langle 1,x \rangle]$ is unchanged
under this continued fraction algorithm, since $\langle 1,x \rangle = (p-x)\langle 1,\tfrac1{p-x} \rangle$.  
The following lemma is the analogue of Lemma~\ref{slowLag} above.
\par
\begin{Lemma} The fast minus continued fraction expansion of any quadratic 
irrational~$x$ is periodic, and is pure periodic if and only if $x$ is reduced. 
\end{Lemma}
\par
\begin{proof} This is well known and is stated, for instance, in \S2.5 of~\cite{Hi73},
where the relation of the fast plus continued fraction expansion
and the fast minus continued fraction expansion is also given.
\end{proof}
\par
This observation gives us an algorithm to compute the convex hull.
Write $M = \mu_0\langle 1,x_0\rangle$ and apply the continued 
fraction algorithm until it becomes pure periodic, say at a reduced
quadratic irrational $x_1$ with  $M = \mu_1 \langle 1,x_1\rangle$.
Take $A_0 = \mu_1$ and $A_1 = \mu_1 x_1$ and then $A_k = A_{k-1}/x_k$
for $k \geq 2$, where $x_1,x_2,\ldots,x_n$ form the 
cycle of the continued fraction algorithm. The recursive definition
of the $x_k$ in~\eqref{eq:deffmcf} then is equivalent to~\eqref{eq:Basis}.
\par
\par
{\bf Conclusion:}  The Hirzebruch compactification is given by
the fan stemming from the lower convex hull or, equivalently,
triggered by the fast minus continued fraction algorithm.
\par

\subsection{Bainbridge's  compactification} \label{sec:bainconstr}

Bainbridge's compactification of a Hilbert modular surface is
defined using the Deligne-Mumford compactification of the moduli space
of curves of genus two. For the details we recall facts about
various bundles of one-forms over moduli spaces.
\par
In Section \ref{sec:defTeich} we introduced together with \Teichmuller\ curves
the vector bundle $\Omega \M_g$ of holomorphic one-forms over the
moduli space  $\M_g$ of curves of genus $g$. The moduli space of curves comes
with the Deligne-Mumford compactification by stable curves 
$\overline{\M_g}$ and the vector bundle  $\Omega \M_g$ extends
to a vector bundle $\overline{\Omega \M_g}$, whose sections are
{\em stable forms}. A stable form is a differential form on the
normalization of the stable curves, holomorphic except for at
most simple poles at the preimages of nodes and such that the
residues at the two branches of a node add up to zero. We refer to
$\overline{\Omega \M_g}$ and to the corresponding projective bundle
$\PP\overline{\Omega \M_g}$ as the {\em Deligne-Mumford compactification}
of $\Omega \M_g$ (resp.\ of $\PP\Omega \M_g$).
\par
Strictly contained between $\M_g$ and $\overline \M_g$ is the partial compactification
$\widetilde{\M_g}$ of stable curves of compact type, i.e.\ stable
curves whose Jacobian is compact or equivalently of arithmetic genus
$g$.
\par
For $g=2$ the Torelli map $t: \widetilde{\M_2} \to \cAA_2$ is an
isomorphism. It extends to an isomorphism of the bundles of
stable one-forms $t: \widetilde{\Omega \M_2} \to \Omega \cAA_2$
and also to the projectivized bundles 
$t: \PP \widetilde{\Omega \M_2} \to \PP\Omega \cAA_2$.
\par
On the other hand, we have seen in Section~\ref{sec:HMGHMS} that a Hilbert modular
surface $X_D$ parametrizes principally polarized abelian varieties 
with real multiplication. There is a unique holomorphic one-form
on such an abelian variety that is an eigenform for the action
of real multiplication (with the embedding $K \to \RR$ that
we fixed throughout). The quotient map of a Siegel modular embedding by
the action of $\SL{\ord_D}$ defines a map $X_D \to \cAA_2$
and the choice of an eigenform lifts this map to an injection
$$ \psi: X_D \to  \PP\Omega \cAA_2. $$ 
We thus use the same letter for this map as for the modular embedding. The 
image $t^{-1}(\psi(X_D))$ is called the {\em eigenform locus}
(maybe the {\em projectivized eigenform locus} would be more precise).
It parametrizes stable curves of genus two of compact type with
real multiplication by $\ord_D$.  We denote by $\overline{X_D}^{DM}$ the 
closure of the eigenform locus in the Deligne-Mumford
compactification $\PP \overline{\Omega \M_2}$.
\par 
\begin{Defi}
The {\em Bainbridge compactification} {\em (called the {\em geometric  
compactification} in~\cite{Ba07})} $\overline{X_D}^B$  of 
a Hilbert modular surface $X_D$ 
is the normalization of~$\overline{X_D}^{DM}$. 
\end{Defi}
\par
We now identify this compactification as a toroidal compactification. We
restrict to the case $D$ fundamental for simplicity and since the preparations
in Section~\ref{sec:cuspTheta0} have been carried out for this case only.
In~\S\ref{sec:proofBviaMM} we will prove:
\begin{Thm} \label{thm:bainviamultimin}
Suppose that $D$ is a fundamental discriminant.
For each cusp given by the ideal class $\fraca$ the sequence of multiminimizers for $\fraca$, as  
derived from the continued fraction algorithm in \S~\ref{sec:compMMM}, forms a fan. 
The Bainbridge compactification is the toroidal compactification 
of the Hilbert modular surface obtained by using this sequence of multiminimizers.
\end{Thm}
\par
Comparing with \S\ref{sub:HirzComp}, we can summarize this theorem and
the algorithm in \S\ref{sec:compMMM} as follows.
\par
{\bf Conclusion:}  The Bainbridge compactification is given by
the fan stemming from the multiminimizers or, equivalently,
triggered by the slow plus greater-than-one  continued fraction algorithm.
\par
In \S\ref{sec:propbain} we review the properties of the compactification
by Bainbridge. For the proof of Theorem~\ref{thm:bainviamultimin} we will need only part of 
these properties. The local structure at the intersection points of the compactification curves
is forced by normality and could also be derived from Proposition~\ref{prop:TorGeneral}.
 
\subsection{Period coordinates and properties of the Bainbridge compactification.} 
\label{sec:propbain}
In order to identify the Bainbridge compactification and
to prove Theorem~\ref{thm:bainviamultimin} we use the part
of the work of Bainbridge, where he gives a coordinate
system of $\overline{\Omega \M_2}$ and describes $\overline{X_D}^{DM}$ 
in there. He uses a lift by choosing a scaling of the one-form $\omega$ 
by fixing locally a loop $\alpha_1$ and by imposing that $\int_{\alpha_1} \omega 
= r_1$ for some $r_1 \in K$.
We will then compare the coordinates introduced for toroidal compactifications to 
this coordinate system and thereby prove the claimed isomorphism.
\par
There are two relevant types of coordinate systems, both called  
{\em period coordinates} depending on the type of the stable curves. 
We follow \cite{Ba07}, Section~6.6. If $(X,\omega)$ is
in the boundary of  $\overline{X_D}^{DM}$, then $g(X)$ is zero.
The first type of coordinate system is around a stable
curve $X$ with two non-separating nodes. We moreover suppose that $\omega$ has two
simple zeros. We let $\pm r_1$
and $\pm r_2$ be the residues of $\omega$. 
Choose loops $\alpha_1,\alpha_2$ around the punctures such that
$\int_{\alpha_i} \omega = r_i$. For a smooth surfaces in a neighborhood
choose loops $\beta_1,\beta_2$ that complete the $\alpha_i$-curves
to a symplectic basis. Finally, choose a path $I$ joining the two zeros
of $\omega$. Then on a neighborhood of $(X,\omega)$ in $\Omega \M_2$ the 
functions 
\begin{flalign} \label{eq:per_edge}
v^{(E)} &\= \int_{\alpha_1} \omega, \quad \quad w^{(E)} \= \int_{\alpha_2} \omega, \\
y^{(E)} &\= \e\left(\int_{\beta_1}
\omega/v\right),  \quad z^{(E)} \= \e\left(\int_{\beta_2}
\omega/w\right), \quad x^{(E)} \= \int_I \omega \nonumber
\end{flalign}
are well-defined, i.e.\ independent of ambiguity in the choice of $\beta_i$
given by Dehn twists around the corresponding $\alpha_i$. 
These five functions form a system of coordinates on $\Omega \M_2$.
We provide them with a superscript $E$ (edge) to distinguish them. They
will correspond to edges of the boundary of  $\overline{X_D}^{DM}$.
Note that $x=x^{(E)}$ is only well-defined up to an additive constant depending on
the path of integration and that its sign depends on the choice of
ordering of the two zeros. 
\par 
From the geometry of the Deligne-Mumford compactification Bainbridge derives the following 
proposition that we use to prove Theorem~\ref{thm:bainviamultimin}.
\par
\begin{Prop}[{\cite{Ba07}, Proposition~7.18, Theorem~7.17 and Theorem~7.22}] 
\label{prop:fromBain}
A pair $(r_1,r_2)$ appears as residues of an eigenform of $X_D$ 
if and only if $\lambda = r_2/r_1 \in \QQ(\sqrt{D})$ and $N(\lambda) < 0$. More precisely,
the irreducible components of the boundary of $\overline{X_D}^{DM}$
are in bijection with the unordered projective tuples $(r_1:r_2)$,
or equivalently to \Pred\ quadratic irrationals $\lambda$ in~$\ord_D$.
\par
Near a boundary component labeled by $\lambda$, the Hilbert modular
surface $\overline{X_D}^{DM}$ is cut out in $\Omega \M_g$ by the equations
\begin{equation} \label{eq:EigLocal}
 v^{(E)} \=r_1, \quad w^{(E)} \= r_2 \quad \text{and} \quad (y^{(E)})^a \= (z^{(E)})^{-c}.
\end{equation}
and the boundary curve is given in these coordinates as $ \{y^{(E)}=z^{(E)}=0\}.$
\par
The boundary of $\overline{X_D}^B$ is a union of rational curves
$C_\lambda$ where $\lambda=r_1/r_2$.
\end{Prop}
\par
We explain the last statement. If ${\rm gcd}(a,c)>1$, then~\eqref{eq:EigLocal} shows 
that the compactification is not normal near $\{y^{(E)}=z^{(E)}=0\}\,$: the normalization
has ${\rm gcd}(a,c)$ local branches. Nevertheless, the preimage of $\{y^{(E)}=z^{(E)}=0\}$ 
in the normalization $\overline{X_D}^B$  is a connected curve $C_\lambda$. The normalization
map is a cyclic covering of order ${\rm gcd}(a,c)$ ramified precisely over the
intersection points of $C_\lambda$ with its two adjacent curves in the cusp
resolution.
\par
We include the following proposition for a complete description of the
Bainbridge compactification. It is not needed for the proof of
Theorem~\ref{thm:bainviamultimin}.
\par
The second type of coordinate system is around a stable
curve $X$ consisting of two irreducible components joined at three 
non-separating nodes. Here we choose $\alpha_1,\alpha_2,\alpha_3$
to be the loops around the punctures and let $r_i = \int_{\alpha_i} \omega$
be the residue. We may orient the $\alpha_i$ so that $r_1-r_2+r_3=0$.
The one-form $\omega$ necessarily has one zero on each of the irreducible
components.  We let $\gamma_i$ be a path joining these two zeros that
crosses $\alpha_i$ once with positive intersection and no other $\alpha_j$.
Then the five functions 
\begin{flalign} 
&v^{(V)} \= \int_{\alpha_1} \omega, \quad \quad w^{(V)} = \int_{\alpha_2} \omega,  
\quad x^{(V)} \= \e\left(\int_{\gamma_3}
\omega/(w^{(V)} -v^{(V)})\right), \nonumber \\
& y^{(V)} \= \e\left(\int_{\gamma_1}
\omega/v^{(V)}\right), \quad z^{(V)} \= \e\left(\int_{\gamma_2}\omega/w^{(V)}\right) 
\end{flalign}
form a system of coordinates on \changed{$\Omega \M_2$} 
near $(X,\omega)$.  We provide them with a 
superscript $V$ (vertex of the compactification cycle) to distinguish them. We also have to say how
the coordinates $v^{(V)},\dots,z^{(V)}$ and $v^{(E)},\dots,z^{(E)}$ are related near a vertex.
Suppose we unpinch the node corresponding to $\alpha_3$. Then on this nearby surface 
$$\beta_1 \= \gamma_1 - \gamma_3, \quad \beta_2 \= \gamma_2 + \gamma_3, \quad I \= \gamma_3.$$
Consequently, we have
$$v^{(E)} \= v^{(V)} =: v\,, \qquad w^{(E)} \= w^{(V)} =: w $$
and
$$ x^{(V)} \= \e\bigl(x^{(E)}/(w-v)\bigr), \quad 
y^{(V)} \= y^{(E)} \cdot \e\bigl(x^{(E)}/v\bigr),  \quad z^{(V)} \= z^{(E)}\cdot \e\bigl(x^{(E)}/w\bigr)\,.$$
\par
The following proposition describes the intersection points of boundary curves
of $\overline{X_D}^B$. It is proven in \cite{Ba07},~Theorem~7.27 but it is also an immediate
consequence of Theorem~\ref{thm:bainviamultimin} and Proposition~\ref{prop:TorGeneral}, 
in particular its last statement.
\par
\begin{Prop}
If $\lambda^+$ denotes the successor of $\lambda$ for the slow plus
greater-than-one continued fraction, then the curves $C_\lambda$ and
$C_{\lambda^+}$ have exactly one point $c_\lambda$ in common.
Near this intersection point the Hilbert modular
surface $\overline{X_D}^B$ is cut out in \changed{$\overline{\Omega \M_g}$} 
by the equations
$$ v^{(V)} \=r_1, \quad w^{(V)} \= r_2 \quad \text{and} 
\quad \bigl(y^{(V)}\bigr)^a \= \bigl(z^{(V)}\bigr)^{-c}\,\bigl(x^{(V)}\bigr)^{-a-b-c}\,.$$
The point $c_\la$ is a cyclic quotient singularity of order
$$m_\lambda =  \frac{a}{\gcd(a,c)\gcd(a,a+b+c)}\, .$$
\end{Prop}
\par

\subsection{\bf The proof of  Theorem~\ref{thm:bainviamultimin}} 
\label{sec:proofBviaMM}
The first step is to show that the 
multiminimizers form a fan. For the cusp $\fraca$
we may choose the normalization as in~\eqref{eq:alphaoverg} and suppose that
the multiminimizers are primitive elements in $(\fraca^2)^\vee$.  We
next examine the slopes. 
Suppose that at some step of the slow plus greater-than-one continued fraction we have
the representation $\fraca = \mu_k \langle \lambda_k,1 \rangle$. Then 
$$ \frac{ \sigma(\alpha_k)}{\alpha_k} \= 
\frac{ \lambda_k \mu_k^2}{\sigma(\lambda_k \mu_k^2) } 
\quad \text{and} \quad 
\frac{ \sigma(\alpha_{k+1})}{\alpha_{k+1}}  \= 
\frac{ (\lambda_k-1) \mu_k^2}{\sigma((\lambda_k-1) \mu_k^2) }
$$
in both cases of the continued fraction algorithm. 
The ratio of these two fractions is 
$$ \frac{ \sigma(\alpha_k)}{\alpha_k} /\frac{ \sigma(\alpha_{k+1})}{\alpha_{k+1}}  = 
\Bigl(1 + \frac1{\lambda_k-1}\Bigr)\Bigl(1-\frac1{\sigma(\lambda_k)}\Bigr) > 1$$
since $\lambda_k$ is \Pred. Consequently the slopes of multiminimizers 
are 
decreasing. Since the continued fraction algorithm is periodic this shows 
that the sequence 
multiminimizers in continued fraction order forms a fan.
We denote the compactification of the Hilbert modular surface using these
fans of multiminimizers by $\overline{X_D}^{MM}$.
\par
We let $\overline{X_D}^{MM,*}$,  $\overline{X_D}^{B,*}$, resp.\  $\overline{X_D}^{DM,*}$
be the complement of the intersection points of the cusps resolution curves
in the three compactifications, i.e.\ obtained by removing the codimension 
two boundary strata. 
Our aim to show that there is a map $m: \overline{X_D}^{MM,*} \to \overline{X_D}^{DM,*}$, 
which is an isomorphism over $X_D$ and which maps the boundary components
labeled by a \Pred\ quadratic irrational $\lambda$ onto the component 
with the same label by an unramified cyclic covering of degree $g = {\rm gcd}(a,c)$. 
Since the Bainbridge compactification is normal, the map $m$ factors through  a map
$\widehat{m}: \overline{X_D}^{MM,*} \to \overline{X_D}^{B,*}$, which is an isomorphism
by the local description of the map $\overline{X_D}^{B,*} \to \overline{X_D}^{DM,*}$.
Since the multiminimizer compactification and the Bainbridge compactification are
normal, the codimension two indeterminacy of $\widehat{m}$ 
(on domain an range) can be resolved to a global isomorphism.
\par
We want to define $m$ in a neighborhood of a point on the component of
the boundary of  $\overline{X_D}^{MM,*}$ given by the multiminimizer $\alpha$ with 
corresponding \Pred\ quadratic irrational $\lambda$. 
Since $\alpha^\vee$ is the unique (up to sign) primitive element in $(\fraca^2)^\vee$ with 
$\tr(\alpha\alpha^\vee) = 0$, local coordinates near this point as 
defined in the proof of Proposition~\ref{prop:TorGeneral} are just
the coordinates $X = q_1^{\alpha^*}q_2^{\sigma(\alpha^*)}$ and 
$Y = q_1^{\beta^*}q_2^{\sigma(\beta^*)}$ used in the proof of Theorem~\ref{thm:cuspsWD}.
\par
The map $m$ is given by assigning to a point in $X_D$ the curve given
as the vanishing locus of the Siegel theta function restricted to $X_D$
together with the first eigenform. In order to understand the local behavior
of this map near the boundary we may choose any convenient translate of
the theta function. As in the previous proofs we take the characteristic
$((\tfrac12,\tfrac12),(\tfrac12,0)$ and the basis $\bfom = (\omega_1, \omega_2)$
of $\fraca$ that is distinguished by the multiminimizer. Moreover, we define elliptic coordinates
$$ S = \e(\tr(\omega_1^\vee \uu)), \quad T=\e(\tr(\omega_2^\vee \uu))\,.$$
In the coordinates $X,\,Y,\,S,\,T$,  the theta function is
\begin{flalign} \label{eq:thetaFG}
& \phantom{=} \theta\TchiP{m}{m'}{\bfom}(z_1,z_2) \= \sum_{x \in \ZZ^2}
(-1)^{2x\cdot(m')^T}  X^{F(\tx_1, \tx_2)/2} Y^{G(\tx_1, \tx_2)/2}S^{x_1}T^{x_2}  \\
&=2 X^{\frac{a+|c|}2}S^{-1/2}T^{-1/2}\left(Y^{G(\tfrac12,\tfrac12)/2}(ST-1)
+ Y^{G(\tfrac12,-\tfrac12)/2}(S-T)+O(X)  \right) \,, \nonumber
\end{flalign}
where $\tx_i = x_i + \tfrac12$ and where 
\be \label{eq:AGAINdefFG}
F \= \bigl[a/g, 0,-c/g \bigr], \quad G \= \bigl[-bt , 2g, -bs \bigr]\,.
\ee
\par
Consequently, since the boundary is given by $X=0$ the vanishing locus of $\theta$
degenerates there. As in \S\ref{UnFa} the limiting curve \changed{is} 
a rational curve.  We abbreviate
$Z = Y^{G(\tfrac12,\tfrac12)/2}/ Y^{G(\tfrac12,-\tfrac12)/2} = Y^{g/2}$,
and obtain (for $Z$ fixed) the equation 
$$Z^{-1}(ST-1) + (S-T) \= 0 \quad \text{or, equivalently} \quad (S-Z)(T+Z)\= 1-Z^2$$
for the rational curve. A parametrization with coordinate $t$ is given by
$S = t+Z$, $T = (1-Z^2)/t - Z$. 
\par
On a boundary component of the genus two Deligne-Mumford compactification
parametrizing rational curves with two nodes, the cross-ratio of the
four points is a coordinate. This is also the relative period $x_E$  that
was used by Bainbridge as coordinate (cf.~Proposition~\ref{prop:fromBain}), 
up to a M\"obius transformation depending on $r_1$ and $r_2$, as one
checks easily be integrating the limiting stable form.
\par
Our last task is therefore to express the cross-ratio in terms of the coordinate~$X$. 
The first eigenform is 
\begin{equation} \label{eq:stableomega}
 \omega \= du_1  \= (\omega_2^\vee)^\sigma \frac{dS}{S} - (\omega_1^\vee)^\sigma \frac{dT}{T}
\= \left(\frac{(\omega_2^\vee)^\sigma}{t + Z} + \frac{(\omega_1^\vee)^\sigma}{t(1-\frac{Z}{1-Z^2}t)} \right)dt.
\end{equation}
in the given rational parametrization.
The cross-ratio of the four poles $0,\infty, -Z,$ and $\tfrac{1-Z^2}{Z}$  is given by
a M\"obius transformation of $Z^{-2} = Y^{{\rm gcd}(a,c)}$. 
This completes the proof of the properties claimed about the map $m$, and 
of Theorem~\ref{thm:bainviamultimin}.
\par
\subsection{\bf Intersection of $F_N$ with Bainbridge's and Hirzebruch's 
boundary.} \label{sec:intFnWDcomp}
We can now answer the question about the intersection of the modular
and \Teichmuller\ curves with the two compactifications. 
For the Hirzebruch compactification the statement in 
Theorem~\ref{thm:bdresintersect} is of course already in the literature.
\par
Fix the cusp to be infinity and let $F_{N}$ be one of the modular 
curves, as defined in \S\ref{sec:redloc}, passing 
through the cusp at infinity. Then $F_{N}$ can be given by an equation
of the form 
\be\label{eqFn} \lambda z_1 + \lambda^\sigma z_2 + B \= 0 \qquad (B \in \ZZ,\;\lambda \in \O,\;
\lambda\lambda^\sigma = -N). \ee
\par
The following result now follows immediately from 
Proposition~\ref{prop:TorGeneral}
and the descriptions of the Hirzebruch and Bainbridge compactifications
given in this section.
\par
\begin{Thm} \label{thm:bdresintersect}
The curve $F_{N}$ passes through an interior point of the cusp resolution
cycle of the Bainbridge compactification (resp.\ of the Hirzebruch
compactification) if there is a multiminimizer $\alpha_n$ for this
cusp such that $\alpha_n /\changed{\lambda} \in \QQ$ (resp.\ if there is an element $A_n$
of Hirzebruch's lower convex hull fan such that $A_n/\changed{\lambda} \in \QQ$). 
\par
Otherwise, if
$$ \frac{\alpha_n}{\alpha^\sigma_n} > \frac{\changed{\lambda}}{\changed{\lambda}^\sigma} 
 > \frac{\alpha_{n+1}}{\alpha^\sigma_{n+1}},
    \quad \text{resp.} \quad 
\frac{A_n}{A^\sigma_n} > \frac{\changed{\lambda}}{\changed{\lambda}^\sigma} > \frac{A_{n+1}}{A^\sigma_{n+1}}, $$
then the curve $F_{N_\mu}$ passes through the node corresponding
to the intersection of the curves associated with $\alpha_n$ and
$\alpha_{n+1}$ (resp.\ with $A_n$ and $A_{n+1}$).
\end{Thm}
\par
This result together with Theorem~\ref{thm:cuspsWD} and 
Theorem~\ref{thm:cuspsPD}
reproves from the theta viewpoint another result of Bainbridge.
\par
\begin{Cor}
The curves $W_D$ and the components of the reducible locus $P_D$ intersect the 
boundary of the Bainbridge compactification only in interior points of the boundary curves.
\end{Cor}

Note that the component $P_{D,\nu}$  of $P_D$ as defined
in Proposition~\ref{prop:redunionFN} is given by equation~\eqref{eqFn} 
with $\la=\nu\sqrt D$, $B=0$.

\subsection{\bf Examples} \label{sec:CompExamples}

{\rm {\em Case $D=17$}. We consider the cusp at $\infty$ for  $\SLOA{\O^\vee}$.
In this case $M = \O_{17}^\vee$ and the \Pred\ quadratic forms are
$[2,-3,-1]$, $[2,-1,-2]$, $[1,-3,-2]$,$[1,-1,-4]$ and $[1,1,-4]$. 
The following table contains the corresponding multiminimizers $\alpha \in \O_D^\vee$,
scaled by the factor~$\sq$.
\par
$$
\begin{array}{|c|c|c|c|c|c|}
\hline
\text{\Pred} &  [1,-3,-2] & [1,-1,-4] & [1,1,-4] & [2,-3,-1] & [2,-1,-2] \\
n & 1 & 2 & 3 & 4 & 5 \\
\lambda &  \frac{3+\sq}{2} & \frac{1+\sq}{2}  & 
\frac{-1+\sq}{2} & \frac{3+\sq}{4} & 
\frac{1+\sq}{4} \\
\sq\alpha & \frac{-3+\sq}{2} &\frac{-1+\sq}{2} & \frac{1+\sq}{2}  & 
\frac{3+\sq}{2} & 4+\sq   \\
\hline
\sq A_k & \frac{-3+\sq}{2} & \frac{-1+\sq}{2} & \frac{1+\sq}{2}  & 
\frac{3+\sq}{2} & 4+\sq   \\
x_k & \frac{7+\sq}{4} & \frac{9+\sq}{8} & \frac{7+\sq}{8} & 
\frac{5+\sq}{4} & \frac{5+\sq}{2} \\
k & 5 & 4 & 3 & 2 & 1 \\
\hline
\end{array} \medskip
$$
By the singularity criterion Proposition~\ref{prop:TorGeneral} this
compactification is smooth. 
\par
A reduced quadratic \changed{irrationality} is $x_1 = (5+\sq)/2$. Its fast minus 
continued fraction is listed in the lower part of that table.  The point 
$A_1 = \frac{1}{\sq} (4+\sq)$ lies on the lower convex
hull of $\ord_{\sq}^\vee \cap (\RR_+)^2$. By the algorithm for
Hirzebruch's compactification the subsequent points on the lower
convex hull are defined by $A_{k+1} = A_k/x_k$, as listed in the table. 
\par
The Bainbridge compactification has no singularities at the points $c_\lambda$
and the number of reduced quadratic forms equals the number of \Pred\
quadratic forms. Hence the Hirzebruch and the Bainbridge compactification 
coincide in this case. This is consistent with the table listing the
same values for $\sq\alpha_n$ and for $\sq A_k$.
\par
{\em Case $D=41$.}  There are 11 \Pred\ and also 11 reduced quadratic forms. But here the
Hirzebruch and the Bainbridge compactification do not coincide.
$$
\begin{array}{|c|c|c|c|c|c|c|c|}
\hline
\text{\Pred} &   \scriptstyle[1, -5, -4]  & \cdots & \scriptstyle[4, -5, -1] & 
\scriptstyle[2, -3, -4] & 
 \scriptstyle[2, 1, -5] & \cdots &  \scriptstyle [4, -3, -2] \\
n & 1 &  \cdots & 6 & 7 & 8 & \cdots & 11 \\ 
\lambda_n &  \frac{5+\sqrt{41}}{2}  & \cdots & \frac{5+\sqrt{41}}{8}  & 
\frac{3+\sqrt{41}}{4} & \frac{-1+\sqrt{41}}{4} & \cdots & 
\frac{3+\sqrt{41}}{8} \\
\sqrt{41}\alpha_n & \tfrac{-5+\sqrt{41}}2  & \cdots & \tfrac{5+\sqrt{41}}2 & \frac{19+3\sqrt{41}}{2} & \frac{83+13\sqrt{41}}2 
&  \cdots   &  \frac{429+67\sqrt{41}}{2} \\ 
\hline
\sqrt{41}A_k & \tfrac{-5+\sqrt{41}}2   & \cdots &  \tfrac{5+\sqrt{41}}2    
& 6+\sqrt{41} &
\frac{19+3\sqrt{41}}2  &   \cdots &   {826+129\sqrt{41}} \\
x_k & \frac{13+\sqrt{41}}{8}  & \cdots & \frac{11+\sqrt{41}}{8} & \frac{9+\sqrt{41}}{10} & 
\frac{7+\sqrt{41}}{4} &  \cdots &  \frac{11+\sqrt{41}}{10} \\
k & 11  & \cdots & 6 & 5 & 4 &  \cdots & 1 \\
\hline
\end{array} \medskip
$$
At the intersection points of the curves $C_{\lambda_7}$ and $C_{\lambda_8}$
and also at the intersection points of $C_{\lambda_8}$ and $C_{\lambda_9}$
the Bainbridge compactification is smooth, but 
$$\alpha_7+\alpha_9 = \alpha_8 \quad \text{and also} \quad  
\alpha_9+\alpha_{11} = \alpha_{10} $$
and hence this compactification is not minimal. 
In fact, $C_{\lambda_8}$ (and also $C_{\lambda_{10}}$) is a $(-1)$-curve and the corresponding 
values $\alpha_8$ and $\alpha_{10}$ do not show up in the list of $A_k$. 
On the other hand, at the 
intersection point of  $C_{\lambda_6}$ and $C_{\lambda_7}$ the Bainbridge compactification
has a quotient singularity of order two, since $\alpha_{6}$ and $\alpha_{7}$ generate
an index two subgroup of $\O_{41}^\vee$.  It can be resolved by blowing up, adding
a $(-2)$-curve, corresponding to the value $\sqrt{41}A_5 = 6 + \sqrt{41}$ that 
does not show up in the list of $\sqrt{41}\alpha_n$. The singularity can also be read off
from the quadratic form  $[4,5,-1]$ and Proposition~\ref{prop:fromBain}.
In terms of the convex hull of $\O_{41}^\vee$, the multiminimizer fan has 
two interior points and skips two boundary points, as is shown in 
Figure~\ref{cap:fancomparison} below. 
\medskip
\begin{figure}
\begin{tikzpicture}
\draw[->,thick] (-0.1,0)--(9,0) node [right] {$$};
\draw[->,thick] (0,-0.1)--(0,10) node [above] {$$};
\foreach \x in {1,...,6}
   \draw (\x,0.1) -- (\x,-0.1);
\foreach \y in {1,...,6}
   \draw (0.1,\y) -- (-0.1,\y);


\node (11) at (7.444363547, 0.007280748940){};
\node (10) at (8.555203590, 0.03167694459){};
\node (9) at   (1.110840042, 0.02439619565){};
\node (8) at   (1.442356791, 0.1878888162){};
\node (7) at   (0.3315167487, 0.1634926206){};
\node (6) at  (0.09893715608, 1.095655953){};
\node (5) at  (0.08158451587, 2.657393572){};
\node (4) at  (0.06423187566, 4.219131191){};
\node (3) at  (0.04687923545, 5.780868809){};
\node (2) at  (0.02952659524, 7.342606428){};
\node (1) at  (0.01217395503, 8.904344047){};

\path [draw=red,thick]  (11)--(10)--(9)--(8)--(7) --(6) -- (5) -- (4) -- (3) --  (2) -- (1);
\draw[fill] 
(11) circle (2pt)
(10) circle (2pt)
(9) circle (2pt)
(8) circle (2pt)
(7) circle (2pt)
(6) circle (2pt)
(5) circle (2pt)
               (4) circle (2pt)
		  (3) circle (2pt)
		  (2) circle (2pt)
		  (1) circle (2pt);

\node (102) at (7.444363547, 0.007280748940){};
\node (103) at (1.110840042, 0.02439619565){};
\node (104) at (0.3315167487, 0.1634926206){};
\node (105) at (0.2152269524, 0.6295742867){};
\node (106) at (0.09893715608, 1.095655953){};
\node (107) at (0.08158451587, 2.657393572){};
\node (108) at (0.06423187566, 4.219131191){};
\node (109) at (0.04687923545, 5.780868809){};
\node (110) at (0.02952659524, 7.342606428){};
\node (111) at (0.01217395503, 8.904344047){};

\path [draw=black,thick] (102)--(103)--(104)--(105)--(106)--(107)--(108)--(109)--(110)--(111);
\draw 
(102) circle (4pt)
(103) circle (4pt)
(104) circle (4pt)
(105) circle (4pt)
(106) circle (4pt)
(107) circle (4pt)
(108) circle (4pt)
(109) circle (4pt)
(110) circle (4pt)
(111) circle (4pt);

\draw[fill] 
(0.01217395503, 8.904344047) circle (1pt)
(0.02952659524, 7.342606428) circle (1pt)
(0.04687923545, 5.780868809) circle (1pt)
(0.06423187566, 4.219131191) circle (1pt)
(0.08158451587, 2.657393572) circle (1pt)
(0.1284637513, 8.438262381) circle (1pt)
(0.09893715608, 1.095655953) circle (1pt)
(0.1458163915, 6.876524762) circle (1pt)
(0.1631690317, 5.314787143) circle (1pt)
(0.1805216720, 3.753049524) circle (1pt)
(0.1978743122, 2.191311906) circle (1pt)
(0.2447535476, 7.972180715) circle (1pt)
(0.2152269524, 0.6295742867) circle (1pt)
(0.2621061878, 6.410443096) circle (1pt)
(0.2794588280, 4.848705477) circle (1pt)
(0.2968114682, 3.286967858) circle (1pt)
(0.3141641085, 1.725230239) circle (1pt)
(0.3610433439, 7.506099049) circle (1pt)
(0.3315167487, 0.1634926206) circle (1pt)
(0.3783959841, 5.944361430) circle (1pt)
(0.3957486243, 4.382623811) circle (1pt)
(0.4131012645, 2.820886192) circle (1pt)
(0.4599805000, 8.601755002) circle (1pt)
(0.4304539047, 1.259148573) circle (1pt)
(0.4773331402, 7.040017383) circle (1pt)
(0.4946857804, 5.478279764) circle (1pt)
(0.5120384206, 3.916542145) circle (1pt)
(0.5293910608, 2.354804526) circle (1pt)
(0.5762702963, 8.135673336) circle (1pt)
(0.5467437010, 0.7930669073) circle (1pt)
(0.5936229365, 6.573935717) circle (1pt)
(0.6109755767, 5.012198098) circle (1pt)
(0.6283282169, 3.450460479) circle (1pt)
(0.6456808571, 1.888722860) circle (1pt)
(0.6925600926, 7.669591669) circle (1pt)
(0.6630334973, 0.3269852412) circle (1pt)
(0.7099127328, 6.107854051) circle (1pt)
(0.7272653730, 4.546116432) circle (1pt)
(0.7446180132, 2.984378813) circle (1pt)
(0.7914972486, 8.765247622) circle (1pt)
(0.7619706534, 1.422641194) circle (1pt)
(0.8088498889, 7.203510003) circle (1pt)
(0.8262025291, 5.641772385) circle (1pt)
(0.8435551693, 4.080034766) circle (1pt)
(0.8609078095, 2.518297147) circle (1pt)
(0.9077870449, 8.299165956) circle (1pt)
(0.8782604497, 0.9565595278) circle (1pt)
(0.9251396851, 6.737428337) circle (1pt)
(0.9424923254, 5.175690718) circle (1pt)
(0.9598449656, 3.613953100) circle (1pt)
(0.9771976058, 2.052215481) circle (1pt)
(1.024076841, 7.833084290) circle (1pt)
(0.9945502460, 0.4904778617) circle (1pt)
(1.041429481, 6.271346671) circle (1pt)
(1.058782122, 4.709609052) circle (1pt)
(1.076134762, 3.147871433) circle (1pt)
(1.123013997, 8.928740243) circle (1pt)
(1.093487402, 1.586133815) circle (1pt)
(1.140366638, 7.367002624) circle (1pt)
(1.110840042, 0.02439619565) circle (1pt)
(1.157719278, 5.805265005) circle (1pt)
(1.175071918, 4.243527386) circle (1pt)
(1.192424558, 2.681789767) circle (1pt)
(1.239303794, 8.462658577) circle (1pt)
(1.209777198, 1.120052148) circle (1pt)
(1.256656434, 6.900920958) circle (1pt)
(1.274009074, 5.339183339) circle (1pt)
(1.291361714, 3.777445720) circle (1pt)
(1.308714354, 2.215708101) circle (1pt)
(1.355593590, 7.996576911) circle (1pt)
(1.326066995, 0.6539704823) circle (1pt)
(1.372946230, 6.434839292) circle (1pt)
(1.390298870, 4.873101673) circle (1pt)
(1.407651511, 3.311364054) circle (1pt)
(1.425004151, 1.749626435) circle (1pt)
(1.471883386, 7.530495245) circle (1pt)
(1.442356791, 0.1878888162) circle (1pt)
(1.489236026, 5.968757626) circle (1pt)
(1.506588667, 4.407020007) circle (1pt)
(1.523941307, 2.845282388) circle (1pt)
(1.570820542, 8.626151197) circle (1pt)
(1.541293947, 1.283544769) circle (1pt)
(1.588173182, 7.064413578) circle (1pt)
(1.605525823, 5.502675960) circle (1pt)
(1.622878463, 3.940938341) circle (1pt)
(1.640231103, 2.379200722) circle (1pt)
(1.687110339, 8.160069531) circle (1pt)
(1.657583743, 0.8174631029) circle (1pt)
(1.704462979, 6.598331912) circle (1pt)
(1.721815619, 5.036594293) circle (1pt)
(1.739168259, 3.474856675) circle (1pt)
(1.756520899, 1.913119056) circle (1pt)
(1.803400135, 7.693987865) circle (1pt)
(1.773873540, 0.3513814368) circle (1pt)
(1.820752775, 6.132250246) circle (1pt)
(1.838105415, 4.570512627) circle (1pt)
(1.855458055, 3.008775008) circle (1pt)
(1.902337291, 8.789643818) circle (1pt)
(1.872810696, 1.447037390) circle (1pt)
(1.919689931, 7.227906199) circle (1pt)
(1.937042571, 5.666168580) circle (1pt)
(1.954395212, 4.104430961) circle (1pt)
(1.971747852, 2.542693342) circle (1pt)
(2.018627087, 8.323562152) circle (1pt)
(1.989100492, 0.9809557235) circle (1pt)
(2.035979727, 6.761824533) circle (1pt)
(2.053332368, 5.200086914) circle (1pt)
(2.070685008, 3.638349295) circle (1pt)
(2.088037648, 2.076611676) circle (1pt)
(2.134916883, 7.857480486) circle (1pt)
(2.105390288, 0.5148740574) circle (1pt)
(2.152269524, 6.295742867) circle (1pt)
(2.169622164, 4.734005248) circle (1pt)
(2.186974804, 3.172267629) circle (1pt)
(2.233854040, 8.953136439) circle (1pt)
(2.204327444, 1.610530010) circle (1pt)
(2.251206680, 7.391398820) circle (1pt)
(2.221680085, 0.04879239129) circle (1pt)
(2.268559320, 5.829661201) circle (1pt)
(2.285911960, 4.267923582) circle (1pt)
(2.303264600, 2.706185963) circle (1pt)
(2.350143836, 8.487054772) circle (1pt)
(2.320617241, 1.144448344) circle (1pt)
(2.367496476, 6.925317154) circle (1pt)
(2.384849116, 5.363579535) circle (1pt)
(2.402201756, 3.801841916) circle (1pt)
(2.419554397, 2.240104297) circle (1pt)
(2.466433632, 8.020973106) circle (1pt)
(2.436907037, 0.6783666780) circle (1pt)
(2.483786272, 6.459235487) circle (1pt)
(2.501138913, 4.897497869) circle (1pt)
(2.518491553, 3.335760250) circle (1pt)
(2.535844193, 1.774022631) circle (1pt)
(2.582723428, 7.554891440) circle (1pt)
(2.553196833, 0.2122850119) circle (1pt)
(2.600076069, 5.993153821) circle (1pt)
(2.617428709, 4.431416202) circle (1pt)
(2.634781349, 2.869678584) circle (1pt)
(2.681660585, 8.650547393) circle (1pt)
(2.652133989, 1.307940965) circle (1pt)
(2.699013225, 7.088809774) circle (1pt)
(2.716365865, 5.527072155) circle (1pt)
(2.733718505, 3.965334536) circle (1pt)
(2.751071145, 2.403596917) circle (1pt)
(2.797950381, 8.184465727) circle (1pt)
(2.768423786, 0.8418592986) circle (1pt)
(2.815303021, 6.622728108) circle (1pt)
(2.832655661, 5.060990489) circle (1pt)
(2.850008301, 3.499252870) circle (1pt)
(2.867360942, 1.937515251) circle (1pt)
(2.914240177, 7.718384061) circle (1pt)
(2.884713582, 0.3757776325) circle (1pt)
(2.931592817, 6.156646442) circle (1pt)
(2.948945458, 4.594908823) circle (1pt)
(2.966298098, 3.033171204) circle (1pt)
(3.013177333, 8.814040014) circle (1pt)
(2.983650738, 1.471433585) circle (1pt)
(3.030529973, 7.252302395) circle (1pt)
(3.047882614, 5.690564776) circle (1pt)
(3.065235254, 4.128827157) circle (1pt)
(3.082587894, 2.567089538) circle (1pt)
(3.129467129, 8.347958347) circle (1pt)
(3.099940534, 1.005351919) circle (1pt)
(3.146819770, 6.786220729) circle (1pt)
(3.164172410, 5.224483110) circle (1pt)
(3.181525050, 3.662745491) circle (1pt)
(3.198877690, 2.101007872) circle (1pt)
(3.245756926, 7.881876681) circle (1pt)
(3.216230331, 0.5392702530) circle (1pt)
(3.263109566, 6.320139062) circle (1pt)
(3.280462206, 4.758401444) circle (1pt)
(3.297814846, 3.196663825) circle (1pt)
(3.344694082, 8.977532634) circle (1pt)
(3.315167487, 1.634926206) circle (1pt)
(3.362046722, 7.415795015) circle (1pt)
(3.332520127, 0.07318858694) circle (1pt)
(3.379399362, 5.854057396) circle (1pt)
(3.396752002, 4.292319777) circle (1pt)
(3.414104643, 2.730582159) circle (1pt)
(3.460983878, 8.511450968) circle (1pt)
(3.431457283, 1.168844540) circle (1pt)
(3.478336518, 6.949713349) circle (1pt)
(3.495689159, 5.387975730) circle (1pt)
(3.513041799, 3.826238111) circle (1pt)
(3.530394439, 2.264500493) circle (1pt)
(3.577273674, 8.045369302) circle (1pt)
(3.547747079, 0.7027628736) circle (1pt)
(3.594626315, 6.483631683) circle (1pt)
(3.611978955, 4.921894064) circle (1pt)
(3.629331595, 3.360156445) circle (1pt)
(3.646684235, 1.798418826) circle (1pt)
(3.693563471, 7.579287636) circle (1pt)
(3.664036875, 0.2366812075) circle (1pt)
(3.710916111, 6.017550017) circle (1pt)
(3.728268751, 4.455812398) circle (1pt)
(3.745621391, 2.894074779) circle (1pt)
(3.792500627, 8.674943589) circle (1pt)
(3.762974032, 1.332337160) circle (1pt)
(3.809853267, 7.113205970) circle (1pt)
(3.827205907, 5.551468351) circle (1pt)
(3.844558547, 3.989730732) circle (1pt)
(3.861911188, 2.427993113) circle (1pt)
(3.908790423, 8.208861923) circle (1pt)
(3.879263828, 0.8662554942) circle (1pt)
(3.926143063, 6.647124304) circle (1pt)
(3.943495704, 5.085386685) circle (1pt)
(3.960848344, 3.523649066) circle (1pt)
(3.978200984, 1.961911447) circle (1pt)
(4.025080219, 7.742780256) circle (1pt)
(3.995553624, 0.4001738281) circle (1pt)
(4.042432860, 6.181042638) circle (1pt)
(4.059785500, 4.619305019) circle (1pt)
(4.077138140, 3.057567400) circle (1pt)
(4.124017375, 8.838436209) circle (1pt)
(4.094490780, 1.495829781) circle (1pt)
(4.141370016, 7.276698590) circle (1pt)
(4.158722656, 5.714960971) circle (1pt)
(4.176075296, 4.153223353) circle (1pt)
(4.193427936, 2.591485734) circle (1pt)
(4.240307172, 8.372354543) circle (1pt)
(4.210780577, 1.029748115) circle (1pt)
(4.257659812, 6.810616924) circle (1pt)
(4.275012452, 5.248879305) circle (1pt)
(4.292365092, 3.687141686) circle (1pt)
(4.309717733, 2.125404068) circle (1pt)
(4.356596968, 7.906272877) circle (1pt)
(4.327070373, 0.5636664487) circle (1pt)
(4.373949608, 6.344535258) circle (1pt)
(4.391302248, 4.782797639) circle (1pt)
(4.408654889, 3.221060020) circle (1pt)
(4.426007529, 1.659322401) circle (1pt)
(4.472886764, 7.440191211) circle (1pt)
(4.443360169, 0.09758478258) circle (1pt)
(4.490239405, 5.878453592) circle (1pt)
(4.507592045, 4.316715973) circle (1pt)
(4.524944685, 2.754978354) circle (1pt)
(4.571823920, 8.535847164) circle (1pt)
(4.542297325, 1.193240735) circle (1pt)
(4.589176561, 6.974109545) circle (1pt)
(4.606529201, 5.412371926) circle (1pt)
(4.623881841, 3.850634307) circle (1pt)
(4.641234481, 2.288896688) circle (1pt)
(4.688113717, 8.069765498) circle (1pt)
(4.658587121, 0.7271590693) circle (1pt)
(4.705466357, 6.508027879) circle (1pt)
(4.722818997, 4.946290260) circle (1pt)
(4.740171637, 3.384552641) circle (1pt)
(4.757524278, 1.822815022) circle (1pt)
(4.804403513, 7.603683831) circle (1pt)
(4.774876918, 0.2610774032) circle (1pt)
(4.821756153, 6.041946213) circle (1pt)
(4.839108793, 4.480208594) circle (1pt)
(4.856461434, 2.918470975) circle (1pt)
(4.903340669, 8.699339784) circle (1pt)
(4.873814074, 1.356733356) circle (1pt)
(4.920693309, 7.137602165) circle (1pt)
(4.938045949, 5.575864547) circle (1pt)
(4.955398590, 4.014126928) circle (1pt)
(4.972751230, 2.452389309) circle (1pt)
(5.019630465, 8.233258118) circle (1pt)
(4.990103870, 0.8906516899) circle (1pt)
(5.036983106, 6.671520499) circle (1pt)
(5.054335746, 5.109782880) circle (1pt)
(5.071688386, 3.548045262) circle (1pt)
(5.089041026, 1.986307643) circle (1pt)
(5.135920262, 7.767176452) circle (1pt)
(5.106393666, 0.4245700237) circle (1pt)
(5.153272902, 6.205438833) circle (1pt)
(5.170625542, 4.643701214) circle (1pt)
(5.187978182, 3.081963595) circle (1pt)
(5.234857418, 8.862832405) circle (1pt)
(5.205330822, 1.520225977) circle (1pt)
(5.252210058, 7.301094786) circle (1pt)
(5.269562698, 5.739357167) circle (1pt)
(5.286915338, 4.177619548) circle (1pt)
(5.304267979, 2.615881929) circle (1pt)
(5.351147214, 8.396750739) circle (1pt)
(5.321620619, 1.054144310) circle (1pt)
(5.368499854, 6.835013120) circle (1pt)
(5.385852494, 5.273275501) circle (1pt)
(5.403205135, 3.711537882) circle (1pt)
(5.420557775, 2.149800263) circle (1pt)
(5.467437010, 7.930669073) circle (1pt)
(5.437910415, 0.5880626443) circle (1pt)
(5.484789651, 6.368931454) circle (1pt)
(5.502142291, 4.807193835) circle (1pt)
(5.519494931, 3.245456216) circle (1pt)
(5.536847571, 1.683718597) circle (1pt)
(5.583726807, 7.464587407) circle (1pt)
(5.554200211, 0.1219809782) circle (1pt)
(5.601079447, 5.902849788) circle (1pt)
(5.618432087, 4.341112169) circle (1pt)
(5.635784727, 2.779374550) circle (1pt)
(5.682663963, 8.560243359) circle (1pt)
(5.653137367, 1.217636931) circle (1pt)
(5.700016603, 6.998505740) circle (1pt)
(5.717369243, 5.436768122) circle (1pt)
(5.734721883, 3.875030503) circle (1pt)
(5.752074524, 2.313292884) circle (1pt)
(5.798953759, 8.094161693) circle (1pt)
(5.769427164, 0.7515552649) circle (1pt)
(5.816306399, 6.532424074) circle (1pt)
(5.833659039, 4.970686455) circle (1pt)
(5.851011680, 3.408948837) circle (1pt)
(5.868364320, 1.847211218) circle (1pt)
(5.915243555, 7.628080027) circle (1pt)
(5.885716960, 0.2854735988) circle (1pt)
(5.932596195, 6.066342408) circle (1pt)
(5.949948836, 4.504604789) circle (1pt)
(5.967301476, 2.942867170) circle (1pt)
(6.014180711, 8.723735980) circle (1pt)
(5.984654116, 1.381129552) circle (1pt)
(6.031533352, 7.161998361) circle (1pt)
(6.048885992, 5.600260742) circle (1pt)
(6.066238632, 4.038523123) circle (1pt)
(6.083591272, 2.476785504) circle (1pt)
(6.130470508, 8.257654314) circle (1pt)
(6.100943912, 0.9150478855) circle (1pt)
(6.147823148, 6.695916695) circle (1pt)
(6.165175788, 5.134179076) circle (1pt)
(6.182528428, 3.572441457) circle (1pt)
(6.199881068, 2.010703838) circle (1pt)
(6.246760304, 7.791572648) circle (1pt)
(6.217233709, 0.4489662194) circle (1pt)
(6.264112944, 6.229835029) circle (1pt)
(6.281465584, 4.668097410) circle (1pt)
(6.298818225, 3.106359791) circle (1pt)
(6.345697460, 8.887228601) circle (1pt)
(6.316170865, 1.544622172) circle (1pt)
(6.363050100, 7.325490982) circle (1pt)
(6.380402740, 5.763753363) circle (1pt)
(6.397755381, 4.202015744) circle (1pt)
(6.415108021, 2.640278125) circle (1pt)
(6.461987256, 8.421146934) circle (1pt)
(6.432460661, 1.078540506) circle (1pt)
(6.479339896, 6.859409316) circle (1pt)
(6.496692537, 5.297671697) circle (1pt)
(6.514045177, 3.735934078) circle (1pt)
(6.531397817, 2.174196459) circle (1pt)
(6.578277053, 7.955065268) circle (1pt)
(6.548750457, 0.6124588400) circle (1pt)
(6.595629693, 6.393327649) circle (1pt)
(6.612982333, 4.831590031) circle (1pt)
(6.630334973, 3.269852412) circle (1pt)
(6.647687613, 1.708114793) circle (1pt)
(6.694566849, 7.488983602) circle (1pt)
(6.665040254, 0.1463771739) circle (1pt)
(6.711919489, 5.927245983) circle (1pt)
(6.729272129, 4.365508364) circle (1pt)
(6.746624769, 2.803770746) circle (1pt)
(6.793504005, 8.584639555) circle (1pt)
(6.763977410, 1.242033127) circle (1pt)
(6.810856645, 7.022901936) circle (1pt)
(6.828209285, 5.461164317) circle (1pt)
(6.845561926, 3.899426698) circle (1pt)
(6.862914566, 2.337689079) circle (1pt)
(6.909793801, 8.118557889) circle (1pt)
(6.880267206, 0.7759514606) circle (1pt)
(6.927146441, 6.556820270) circle (1pt)
(6.944499082, 4.995082651) circle (1pt)
(6.961851722, 3.433345032) circle (1pt)
(6.979204362, 1.871607413) circle (1pt)
(7.026083598, 7.652476223) circle (1pt)
(6.996557002, 0.3098697945) circle (1pt)
(7.043436238, 6.090738604) circle (1pt)
(7.060788878, 4.529000985) circle (1pt)
(7.078141518, 2.967263366) circle (1pt)
(7.125020754, 8.748132176) circle (1pt)
(7.095494158, 1.405525747) circle (1pt)
(7.142373394, 7.186394557) circle (1pt)
(7.159726034, 5.624656938) circle (1pt)
(7.177078674, 4.062919319) circle (1pt)
(7.194431314, 2.501181700) circle (1pt)
(7.241310550, 8.282050509) circle (1pt)
(7.211783955, 0.9394440811) circle (1pt)
(7.258663190, 6.720312891) circle (1pt)
(7.276015830, 5.158575272) circle (1pt)
(7.293368471, 3.596837653) circle (1pt)
(7.310721111, 2.035100034) circle (1pt)
(7.357600346, 7.815968843) circle (1pt)
(7.328073751, 0.4733624150) circle (1pt)
(7.374952986, 6.254231224) circle (1pt)
(7.392305627, 4.692493606) circle (1pt)
(7.409658267, 3.130755987) circle (1pt)
(7.456537502, 8.911624796) circle (1pt)
(7.427010907, 1.569018368) circle (1pt)
(7.473890142, 7.349887177) circle (1pt)
(7.444363547, 0.007280748940) circle (1pt)
(7.491242783, 5.788149558) circle (1pt)
(7.508595423, 4.226411939) circle (1pt)
(7.525948063, 2.664674321) circle (1pt)
(7.572827299, 8.445543130) circle (1pt)
(7.543300703, 1.102936702) circle (1pt)
(7.590179939, 6.883805511) circle (1pt)
(7.607532579, 5.322067892) circle (1pt)
(7.624885219, 3.760330273) circle (1pt)
(7.642237859, 2.198592655) circle (1pt)
(7.689117095, 7.979461464) circle (1pt)
(7.659590500, 0.6368550356) circle (1pt)
(7.706469735, 6.417723845) circle (1pt)
(7.723822375, 4.855986226) circle (1pt)
(7.741175015, 3.294248607) circle (1pt)
(7.758527656, 1.732510988) circle (1pt)
(7.805406891, 7.513379798) circle (1pt)
(7.775880296, 0.1707733695) circle (1pt)
(7.822759531, 5.951642179) circle (1pt)
(7.840112172, 4.389904560) circle (1pt)
(7.857464812, 2.828166941) circle (1pt)
(7.904344047, 8.609035751) circle (1pt)
(7.874817452, 1.266429322) circle (1pt)
(7.921696687, 7.047298132) circle (1pt)
(7.939049328, 5.485560513) circle (1pt)
(7.956401968, 3.923822894) circle (1pt)
(7.973754608, 2.362085275) circle (1pt)
(7.991107248, 0.8003476562) circle (1pt);
\end{tikzpicture}
\caption{The multiminimizer fan (thick black points) and the Hirzebruch fan 
(circles) for $D=41$, \changed{connected by red respectively by black 
lines}. The rightmost thick black point is an interior point
of the convex hull. Another point of the Hirzebruch fan skipped by the
multiminimizer fan is not drawn (far to the right, close to the horizontal axis)}
\label{cap:fancomparison}
\end{figure}
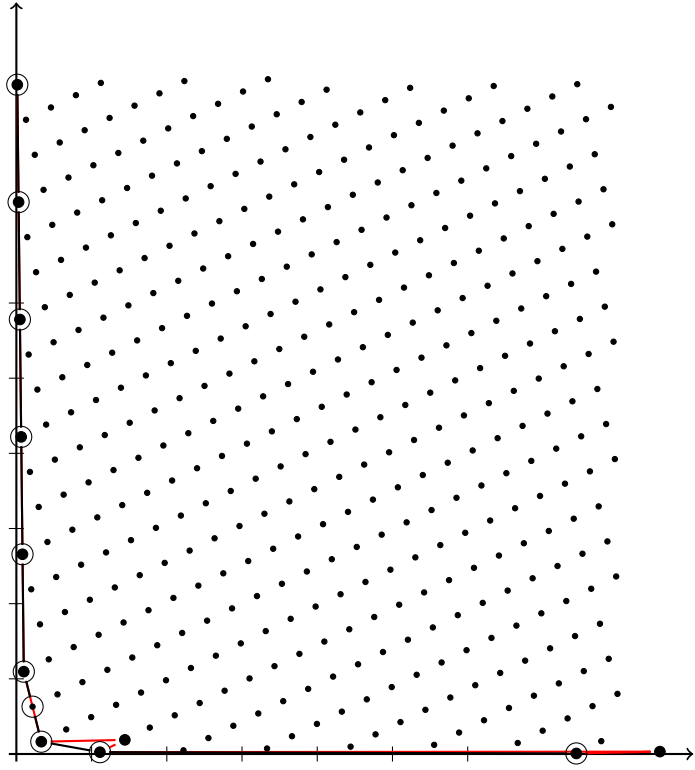
\par
\subsection{\bf Relating the two compactifications and two 
continued fraction algorithms} \label{sec:RelBainHirz}

The preceding examples show that the Bainbridge and the Hirzebruch compactification
may sometimes agree, but that they are different in general. They also illustrate the
general algorithm how to go from Bainbridge's compactification to Hirzebruch's:
blow down curves corresponding to interior points and blow up points where
boundary points of the lower convex hull have been omitted by the multiminimizers.
\par 
At first glance the previous examples suggest that at least the length of the 
boundary cycles in the Bainbridge and the Hirzebruch compactification agree.
This is true for class number one, but otherwise the truth is more subtle,
as we now explain.
\par
To determine the geometry of the cusp $\fraca$ in the Bainbridge compactification, 
we need to run the multiminimizer algorithm given in \S\ref{sec:compMMM} for $\lambda$ such that
$\fraca = \langle 1, \lambda \rangle$ and get as an output $\alpha \in (\fraca^2)^\vee$.
The geometry of Hirzebruch's minimal smooth resolution, however, depends only
on the square of $\fraca^2$, since we need to determine by the algorithm in 
\S\ref{sub:HirzComp} the lower convex hull of $(\fraca^2)^\vee \cap \RR^2_+$. 
\par
Consequently, whenever the squaring map is not an
isomorphism on the ideal class group of $K=\QQ(\sqrt{D})$, only the reduced quadratic
irrationalities $x$ such that $(\langle 1,x \rangle)^\vee$ is a square appear as label
in the cusp resolution of the Hilbert modular surface $X_D$ in the focus of
our interest. However all the curves labeled by \Pred\ quadratic forms appear 
on $X_D$. An example to illustrate this is $D=65$, where the class group 
(both in the narrow and wide sense) is
of order two, so the multiminimizers for both cusps lie in an ideal equivalent
to $\O_{65}^\vee$. Representatives of the multiminimizers are
\bas
{\rm MM}(\O_{65}) & \= \bigl\{\tfrac{5+\sqrt{65}}{2\sqrt{65}}, \tfrac{7+\sqrt{65}}{2\sqrt{65}},  
\tfrac{23+3\sqrt{65}}{2\sqrt{65}}, 
\tfrac{8 + \sqrt{65}}{\sqrt{65}},  
\tfrac{153+19\sqrt{65}}{2\sqrt{65}}, \bigr. \\ 
& \phantom{\= m} \bigl. \tfrac{137+17\sqrt{65}}{2\sqrt{65}},  \tfrac{395+49\sqrt{65}}{2\sqrt{65}},
\tfrac{1685+209\sqrt{65}}{2\sqrt{65}}, \tfrac{1943+241\sqrt{65}}{2\sqrt{65}} \bigr\} \\
\tfrac{7-\sqrt{65}}{2}\,{\rm MM}(\langle 1, \tfrac{-3+\sqrt{65}}{4} \rangle) & \= \bigl\{
 \tfrac{5+\sqrt{65}}{2\sqrt{65}},  \tfrac{7+\sqrt{65}}{2\sqrt{65}},  
\tfrac{8 + \sqrt{65}}{\sqrt{65}}, \tfrac{137+17\sqrt{65}}{2\sqrt{65}} 
\bigr. \\ 
& \phantom{\= m} \bigl. \tfrac{395+49\sqrt{65}}{2\sqrt{65}},  \tfrac{653+81\sqrt{65}}{2\sqrt{65}}, 
 \tfrac{911+113\sqrt{65}}{2\sqrt{65}} \bigr\}
\eas
while the elements on the convex hull, up to multiplication by $U_{65}^2$, are given by
\bas
\bigl\{
 \tfrac{5+\sqrt{65}}{2\sqrt{65}},  \tfrac{7+\sqrt{65}}{2\sqrt{65}},  
\tfrac{8 + \sqrt{65}}{\sqrt{65}}, \tfrac{137+17\sqrt{65}}{2\sqrt{65}}, 
 \tfrac{395+49\sqrt{65}}{2\sqrt{65}},  \tfrac{653+81\sqrt{65}}{2\sqrt{65}},  \\
 \tfrac{911+113\sqrt{65}}{2\sqrt{65}}, \tfrac{1169+145\sqrt{65}}{2}, \tfrac{1427+177\sqrt{65}}{2} \bigr\}
\eas
\par
\medskip
There is however an equality of total lengths of cusp resolution cycles, if one takes all 
Hilbert modular surfaces  $\HH^2/{\rm SL}(\fracb \oplus \O_D)$ into account.
We start our considerations on the level of  quadratic irrationalities.
\par
Reduced quadratic irrationalities as defined in~\eqref{eq:defreduced} 
are well-known and have been used to label the boundary curves of 
Hirzebruch's compactification. Standard quadratic irrationalities as defined 
in~\eqref{eq:defstd} have been used to label the boundary
curves of Bainbridge's compactification. There is an obvious bijection
between these two classes, given by
\par
\begin{tikzpicture}[shorten >=6pt,shorten <=6pt,auto]
\draw (0,3)  node(A) {};
\draw (2,3) node(B)[inner sep=5pt,draw] {\parbox{2.6cm}{$\text{reduced}\\ x>1>x^{\sigma} >0$}};
\draw (9,3) node(D)[inner sep=5pt,draw] {\parbox{2.5cm}{$\text{standard}\\ \lambda>1>0>\lambda^\sigma$}};
\draw[thick,<->](B) -- (D)  node [pos=0.5,above] {$x=\frac{\lambda}{\lambda -1}$};
\draw[thick,<->](B) -- (D)  node [pos=0.5,below] {$\lambda=\frac{x}{x -1}$};
\end{tikzpicture}
\par
\medskip
In order to pass from this correspondence of quadratic irrationalities
to cusp resolutions, we define 
$$ C^{\rm std}(\fraca) \= \{(\mu,\lambda) \in K^+/U_D^2 \times K\,: 
\fraca = \mu \langle 1,\lambda \rangle,\, \,\lambda\,\, \text{\Pred}\}$$
and
$$ C^{\rm red}(\fraca) \= \{(\rho,x) \in K^+/U_D^2 \times K\,: 
\fraca = \rho \langle 1,x \rangle,\,\,x \,\, \text{reduced}\}.$$
Then the map
$$\phi: C^{\rm std}(\fraca) \to  C^{\rm red}(\fraca), \quad
(\mu,\lambda) \mapsto (\rho,x) = (\mu(\lambda-1),\tfrac{\lambda}{\lambda-1})$$
is obviously a bijection.
\par
\begin{Prop} For any $\fraca$, the length of the cycle of the 
Bainbridge compactification for the cusp corresponding to $\fraca^2$ is equal to 
the length of the cycle of the Hirzebruch compactification for the 
cusp corresponding to $\fraca$.
\end{Prop}
\par
\begin{proof} Curves in the Bainbridge compactification for $\fraca^2$
are in bijection to  $C^{\rm std}(\fraca^2)$ by Theorem~\ref{thm:bainviamultimin} 
together with Proposition~\ref{prop:MMchar2}. Curves in the Hirzebruch 
compactification for $\fraca$ are in bijection to  $C^{\rm red}(\fraca^2)$
by the algorithm in~\S\ref{sub:HirzComp}. We now can use the identification $\phi$.
\end{proof} 
\par
This explains the initial observations for $D=17$ and $D=41$ and in general for
$\fraca = \fraca^2 = \O_D$. For a more symmetric formulation we recall that the isomorphism class of the 
Hilbert modular surface $\HH^2/{\rm SL}(\fracb \oplus \O_D)$ depends on $\fracb$
only up to squares of ideals  and multiplication by a totally positive element in~$K$. 
Consequently, there are $2^{t-1} = |{\rm Ker}(\rm{Sq}: {\mathcal Cl}^+(D) \to {\mathcal Cl}^+(D))|$
Hilbert modular surfaces for a given $D$, where $t$ is the number of distinct
prime factors of $D$. This also implies that every ideal class $\fraca$
(in the wide sense) appears $2^{t-1}$ times as the module $M$ associated
with a cusp $\HH^2/G(M,V)$ on the total collection of Hilbert modular 
surfaces for a given $D$. Altogether, this implies that the number 
$$ \ell(D) \= \sum_{[a,b,c] \, \text{reduced} \atop b^2-4ac = D} 1  \= 
\sum_{[a,b,c] \, \text{\Pred} \atop b^2-4ac = D} 1$$
appears in two incarnations.
\par
\begin{Prop} The total number of boundary curves of the
Bainbridge compactification of $X_D$ is equal to $\ell(D)$.
\par
The total number of boundary curves of the
Hirzebruch compactifications of the Hilbert modular surfaces
$\HH^2/{\rm SL}(\fracb_i \oplus \O_D)$ for $\fracb_i$ in a set of
representatives of $ {\mathcal Cl}^+(D)/{\rm Sq}({\mathcal Cl}^+(D))$ 
is equal to $2^{t-1}\ell(D)$.
\end{Prop}
\par
We could presumably make this statement even more symmetrical if we compactified
all Hilbert modular surfaces  $\HH^2/{\rm SL}(\fracb \oplus \O_D)$
using multiminimizers defined by theta functions, but we leave it to
the reader to explore this.

\section{Uniformization and disjointness from the reducible locus}
\label{sec:disjointred}

The aim of this section is to give an independent proof of ``$W_D$ is a \Teichmuller\ 
curve'' using the definition via theta functions and applying 
Theorem~\ref{thm:charTeich}. 
We emphasize that we derive all properties of $W_D$ {\em ab ovo},  
i.e., using just the definition~\changed{\eqref{eq:defofDtheta}} 
as the vanishing locus of $\PTD$ and without using anything that follows from the ``geodesic"
definition. Here again we simplify our counting task by restricting to fundamental discriminants.
\subsection{Transversality of $W_D$ to the foliation $\cFF_1$}
Recall that in~\S\ref{sec:defTeich} we defined the ``first Hilbert modular foliation" $\cFF_1$ 
of the Hilbert modular surface~$X_D$ to be the foliation defined by the constancy of the first 
cooordinate in the uniformization (i.e., by the equation $dz_1=0$, which is invariant under 
the action of the Hilbert modular group). Let $\varphi: \HH \to \HH$ 
be a holomorphic map such that $z \mapsto (z,\varphi(z))$ defines a branch 
of the vanishing locus of~$\PTD$, and suppose that  the corresponding 
component $W_D$ of the vanishing locus has the uniformization $W_D = \HH/\Gamma$.
\par
\begin{Thm}  \label{thm:transversal}
Suppose that $D$ is a fundamental discriminant.  Then the restriction to $W_D$ of the
derivative $\PTDDD = \frac{\partial}{\partial z_2} \PTD$ is a 
$\varphi$-twisted modular form of weight $(3,11)$ for $\Gamma$ that
vanishes only at the cusps of $W_D$. 
\end{Thm}
Since the vanishing locus of a holomorphic function $F$ is transversal to $\cFF_1$
at a point $p$ if and only if $\tfrac{\partial}{\partial z_2} F(p) \neq 0$, the
theorem immediately implies the transversality we want to prove.
\par
For all $(m,m') \in (\tfrac12\ZZ^2)^2$ and for any basis
$\bfom = (\omega_1, \omega_2)$ of an $\O$-ideal $\fraca$, define
$$D_2^3\theta\TchiP{m}{m'}{\bfom} (\zz) \=  
 \frac{\d^3}{\d u_2^3 }\,\Bigl(\Theta \Tchi{m}{m'}(u,\psi_\bfom(\zz)) \Bigr)\Bigr|_{u=0}\,,  $$
where the Siegel modular embedding $\psi_\bfom$ is defined using $\bfom$.
If we drop the index $\bfom$ we tacitly assume that we have
chosen some basis of $\O$, as we did in~\eqref{eq:thetaexpl}. It follows that 
$$D_2^3\theta \changed{\Tchi{m}{m'}}(z_1,z_2) \= \frac{\d}{\d z_2}
\,\Bigl(D_2\theta\Tchi{m}{m'}(z_1,z_2)\Bigr)$$
(by the heat equation or by direct computation). Consequently, 
$$\PTDDD(z_1,z_2)  \= \sum_{(m_0,m_0')\, {\rm odd}} D_2^3\theta \Tchi{m_0}{m_0'}(z_1,z_2) 
\prod_{(m,m')\, {\rm odd} \atop (m,m') \neq (m_0, m_0')} D_2\theta \Tchi{m}{m'}(z_1,z_2), $$
explaining also the name given to this twisted modular form. 
\par
\begin{proof}
By applying the chain rule one sees that the restriction of the $z_i$-derivative 
of a Hilbert modular form to the vanishing locus of the form satisfies a 
modular transformation property with respect to 
the subgroup stabilizing this vanishing locus. 
This calculation also shows that the $z_i$-derivative increases the weight 
in the $i$th component by two, proving the first claim.
\par
To prove the second claim we will show that the vanishing orders of $\PTDDD$
at the cusps sum up to the total vanishing order of a twisted modular form of this bi-weight.
Since $\PTDDD$ is holomorphic, it cannot then have any zeros at finite points.
On a minimal compactification of the Hilbert modular surface $X_D$ the number 
of intersection points of two modular forms of bi-weights $(k_1,\ell_1)$ and $(k_2,\ell_2)$ is
$\tfrac14(k_1\ell_2 + k_2\ell_1)|\chi(X_D)|$. This follows e.g.\ from 
\cite{vG87}, Section~IV.2. This calculation is still valid when intersecting
the vanishing locus $W_D$ of a modular form with a section of the bundle
of modular forms of bi-weight  $(k_2,\ell_2)$ to $W_D$. Since here 
$(k_1,\ell_1) = (3,9)$ and $(k_2,\ell_2)=(3,11)$, the function $\PTDDD_a$ 
has ${15}|\chi(X_D)|$ zeros on the closure of $W_D$. We have to show that they all lie at the cusps.
\par
Note that precisely one of the two quadratic forms $[a,b,c]$ and
$[-c,-b,-a]$ of discriminant $D$ with $a>0$ and $c<0$ satisfies the
additional condition $a+b+c < 0$ required to make it a \Pred\ quadratic form,
\changed{since $a+b+c=0$ would imply that~$D$ is a square}.
From~\eqref{eq:chiformula} we consequently deduce  that 
\bes {15}\chi(X_D)| 
\= \frac 12 \sum_{D = b^2 -4ac \atop a>0,\ c<0 } a  
 \= \frac 12 \sum_{[a,b,c]\, \text{\Pred} \atop D = b^2 -4ac} \bigl(a + |c|\bigr) \,.
\ees
To complete the proof it thus suffices to show that at each of the $g$
cusps of~$W_D$ corresponding to $\lambda$, the vanishing order of 
$\PTDDD$ is at least (and hence precisely) equal to $(a+|c|)/2g$. 
Here $\lambda$ is a zero of the \Pred\ form $[a,b,c]$ and $g = {\rm gcd}(a,c)$.
(The order of zero may indeed be half-integral, in accordance with the
fact that $\PTD$ is a modular form with a quadratic character.) 
\par 
By Lemma~\ref{le:DthetaOtherCusp} the vanishing orders of $\PTDDD$ 
at the cusps of $W_D$ mapping to $\alpha$ can be computed as
the vanishing orders of 
$$\PTDDDa(z_1,z_2)  = \sum_{(m_0,m_0')\, {\rm odd}} D_2^3\theta \TchiP{m_0}{m_0'}{\bfom}
(z_1,z_2) 
\prod_{(m,m')\, {\rm odd} \atop (m,m') \neq (m_0, m_0')} D_2\theta \TchiP{m}{m'}{\bfom}
(z_1,z_2), $$
on the Hilbert modular surface $X_{D,\fraca}$ 
at the cusps mapping to infinity. By Theorem~\ref{thm:cuspsWD} we associate
a multiminimizer $\alpha$ to such a cusp and we may assume that 
$\bfom = (\omega_1,\omega_2)$ has been chosen to be the distinguished bases
for this multiminimizer. If $t$ is a local parameter of a cusps of $W_D$, 
as in the proof of Theorem~\ref{thm:cuspsWD} then the terms appearing
in the expansion of  $D_2\theta \TchiP{m_0}{m_0'}{\bfom}$ and 
$D_2^3\theta \TchiP{m_0}{m_0'}{\bfom}$ are $t^{F(\tx_1,\tx_2)/2}$, where
$\tx_i = x_i +  m_i$ with $x_i \in \ZZ$ and where 
$$ F \= \Bigl[\frac ag, 0, \frac {-c}g \Bigr] $$
as in~\eqref{eq:defFG}. The minimal $t$-exponents are greater or equal to
$(a+|c|)/8g$, $a/8g$, $|c|/8g$ in case $m$ is equal to 
$(\tfrac12,\tfrac12)$, $(\tfrac12,0)$ and $(0,\tfrac12)$
respectively, both for $D_2\theta$ and $D_2^3\theta$. Since for each of these $m$ 
there are precisely two $m'$ such that $(m,m')$ is odd, the total vanishing
order is at least $(a+|c|)/2g$, which is what we wanted to show.
\end{proof}

\subsection{Disjointness of ${W_D}$ from the reducible locus (by counting zeros)}

In this and the following subsection we give two completely different proofs of the following
result, which is the second half of what we need to apply the criteria of Theorem~\ref{thm:charTeich}
and show that the vanishing locus of $\PTD$ is a \Teichmuller\ curve.  The first proof is similar to
the one used for Theorem~\ref{thm:transversal}, by comparing the number of known zeros of
a twisted modular form with its total number of zeros.
\begin{Thm} \label{thm:disjred}
The vanishing locus $W_D$ of $\PTD$ is disjoint in $X_D$ from the reducible locus $P_D$.
\end{Thm}
\par
\begin{proof}
Recall that the reducible locus is the vanishing locus of the product
of all $10$ even theta functions. This product is a Hilbert
modular form of weight $(5,5)$, so its restriction to $W_D$
is a modular form for $W_D$ of bi-weight $(5,5)$. As in the preceding
proof we deduce that the degree of its divisor (on a compactification of~$W_D$)
is $\tfrac14 (5\cdot3 + 5\cdot9)\,|\chi(X_D)|\,=\,15\,|\chi(X_D)|\,$. 
\par
As in the preceding proof it suffices to show that the restriction of the product
$\prod_{(m,m')\, {\rm even}}  \theta\Tchi{m}{m'}$  to $W_D$ vanishes at each of
the $g$ cusps of $W_D$ corresponding to $\lambda$ to the order at least $(a+|c|)/2g$.
Here again we can work at the cusp $\infty$ of $X_{D,\fraca}$. There, by the
same argument as in Lemma~\ref{le:DthetaOtherCusp}, the product of the ten even theta 
functions is  given by $\prod_{(m,m')\, {\rm even}}  \theta\TchiP{m}{m'}{\bfom}$, where $\bfom$
is some basis of $\fraca$ and where 
$\theta\TchiP{m}{m'}{\bfom}(\zz)  = \Theta\TchiP{m}{m'}{\bfom}(0,\psi_\bfom(\zz))$
with the modular embedding $\psi_\bfom$ defined using $\bfom$. The rest
of the proof proceeds as above. To each cusp we associate its multiminimizer
and take $\bfom$ to be the distinguished basis. The terms appearing in 
the expansion of $\theta\TchiP{m}{m'}{\bfom}$ at such a cusp are
$t^{F(\tx_1,\tx_2)/2}$, where
$\tx_i = x_i + m_i$ with $x_i \in \ZZ$ and where 
$F \= \Bigl[\frac ag, 0, \frac {-c}g \Bigr] $. The minimal $t$-exponents are 
greater or equal to
$(a+|c|)/8g$, $a/8g$, $|c|/8g$ and $0$ in case $m$ is equal to 
$(\tfrac12,\tfrac12)$, $(\tfrac12,0)$,  $(0,\tfrac12)$ and $(0,0)$
respectively. Each of the first three cases occurs twice among the ten
even theta characteristics (and the irrelevant last case four times). 
Summing up these contributions gives again the vanishing order  at least 
$(a+|c|)/2g$ that we claimed.
\end{proof}
\par

\subsection{Disjointness of ${W_D}$ from the reducible locus (via theta products)}

In this subsection we give a proof of Theorem~\ref{thm:disjred} based on a completely different idea,
by establishing a formula for the restriction of theta derivatives to the reducible locus. Let
$$
\begin{aligned}
\theta_{00}(z) = \sum_{n \in \ZZ} q^{n^2/2}\,, \quad 
&\theta_{\frac{1}{2}0}(z) = \sum_{n \in \ZZ+1/2} q^{n^2/2}\,, \\
\theta_{0\frac{1}{2}}(z) = \sum_{n \in \ZZ} (-1)^n q^{n^2/2}\,, \quad
&\theta_{\frac{1}{2}\frac{1}{2}} (z)= \sum_{n \in \ZZ+1/2}(-1)^{n-1/2}n q^{n^2/2}\,. \\
\end{aligned}
$$
(Here $\theta_{\frac{1}{2}\frac{1}{2}}$ should perhaps be called 
$\theta'_{\frac{1}{2}\frac{1}{2}}$, since the corresponding Jacobi Thetanull\-wert 
vanishes identically, but with these notations it will be easier to 
write a closed expression.) The product formulas found by Jacobi for 
these four functions, which in a modern notation say that they are 
equal to $\eta(2z)^5/\eta(z)^2\eta(4z)^2$,  $2\eta(4z)^2/\eta(2z)$, 
 $\eta(z)^2/\eta(2z)$ and $\eta(z)^3$, respectively, show that none of 
these functions vanish anywhere in the upper half-plane. 
Recall from Proposition~\ref{prop:redunionFN} that the reducible locus is 
the  union of irreducible curves $P_{D,\nu} = F_N(\nu)$ for $\nu = \frac{r + \sqrt{D}}{2\sqrt{D}}$,
where $D=r^2+4N$ with $N \in \NN$. Therefore Theorem~\ref{thm:disjred} follows 
from the following theorem, which for simplicity we formulate only for $D$ odd, 
the case of even~$D$ being similar.
\par
\begin{Thm} \label{thm:DTHrestFN}
Let $D \equiv 1\pmod4$ be a fundamental discriminant.
Then for any odd theta characteristic $(m,m')$ the restriction of the 
modular form $D_2\theta\Tchi{m}{m'}$ to the curve 
$P_{D,\nu}$ for $\nu$ as above has the factorization
$$D_2\theta\Tchi{m}{m'}(\nu z ,\nu^\sigma z) = \left\{ 
\begin{matrix} 
\phantom{\sqrt{D} \nu}-\, \theta_{\hm_1 m'_1}(z) \theta_{m_2 \hm'_2}(Nz) \quad & \text{if}\;\; m_1 =
m_1' = 1/2 \\ 
- \sqrt{D} \nu^\sigma\, \theta_{\hm_1 m'_1}(z) \theta_{m_2 \hm'_2}(Nz), \quad & \text{if}\;\;
m_2 = m'_2 = 1/2 \\
\end{matrix}
\right.
$$
as a product of Jacobi theta functions, where $\hm_1$ and $\hm'_2$ are defined by
$$
\begin{array}{lll}
\hm_1 = m_1 + m_2 \mod (1), \quad & \hm'_2 = m'_1 + m'_2 \mod (1)
\quad & \text{if} \;\; r \equiv 1\mod 4 \\
\hm_1 = m_1, \quad &\hm_2 = m_2, 
\quad & \text{if} \;\; r \equiv 3\mod 4. \\
\end{array}
$$
for $m = (m_1, m_2)$, $m' = (m'_1,m'_2)$. In particular, this restriction vanishes only at cusps.
\end{Thm}
\par
\begin{proof}  
As above we use a tilde to denote elements of the shifted lattice, 
i.e., $\tx_i = x_i + m_i$. The restriction of the theta derivative 
to $P_{D,\nu}$ is 
$$ 
\begin{aligned}
D_2\theta\Tchi{m}{m'}(\nu z ,\nu^\sigma z) &\= \sum_{(x_1,x_2) \in \ZZ^2} 
(-1)^{2(x_1,x_2)(m')^{T}} \rho(\tx_1,\tx_2)^\sigma  q^{\tr(\nu\rho(\tx_1,\tx_2)^2)/2}\,,
\end{aligned}
$$
where $\rho(x_1,x_2) = x_1 + \tfrac{1+\sqrt{D}}2 x_2$. 
With the $\ZZ$-linear transformation  $\ty_1= \tx_1 + \frac{r+1}2 \tx_2$ 
$\ty_2 = -\tx_2$ we obtain
$$\rho(\tx_1,\tx_2) \= \ty_1 + \frac{r-\sqrt{D}}{2} \ty_2 
\= \ty_1 + \sqrt{D}\nu^\sigma \ty_2 \,=:\,\tau(\ty_1,\ty_2)$$ 
and 
$$ 
\begin{aligned}D_2\theta\Tchi{m}{m'}(\nu z,\nu^\sigma z) 
& = \sum_{(\ty_1,\ty_2) \in \ZZ^2 + (\hm_1,m_2)} 
\epsilon(\ty_1,\ty_2) \tau(\ty_1,\ty_2) q^{(\ty_1^2+N\ty_2^2)/2},
\end{aligned}
$$
where 
$$ \epsilon(\ty_1,\ty_2) = (-1)^{2(\ty_1-\hm_1)m_1'+2(\ty_2-m_2)\hm'_2}$$ 
and where we have used 
$\tr(\nu \tau(y_1,y_2)^2) = y_1^2+Ny_2^2$. The $q$-exponent is invariant under both
$\ty_1\mapsto -\ty_1$ and $\ty_2 \mapsto -\ty_2$.  Under, say, $\ty_2\mapsto -\ty_2$ the sign of
$\epsilon(\ty_1,\ty_2)$ is unchanged unless $m_2 = \hm_2' = 1/2$. Hence, 
unless $m_2 = \hm_2' = 1/2$, the $\sqrt{D}\nu^\sigma \ty_2$-contribution
of $\tau(\ty_1,\ty_2)$ cancels and one checks the formula by multiplying the
unary theta functions. For an odd theta constant $(m,m')$ precisely
one of the cases $m_1 = \hm_1' = 1/2$ or $\hm_2 = m'_2 = 1/2$ happens
and a similar cancellation gives the formula in the second case, too.
Notice that in each case of the theorem, one of the theta series in the decomposition
is the function $\theta_{\frac{1}{2}\frac{1}{2}}$ of weight~$3/2$, so that the
total weight is always~$2$.
\end{proof}
\par
\paragraph{\bf Open problem:} Can one reprove the irreducibility, 
stated in Theorem~\ref{thm:TMclassg2} and proved by McMullen using combinatorial
number theory of the set of cusps, exclusively with techniques
of (Hilbert) modular forms? 
\par

\section{Applications} \label{sec:applications}

\subsection{The modular embedding via theta functions} \label{sec:phiarithmetic}

Using the description of $W_D$ as vanishing locus of $\PTD$, we can now give
the modular embedding $\varphi$ as in the ``Fourier expansion'' as defined 
in~\eqref{eq:phiatinfinity} for any cusp of $W_D$. In fact most of this was already
achieved in the proof of Theorem~\ref{thm:cuspsWD}. There we identified a 
cusp of $W_D$ with a \Pred\ quadratic form $[a,b,c]$ (and hence a multiminimizer 
$\alpha$) together with an element in $r \in \ZZ/g\ZZ$ where $g = {\rm gcd}(a,c)$ 
(equivalently, a solution $S$ of~\eqref{eq:lowesttermeq}). Recall from this 
proof that for each such solution there is a unique branch, given 
by~\eqref{eq:branch}, of the locus
$ D_2\theta  \TchiP{m}{m'}{\fraca}(z_1,z_2) = 0.$
The map $z \mapsto (z,\varphi(z))$ describing this branch was given by
\be \label{eq:varphiagain}
 \varphi(z) =\frac{\alpha^\sigma}{\alpha}z + C + \frac{\ve(q)}{2\pi i}.
\ee
We can now describe the arithmetic properties of this expansion.
\par
\begin{Thm} \label{thm:arithphi}
The coefficients of the modular embedding describing the branch determined by the 
quadratic form $[a,b,c]$ and number $S$ as in~\eqref{eq:lowesttermeq} through 
the cusp $\fraca$ of $X_D$ have the following properties.
\begin{itemize}
\item[i)] The constant $C$ in~\eqref{eq:varphiagain} belongs to $\frac1{2\pi i}K^\times\log K^\times$.  
In fact, $\e\Bigl( \frac{g\alpha\,C}{N(\fraca)^2\sqrt{D}}\Bigr) \in K\ssm\QQ$.
\item[ii)] The number $S$ lies in $K^{1/g}$. 
\item[iii)] For each $\beta$ completing the multiminimizer $\alpha$ determined by $[a,b,c]$
to a basis of $(\fraca^2)^\vee$ there exists $A\in \CC^*$ such that in the
local parameter $Q = Aq$ of the cusp 
\bes
\e(\nu z + \nu^\sigma \varphi(z))  \= S^{\tr(\beta\nu)} \, 
Q^{\tr(\alpha \nu)} e^{\sigma(\nu)\ve(q)} \qquad
\text{for all} \quad \nu \in \fraca^2\,.
\ees
The scalar $A$ is transcendental of Gelfond-Schneider type, more
precisely of the form $x^y$ with $x$ and $y$ in $K \ssm \QQ$.
\item[iv)] The coefficients $a_n$ of the power series 
$\ve(q) = \sum_{n \geq 1} a_n Q^n$ expanded in a local parameter~$Q$ as in~ii)
lie in the number field~$K(S)$.
\end{itemize}
\end{Thm}
\par
\begin{proof} Statement~i) obviously follows from~\eqref{eq:lowesttermeq}
and~\eqref{eq:newS}. For ii) it suffices to test for $\nu$ the dual basis  
$\{\alpha^*, \beta^*\} \subset \fraca^2$ of $\{\alpha,\beta\}$. Plugging in 
$\beta^*$ confirms that $S$ here is the same as in~\eqref{eq:newS}, since
$\beta^* = \alpha^\sigma/N(\fraca)^2\sqrt{D}$. Plugging in 
$\alpha^* = \beta^\sigma/N(\fraca)^2\sqrt{D}$ implies that
$$ A\= \e\Bigl( \frac{\beta\,C}{N(\fraca)^2\sqrt{D}} \Bigr) \= S^{\beta/\alpha}\;\in\,K^{\b/g\a}\,.$$
The last statement iii) follows since the coefficients of Fourier expansion
of the theta function lie in $K(S)$ and solving recursively for
the $a_n$ involves only these coefficients and integral powers of $S$. 
For concreteness, we perform the first step in this
procedure in the case $|c| < a$. The summands $(\tx_1,\tx_2) = (\pm \tfrac 12,
\pm \tfrac 12)$ contribute to the lowest order term (in $Q$) of the
restriction of the theta derivative to the branch $(z,\varphi(z)$. 
The next lowest order term is determined by 
$(\tx_1,\tx_2) = (\pm \tfrac 12,\pm \tfrac 32)$
and we can solve for
$$ a_1 \ = S^{g + 2bs} \frac{\rho_\omega(\tfrac 12, \tfrac 32) \bigl(\sigma(\alpha^*)F(\tfrac 12, \tfrac 32)
   + \sigma(\beta^*)G(\tfrac 12, \tfrac 32)\bigr)}
 {\rho_\omega(\tfrac 12, \tfrac 12)\bigl(\sigma(\alpha^*)F(\tfrac 12, \tfrac 12)
 + \sigma(\beta^*)G(\tfrac 12, \tfrac 12)\bigr)}$$
using the notation $\rho_\omega$ introduced in the lines before~\eqref{eq:thetafraca},
where $F$ and $G$ given in~\eqref{eq:defFG} and $g = sa+tc$.
\end{proof} 
We remark that the ``$A$'' of \S\ref{sec:DEmodular} is not quite the same as the
one above, and that the numerical value given in~\eqref{eq:A} is not of the
form $x^y$ with $x,\,y\in K$, but an algebraic multiple of this.  This is because
we normalized our~$Q$ in~\S\ref{sec:DEmodular} so that the expansion of $t=Q+\cdots$
had leading coefficient~1.  If we changed~$Q$ by the algebraic factor, then $t$ and $y$
would still have~$Q$-expansions with coefficients in~$K$, so that this would 
be an equally good choice in the special case~$D=17$, but for the general
statement it seems best to normalize~$Q$ as above.

\subsection{Fourier coefficients of twisted modular forms} 
\label{sec:Fouriercoeff}

Fix $D$, a fractional $\O_D$-ideal $\fraca$, and a branch
of the vanishing locus of $\PTD$ through the cusp at $\infty$
of the Hilbert modular surface $X_{D,\fraca}$ given by the quadratic form
$[a,b,c]$ and $S$ as above. Let $\Gamma \subset \SLA$
be the subgroup stabilizing this branch. It is the Fuchsian group
uniformizing the curve $W_D$, normalized so that the cusp labeled
with ($[a,b,c],S)$ is the cusp at~$\infty$.
\par
The first twisted modular form for $\Gamma$ we encountered was the form $\varphi'(z)$ of
bi-weight $(2,-2)$ in the case $D=\sqrt{17}$. By the preceding theorem we
know that this modular form has a Fourier expansion of the form
$$\varphi'(z) = \sum_{n \geq 0} b_n (Aq)^n $$
with $A$ transcendental of Gelfond-Schneider type and $b_n$ in the field $K(S)$ 
(= $K$ in this case, since $g=1$).
We will show that such a statement holds for all twisted modular forms.
\par
This section is inspired by work of Wolfart in the case of non-compact Fuchsian triangle groups.
Let $\Delta(\infty,q,r)$ be such a group. We normalize it so that 
$\Delta(\infty,q,r) \subset \SL{\bar{\QQ}}$ and that $\infty$
is a cusp of $\Delta(\infty,q,r)$. 
In this situation Wolfart shows in~\cite{Wo83}) that there exists some $A \in \CC$
such that the space of (ordinary, i.e.\ of 
bi-weight~$(k,0)$) modular forms for $\Delta(\infty,q,r)$ admits a basis given
by forms $f_m$ with Fourier expansions
$$ f_m(z) = \sum_{n \geq 0} r_{m,n} (Aq)^n \quad \text{with}\,\, r_{m,n} \,\, \text{rational}\,, $$
where $q = \e(z/a_0)$. The constant $A$ is transcendental of Gelfond-Schneider 
type if the Fuchsian group is non-arithmetic and is algebraic in the finitely 
many other cases.
\par
The uniformizing group of the curves $W_D$ is non-arithmetic and the
following result extends Wolfart's non-algebraicity result to this class
of curves. 
\par
\begin{Thm} \label{thm:FCtwistedmodgeneral}
The space of twisted modular form of all bi-weights for the \changed{uniformizing
group~$\Gamma$ of~$W_D$} has 
a basis of forms with Fourier expansions $ \sum_{n \geq 0} a_n Q^n $ with $a_n \in K(S)$
with $Q = A \e(z/\alpha)$ and~$S$ as in Theorem~\ref{thm:arithphi}~ii). The number~$A$,
and also the radius of convergence~$|A|$ of this series, 
is transcendental of Gelfond-Schneider type.
\end{Thm}
\par
Wolfart's proof of $A$ being of Gelfond-Schneider 
type relies on properties of $\Gamma$-functions and trigonometric 
calculations. It overlaps with our result in the few cases ($D=5,8,12$)
where the uniformizing group of $W_D$ is a triangle group. The rationality 
of $r_n$ in Wolfart's result follows easily from the rationality of the
coefficients of the Picard-Fuchs differential operators. This rationality does
not hold for general $D$, as can be seen from our example~\eqref{AB}. The
statement of the theorem can indeed not be strengthened to rationality of the
coefficients, even in the case $k=\ell = 0$, since the modular function~$t$ 
in~\eqref{twrtQ} admits no rescaling that has rational coefficients.
\par
\begin{proof}
For each $(k,\ell)$ the space of Hilbert modular forms of weight $(k,\ell)$
has a basis of forms whose Fourier expansions have rational coefficients.
Suppose that $f = \sum_{\nu \in \fraca^2} c_\nu \e(\nu z_1 + \nu^\sigma z_2)$ is such a 
basis element. Then the restriction to $(z_1,z_2)  = (z,\varphi(z))$ 
with $\varphi$ as in~\eqref{eq:varphiagain} is 
$$f(q) = \sum_{\nu\in \fraca^2} c_\nu S^{\tr(\beta\nu)} \, Q^{\tr(\alpha \nu)} e^{\sigma(\nu)\ve(q)}$$
and so for this twisted modular form the claim directly follows from Theorem~\ref{thm:arithphi}.
\par
For $k$ and $\ell$ sufficiently large the restriction map is surjective since  
$\cLL_1 \otimes \cLL_2$ is ample on $X_D$. Here  $\cLL_1$ and $\cLL_2$ are the 
natural line bundles on $X_D$ such that Hilbert modular forms of weight $(k,\ell)$ 
are sections of $\cLL_1^k \otimes \cLL_2^\ell$. (The surjectivity in fact holds
already for $k \geq 4$ and $\l \geq 10$. To see this, we tensor the 
structure sequence for $W_D \subset X_D$ with $\cLL_1^k \otimes \cLL_2^\ell$ 
and note that the cokernel of the restriction map
lies in $H^1(X_D, {\mathcal I}_{W_D} \otimes \cLL_1^k \otimes \cLL_2^\ell)$.
Since ${\mathcal I}_{W_D} \cong {\mathcal O}_{X_D}(-W_D) \cong  
\cLL_1^{-3} \cLL_2^{-9}$ and since $\cLL_1 \otimes \cLL_2$ is ample on $X_D$, 
Kodaira's vanishing theorem implies the claim.)
\par
For the remaining cases, note that with $f$ and $f$ also $1/f$ and $fg$
have Fourier expansions as claimed in the theorem. It thus suffices
to consider the products of modular forms of bi-weight $(k,\ell)$
with a given modular form of large enough weight and then to
apply the restriction argument.
\end{proof}

\subsection{The foliation by constant absolute periods} \label{sec:folperiods}

As quotients of $\HH \times \HH$, Hilbert modular surfaces come with
two natural holomorphic foliations, which we called $\cFF_1$ and $\cFF_2$.
They also admit an interesting foliation defined using the $\SL\RR$-action
on the space $\Omega \M_2$. The leaves of this foliation are upper half-planes, 
but the foliation is not holomorphic. It is studied in detail in~\cite{McM07}.
In the context of studying the $\SL\RR$-action a natural local coordinate
system on $\Omega \M_2$ is given by period coordinates, i.e.\ by integrating
the holomorphic one-form $\omega$ along a chosen basis of the first homology
relative to the zeros of $\omega$. In the special case of genus two, we
have used this coordinate system in~\eqref{eq:per_edge}.
\par
We may identify the (Torelli preimage of a) Hilbert modular surface as a subset of 
$\PP \Omega\M_2 = (\Omega\M_2)/\CC^*$ by mapping $X$ to the class of $(X,\omega)$, 
where $\omega$ is the eigenform for the first embedding of $K$ (in the order that
we have chosen once and for all). It follows from~\cite{Ba07} or~\cite{McM07}
that the foliation  by constant absolute periods, where only $x_E$ defined 
in~\eqref{eq:per_edge} is allowed to vary, is the {\em first} foliation $\cFF_1$.
The function $x_E$ is not globally well-defined: its sign depends on the choice of
an orientation and it may also be altered by a constant by adding the period of
a closed loop. However $q = (dx_E)^2$ is a well-defined quadratic differential on each leaf
of $\cFF_1$, independent of these choices.
\par
The horizontal trajectories of this quadratic differential exhibit beautiful
structures on the leaves of $\cFF_1$. They have been determined by McMullen 
in~\cite{McM12}, using the following theorem.
\par
\begin{Thm} \label{QuadDiff} The quadratic differential $q$ is proportional to the restriction
of the meromorphic modular form
$$ Q(z_1,z_2) \= \Biggl(\prod_{(m,m')\, {\rm odd} } D_2\theta \Tchi{m}{m'}(z_1,z_2)\Biggr)\bigg/ 
\Biggl(\prod_{(m,m')\, {\rm even} }\theta \Tchi{m}{m'}(z_1,z_2) \Biggr) $$
of weight $(-2,4)$ to the leaf where $z_1$ is constant. 
\end{Thm}
\par
\begin{proof}
This follows directly from~\cite{Ba07}, \changed{Theorem~10.2}, 
and Theorem~\ref{thm:TeichviaDTH}. More
precisely, Bainbridge has determined the dependence on $z_1$. Since 
the quadratic differential depends on the choice of a holomorphic one-form
on each Riemann surface, there exists a linear map $Q_1$ from the first
eigenform bundle to quadratic differentials on the leaves of $\cFF_1$ that is locally
defined by $q = (dx_E)^2$. Such a map is the same object as a  meromorphic modular form
of weight $(-2,4)$. This modular form $Q_1$ vanishes at $W_D$, where the zeros collide,
and acquires a pole at the reducible locus $P_D$, where the zeros are infinitely
far apart. It is shown moreover in~\cite{Ba07}, \changed{Theorem~10.2}, 
that both the vanishing order and the
pole order are equal to one. Consequently, by Theorem~\ref{thm:TeichviaDTH} and
the fact that the even theta characteristics vanish precisely at the reducible locus,
$Q/Q_1$ is a holomorphic function on $X_D$.  
It then extends holomorphically to the Baily-Borel compactification by Hartog's theorem, 
since the boundary has codimension two. We deduce that $Q/Q_1$ has to be constant,
which proves the claim.
\end{proof}
\par


\end{document}